\documentclass[10pt, reqno]{amsart}
\usepackage{amsthm, amsfonts, amssymb, color}
\usepackage{mathrsfs}
\usepackage{geometry}
\usepackage{enumitem}
\usepackage{tikz}
\usepackage{graphicx}
\usepackage{subfigure}
\usepackage{threeparttable}  
\usepackage{booktabs}        
\usepackage{diagbox}
\usepackage{romannum}
\usetikzlibrary{angles, quotes}
\usetikzlibrary{intersections, decorations.pathreplacing}
\usepackage[colorlinks=true, linkcolor=blue]{hyperref}

\newcommand{\R}{\mathbb R}

\newcommand{\N}{\mathbb{N}}

\newcommand{\les}{\lesssim}

\DeclareMathOperator{\Ker}{Ker}

\DeclareMathOperator{\sgn}{sgn}

\def\d{{\rm d}}
\def\i{{\rm i}}
\def\k{{\mathbf k}}

\newtheorem{theorem}{Theorem}[section]
\newtheorem{proposition}[theorem]{Proposition}
\newtheorem{assumption}[theorem]{Assumption}
\newtheorem{corollary}[theorem]{Corollary}
\newtheorem{lemma}[theorem]{Lemma}
\newtheorem{definition}[theorem]{Definition}

\newtheorem{remark}[theorem]{Remark}

\numberwithin{equation}{section}

\title[]
{Pointwise estimates for the fundamental solutions of higher order Schr\"{o}dinger equations in low odd dimensions}

\author{Han Cheng,\, Shanlin Huang,\,Tianxiao Huang,\,  Quan Zheng}

\address{Han Cheng,  Institute of Applied Physics and Computational Mathematics, Beijing, 100088, China.}
\email{chh@hust.edu.cn}

\address{Shanlin Huang, School of Mathematics (Zhuhai), Sun Yat-sen University, Zhuhai 519082, Guangdong, China}
\email{huangshlin6@mail.sysu.edu.cn}

\address{Tianxiao Huang, School of Mathematics (Zhuhai), Sun Yat-sen University, Zhuhai, 519082, Guangdong, China }
\email{htx5@mail.sysu.edu.cn}

\address{Quan Zheng, School of Mathematics and Statistics, Huazhong University of Science and Technology, Wuhan 430074, Hubei, China }
\email{qzheng@hust.edu.cn}


\begin{document}
\pagenumbering{arabic}

\begin{abstract}
In this paper, we study the fundamental solution of the higher order Schr\"odinger equation
\begin{equation*}
	\mathrm{i}\partial_t u(x,t) = \big((-\Delta)^m + V(x)\big)u(x,t), \quad t \in \mathbb{R}, \ x \in \mathbb{R}^n,
\end{equation*}
for any odd dimension $n$ and integer $m \geq 1$ satisfying $n < 4m$, where $V$ is a real-valued bounded potential with suitable decay. 

Let $P_{ac}(H)$ denote the projection onto the absolutely continuous spectral subspace of $H = (-\Delta)^m + V$, and assume $H$ has no positive embedded eigenvalues. Our main result says that the evolution operator $e^{-\mathrm{i}tH}P_{ac}(H)$ has an integral kernel $K(t,x,y)$ satisfying the pointwise estimate
\begin{equation*}
	|K(t,x,y)| \leq C (1 + |t|)^{-h} (1 + |t|^{-\frac{n}{2m}}) \left(1 + |t|^{-\frac{1}{2m}}|x - y|\right)^{-\frac{n(m-1)}{2m-1}}, \quad t \neq 0, \ x,y \in \mathbb{R}^n,
\end{equation*}
where the exponent $h$ depends on $m$, $n$, and the zero energy resonance structure of $H$. We also prove analogous estimates for smoothing operators of the form $H^{\frac{\alpha}{2m}}e^{-\mathrm{i}tH}P_{ac}(H)$.

The key innovation of this paper is a unified approach to deriving asymptotic expansions of the perturbed resolvents around zero, which comprehensively addresses all possible resonance types.
\end{abstract}

\maketitle
\begingroup
\tableofcontents

\section{Introduction}\label{section1}

\subsection{Background and Motivation}\

In this paper, we study space-time  pointwise estimate for the fundamental solution of higher order Schr\"{o}dinger equation
\begin{equation}\label{equ1.1.1}
	\i\partial_tu(x, t)=((-\Delta)^m+V(x))u(x, t), \quad t\in \mathbb{R},\,\,x\in \mathbb{R}^n,
\end{equation}
where $m$ is any positive integer, and $V$ is a real-valued decaying potential in $\mathbb{R}^n$.

In the second order case ($m=1$), it is well known that the fundamental solution of the free Schr\"{o}dinger equation is given by
\begin{equation*}\label{eq1.1.2}
	e^{\i t \Delta}(x, y):=(4\pi \i t)^{-\frac{n}{2}}e^{\frac{\i|x-y|^2}{4t}},\qquad t\ne 0,\,\,x, y\in \mathbb{R}^n,
\end{equation*}
from which, we obtain immediately the $L^1-L^{\infty}$ dispersive estimate
\begin{equation*}
	\|e^{\i t \Delta}\|_{L^1-L^{\infty}}\lesssim |t|^{-\frac{n}{2}}.
\end{equation*}
The dispersive estimate for  $e^{-\i t H}$ (where $H=-\Delta+V(x)$ with $V\ne 0$) has attracted significant attention  in the past three decades, particularly due to their important applications in  nonlinear Schr\"odinger equations. A foundational contribution was made by Journ\'e-Soffer-Sogge \cite{JSS}, who obtained the $L^1 - L^\infty$ estimate in the regular case for dimensions $n \geq 3$. Subsequent advances for $n \leq 3$ were achieved by Weder \cite{We}, Rodnianski-Schlag \cite{RS}, Goldberg-Schlag \cite{Sch-G} and Schlag \cite{sch}. Notably, Yajima \cite{Ya, Ya-99} developed the wave operator method to derive $L^1 - L^\infty$ estimates for $-\Delta+V$. Further developments in this field can be found in \cite{BG,DF06,ES04,ES06,GW15,GW17,We99} and references therein. For comprehensive overviews of related developments, we refer to the survey papers \cite{Sch07,Sch21}.

Recent years have witnessed  considerable progress in dispersive estimates for $m>1$ in \eqref{equ1.1.1}.
Local decay estimates were studied by Feng-Soffer-Yao \cite{FSY} and  Feng-Soffer-Wu-Yao  \cite{FSWY}.
For the fourth order ($m =2$) Schr\"{o}dinger operators,  Erdo\v{g}an-Green-Toprak \cite{EGT} obtained (global) dispersive estimates when $n=3$, and the case of $n=4$ was established by  Green-Toprak \cite{GT19}. Later, Soffer-Wu-Yao \cite{SWY} and Li-Soffer-Yao \cite{lsy} proved dispersive bounds in 1-D and 2-D respectively.
Very recently, Erdo\v{g}an-Goldberg-Green \cite{EGG23-1} proved dispersive estimates for  scaling-critical potentials for $2m<n<4m$ and $m>1$.

On the other hand,   in higher dimensions $n>2m\ge4$, dispersive estimates follow from the work of  Erdo\v{g}an-Green \cite{EG22,EG23} on the $L^p$ boundedness of wave operators for $1\le p\le \infty$ assuming regular zero energy and  sufficient smoothness of the potential. 
However, in lower dimensions $n<2m$, the wave operators are generally not expected to be bounded on  $L^p$ for the endpoints $p=1, \infty$. 
For example, Goldberg-Green \cite{GG21}  established  $L^p$ boundedness for $1<p<\infty$ in the case $(m, n)=(2, 3)$ under the assumption that zero energy is regular, while  Mizutani-Wan-Yao \cite{MWY24} provided counterexamples showing the unboundedness at $p=1, \infty$. The case $n=2m=4$ was addressed by Galtbayar-Yajima \cite{GY}, where $L^p$ boundedness was proved for $1<p<p_0$ with certain $p_0$ depending on the type of singularity at zero energy. More recently,  Mizutani-Wan-Yao \cite{MWY} studied the case $(m, n)=(2, 1)$, showing that the wave operators are bounded for $ 1<p<\infty$, but not for  $p=1, \infty$. In the latter case, they established weak boundedness within the frameworks such as  $L^{1, \infty}$, $H^1$ and $\mbox{BMO}$,  depending on the type of threshold obstructions.

When examining fundamental solution, the higher order case  ($m>1$) differs due to its additional spatial decay, which is absent in the second order case.
For the polyharmonic operator  $(-\Delta)^m$, $m>1$,  the fundamental solution of $e^{-\i t (-\Delta)^m}$ exhibits the following  space-time pointwise decay (see e.g. \cite{BKS,Mi}):
\begin{equation}\label{eq1.1.3}
	|e^{-\i t (-\Delta)^m}(x, y)|\le C|t|^{-\frac{n}{2 m}}\left(1+|t|^{-\frac{1}{2 m}}|x-y|\right)^{-\frac{n(m-1)}{2 m-1}},\qquad t\ne 0,\,\,x, y\in \mathbb{R}^n.
\end{equation}
\eqref{eq1.1.3} clearly implies the dispersive estimate,  but the converse is not true. We refer to \cite{HHZ,KPV,Mi} for fundamental solution estimates of $e^{\i t P(D)}$  with general elliptic $P(D)$ of constant coefficients.

The pointwise estimate \eqref{eq1.1.3} for polyharmonic operators, combined with existing dispersive estimates for higher order Schr\"{o}dinger operators, naturally leads us to investigate the following problem:

\begin{quote}
	Does the space-time pointwise estimate \eqref{eq1.1.3} hold for general higher order Schr\"{o}dinger operators $H=(-\Delta)^m+V$?
\end{quote}
To our best knowledge, such pointwise estimates have not yet been established.

We aim at  establishing estimates for the
fundamental solution of \eqref{equ1.1.1} in all odd dimensions $n\ge 1$ and for all integers $m\geq1$.
However,  the main challenge in high dimensions ($n>4m$)  lies in the high energy part, while for dimensions $n<4m$, the key difficulties arise in the low  energy regime. Given the fundamentally distinct nature of these technical  challenges, the current paper only focuses on the low dimensional case, and the results for high dimensions are available  in the sequel to this work, see \cite{CHHZ}.

\subsection{Main results}\

Throughout the paper, we assume, unless stated otherwise,  $n$ is odd  and $1\le n<4m$.

To state our main result, we first give the definition of zero resonances.
Recall that zero is an eigenvalue of $H=(-\Delta)^{m}+V$ if there exists some nonzero  $\phi\in L^2(\R^n)$, such that
\begin{equation}\label{equ1.2}
	((-\Delta)^{m}+V)\phi=0,
\end{equation}
in distributional sense. In general, there may exist nontrivial resonant solutions of \eqref{equ1.2} in the weighted $L^2$ space $L_s^2(\R^n):=\{f; \,(1+\lvert \cdot \rvert)^sf\in L^2(\R^n)\}$ if $s<0$, but not in $L^2(\mathbb{R}^n)$. 
For example, consider $\phi = 1 + e^{-|x|^2}$ and define $V = -\frac{(-\Delta)^{m}\phi}{\phi}$, then \eqref{equ1.2}  holds. Such potential is smooth with exponential decay at $\infty$, however, $\phi\in L_{-\frac{n+1}{2}}^2(\R^n)\setminus L^2(\R^n)$. 
It tuns out that the existence of such   resonant solutions can potentially (though does not necessarily) influence the time decay of the propagator $e^{-\i tH}$. To address this systematically, we give a precise classification of different types of resonaces. Throughout the paper, we use the notation (see \cite{SWY,MWY})
\begin{equation}\label{equ1.3}
	W_s(\R^n)=\bigcap\limits_{\sigma<s}L_\sigma^2(\R^n).
\end{equation}
We also denote
\begin{equation}\label{equ0.1}
	m_n=\left\{
	\begin{array}{ll}
		m,\quad \,\qquad\,\quad \mbox{if}\,\,  1\le n\le 2m-1,\\[0.3cm]
		\mbox{$2m-\frac{n-1}{2}$},\,\,\,\,\,\mbox{if}\,\, 2m+1\le n\le 4m-1,
	\end{array}
	\right.
\end{equation}

We now present the following definition for classifying the spectrum of $H$ at zero energy.
\begin{definition}\label{def1.1}\
Let $V\in L^{\infty}(\mathbb{R}^n)$ be real-valued.

\noindent\emph{(\romannumeral1) } We say zero is a resonance of the $\k$\text{-}th kind $(1\le \mathbf{k}\le m_n)$, if \eqref{equ1.2} has a non-trivial solution in $W_{-\frac{1}{2}-m_n+\mathbf{k}}(\R^n)$, but has no non-trivial solution in $W_{\frac{1}{2}-m_n+\mathbf{k}}(\R^n)$.

\noindent\emph{(\romannumeral2) } We say zero is an eigenvalue of $H$, if \eqref{equ1.2} has a non-trivial solution in $L^2(\mathbb{R}^n)$. In this case, zero is also said to be a resonance of the $(m_n+1)$-th kind.
		
\noindent\emph{(\romannumeral3) } We say zero is a regular point ,  if $(i)$ and $(ii)$ are not satisfied. In this case, zero is also said to be a resonance of the $0$\text{-}th kind.
\end{definition}
	
The following remarks concerning the Definition \ref{def1.1} are in order. 
	
\begin{itemize}
\item [($\textbf{a}_1$)] In dimensions $n > 4m$, no resonance occurs at the zero energy in the sense that \eqref{equ1.2} has no solution in $L_s^2(\mathbb{R}^n)\setminus L^2(\mathbb{R}^n)$ for any $s<0$ if $V$ has sufficient decay (see \cite[Remark 2.13]{FSWY}). This parallels the behavior observed for the classical Schr\"odinger operators in dimensions $n>4$, where zero resonances are also absent.
		
		
\item  [($\textbf{a}_2$)]  For any real-valued $V\in L^{\infty}(\mathbb{R}^n)$, Definition \ref{def1.1} implies that zero must be a $\mathbf{k}$\text{-}th kind resonance  for some $0\le \mathbf{k}\le m_n+1$.
Note that in  \cite{EGG23-1},   a resonance is defined when $(1 + |x|)^{-\sigma}\phi \in L^2$ (with $\phi$ satisfying \eqref{equ1.2}) for some $\sigma > 2m - \frac{1}{2}$ when $n>2m$.  Our definition not only agrees with this characterization but also provides  a more detailed  classification of zero resonances. Furthermore, it is consistent with the frameworks in   \cite{SWY,MWY}, where dispersive estimates and $L^p$ boundedness of wave opeators are established under different types of threshold obstructions in the special case  $(m, n)=(2,1)$.

\item  [($\textbf{a}_3$)] For a given order of the Laplacian and  the dimension, the index $m_n$ represents the number of non-trivial zero resonance types. Its value differs between the regimes $n < 2m$ and $n > 2m$. The reason can be explained as follows: In dimensions $n > 2m$, if \eqref{equ1.2} holds in the distributional sense, we immediately obtain
	 $\phi=(-\Delta)^{-m}(V\phi)$.
It is well known that the following estimate (see e.g. \cite[Lemma 2.3]{J80})
\begin{equation}\label{eq-7-27-1}
    \|(-\Delta)^{-m}\|_{L^2_{s}-L^2_{-s'}}<\infty
\end{equation}
holds provided $s,s'>2m-\frac n2$ and $s+s'>2m$. Thus $\phi\in W_{\frac{n}{2}-2m}$ holds provided $V$ decays sufficiently fast at the infinity, which coincides with our choice of $W_{\frac{1}{2} - m_n }(\mathbb{R}^n)$ in Definition \ref{def1.1}.
While in the case $n < 2m$, the operator $(-\Delta)^{-m}$ cannot be properly defined through Riesz potential, the analysis here of its mapping properties requires the incorporation of additional vanishing conditions, leading to corresponding modifications for the associated indices. More specifically, Proposition \ref{pro:one to one corres of solutions} (see \eqref{eq-727-2}) establishes that $\phi\in W_{\frac{1}{2}-m}(\R^n)$, again consistent with our choice of $W_{\frac{1}{2} - m_n }(\mathbb{R}^n)$ in Definition \ref{def1.1}.

\item [($\textbf{a}_4$)] As demonstrated  in \cite{EGT} (for the case $(m, n)=(2,3)$) and \cite{SWY} (for the case $(m, n)=(2,1)$), the presence of resonances may (but does not always) affect the time-decay properties of the evolution operator $e^{-\i tH}$ if $n<2m$. The above classification scheme allows us to see explicitly how different types of zero resonances affect the time-decay, see Theorems \ref{thm1.1} and \ref{thm1.1-723}.
\end{itemize}

In order to explore how resonance type affects the study of the resolvents of $(-\Delta)^m+V$ and of the fundamental solution estimate, we need the following technical assumption on the potential.

	
\begin{assumption}\label{assum1} Let $V\in L^{\infty}(\mathbb{R}^n
)$ be  real-valued and zero be a  $\mathbf{k}$\text{-}th kind
resonance of  $H=(-\Delta)^m+V$ for some $0\le \mathbf{k}\le m_n+1$. We assume that
		
\noindent\emph{(\romannumeral1) }$H$ has no positive embedded eigenvalue.
		
\noindent\emph{(\romannumeral2) }$V$ satisfies 
\begin{equation}\label{eq:decay assump}
\lvert V(x)\rvert\lesssim \langle x \rangle^{-\beta-}, \quad  \beta=\max\left\{4m-n,\,n\right\}+4\mathbf{k}+4.
\end{equation}
\end{assumption}
		
	Some remarks concerning the above assumptions are as follows.
\begin{itemize}
\item  [($\textbf{b}_1$)] In the second order case, a classical result due to Kato (see \cite[Theorem  \uppercase\expandafter{\romannumeral13}.58]{RS79}) says that $-\Delta+V$ has no positive eigenvalue embedded into the absolutely continuous spectrum provided $V$ is bounded and decays faster than $|x|^{-1}$ at the infinity. 
 However, the higher order case exhibits subtle differences. For instance, when $m\ge 2$ is even, $H=(-\Delta)^m+V$ may have positive eigenvalues even if $V\in C^{\infty}_{0}$, see e.g. \cite{FSWY}. 

\item  [($\textbf{b}_2$)]  The decay assumption \eqref{eq:decay assump} is explicitly tied to the resonance kind, which is needed during the low energy analysis, particularly in the proof of  Theorem \ref{thm3.4}. While in the high energy part, we require weaker decay assumptions, as specified in  \eqref{eq-V-815-1}.
\end{itemize}

We present our main results by considering two distinct dimensional regimes: $1\le n<2m$ and $2m<n<4m$.
Let $P_{ac}(H)$ denote the projection onto the absolutely continuous spectrum space of $H=(-\Delta)^m+V$ and let $K(t,x,y)$ be the Schwartz kernel of $e^{-\i tH}P_{ac}(H)$.
\begin{theorem}\label{thm1.1}
Let $n$ be an odd integer in the range $1 \leq n < 2m$. Under Assumption~\ref{assum1}, $K(t,x,y)$ is an integral kernel when $t\neq0$, satisfying the following.

\noindent (i) If $0\le \k\le m-\frac{n-1}{2}$, then 
\begin{equation}\label{equ1.2.1-231}
|K(t, x,y)|\lesssim  |t|^{-\frac{n}{2 m}}\left(1+|t|^{-\frac{1}{2 m}}|x-y|\right)^{-\frac{n(m-1)}{2 m-1}}.
\end{equation}
\noindent (ii) If $m-\frac{n-1}{2}<\k\le m(=m_n)$, then
\begin{equation}\label{equ1.2.1-232}
|K(t, x,y)|\lesssim  (1+|t|)^{-\frac{2m+1-2\k}{2m}}(1+|t|^{-\frac{n}{2 m}})\left(1+|t|^{-\frac{1}{2 m}}|x-y|\right)^{-\frac{n(m-1)}{2 m-1}}.
\end{equation}
\noindent (iii) If zero is an eigenvalue, then \eqref{equ1.2.1-232} holds with $\k=m$.
\end{theorem}

\begin{theorem}\label{thm1.1-723}
Let $n$ be an odd integer in the range $2m<n < 4m$. Under Assumption~\ref{assum1}, $K(t,x,y)$ is an integral kernel when $t\neq0$, satisfying the following.

\noindent (i) If $\k=0$ (i.e. zero energy is regular), then 
\begin{equation}\label{equ1.2.1-233}
|K(t, x,y)|\lesssim  |t|^{-\frac{n}{2 m}}\left(1+|t|^{-\frac{1}{2 m}}|x-y|\right)^{-\frac{n(m-1)}{2 m-1}}.
\end{equation}
\noindent (ii) If $1\le \k\le 2m-\frac{n-1}{2}(=m_n)$, then
\begin{equation}\label{equ1.2.1-234}
|K(t, x,y)|\lesssim  (1+|t|)^{-\frac{4m-n+2-2\k}{2m}}(1+|t|^{-\frac{n}{2 m}})\left(1+|t|^{-\frac{1}{2 m}}|x-y|\right)^{-\frac{n(m-1)}{2 m-1}},
\end{equation}
\noindent (iii) If zero is an eigenvalue, then \eqref{equ1.2.1-234} holds with $\k=2m-\frac{n-1}{2}$.
\end{theorem}
We make the following remarks related to Theorems \ref{thm1.1} and \ref{thm1.1-723}.
\begin{itemize}
\item  [($\textbf{c}_1$)]  Theorems \ref{thm1.1} and \ref{thm1.1-723} imply immediately the following dispersive estimates
\begin{equation}\label{equ-disp-h}
\|e^{-\i tH}P_{ac}(H)\|_{L^1-L^{\infty}}\lesssim (1+|t|)^{-h(m, n, \mathbf{k})}(1+|t|^{-\frac{n}{2 m}}),\qquad\,\,t\ne 0,
\end{equation}
where 
\begin{equation}\label{equ4.1.3}
h(m, n, \mathbf{k}):=\begin{cases}
\frac{n}{2m},\quad&\mbox{if}\,\,0\le \mathbf{k}\le \mathbf{k}_c,\\
\frac{2m_n+1-2\mathbf{k}}{2m},&\mbox{if}\,\,  \mathbf{k}_c<\mathbf{k}\le m_n,\\
\frac{1}{2m},&\mbox{if}\,\,\mathbf{k}=m_n+1.
\end{cases}
\end{equation}
with $\k_c$ given by
\begin{equation}\label{equ3.3.1}
\mathbf{k}_c=\max\left\{\mbox{$m-\frac{n-1}{2},\,\,  0$}\right\}=
\begin{cases}
m-\frac{n-1}{2},\quad&\mbox{if}\,\,  1\le n\le 2m-1,\\[0.3cm]
0,&\mbox{if}\,\,  \,2m+1\le n\le 4m-1.
\end{cases}
\end{equation}
To our best knowledge, when $1\leq n<2m$ and $m>2$,
the dispersive estimate for higher order Schr\"{o}dinger equations has not been investigated before. Moreover, since the above pointwise bounds have additional  decay in $|x-y|$ when $m>1$, we can further establish $L^p-L^q$ type estimates for $e^{-\i tH}P_{ac}(H)$ where $p, q$ are not dual exponents in the higher order case, see Corollary \ref{cor-Lp}.

\item  [($\textbf{c}_2$)]   We shall outline of the idea of the proof in subsection \ref{section1.3}, and the main  challenge lies in deriving the asymptotic expansion of the perturbed resolvent near zero energy.
We mention that when $n>4m$, the  perturbed resolvent expansion   has been proved in \cite{FSWY}. When $n<4m$, the situation is much more complicated. For instance, when $m=2$ and $n=4$, Green-Toprak \cite{GT19} first obtained the expansion of the perturbed resolvent by an iteration scheme originating from Jensen-Nenciu  \cite{JN}.  Employing such techniques, the results for $n=3,2,1$ were also established in \cite{EGT,lsy,SWY} respectively. It seems that such an iteration scheme will become increasingly  complex when $m$ becomes larger. To avoid such an iteration scheme, we prove the expansion by introducing suitable orthogonal subspaces, and simplify the problem to the study of the inverse of a finite dimensional operator matrix, and in particular, the operator matrix has some specific global structure that allows us to handle in a way unified for  all odd $n<4m$.
\end{itemize}

Theorems \ref{thm1.1} and \ref{thm1.1-723} can be applied to obtain new  $L^p-L^q$ type estimates. More precisely, set
\begin{align}\label{eq3.7}
\square_{\mathrm{ABCD}}=\{\mbox{$(p, q)$}\in\mathbb{R}_+^2;\ \mbox{$({1\over p},{1\over q})$}\ {\rm lies\,in\,the\,closed\,quadrilateral\,ABCD}\}.
\end{align}
Here $A=(\frac12,\frac12)$, $B=(1,\frac{1}{\tau_m})$, $C=(1,0)$ and $D=(\frac{1}{\tau'_m},0)$, where $\tau_m=\frac{2m-1}{m-1}$, see Figure \ref{fig1}.
Moreover, we denote by $H^1$ the Hardy space on $\R^n$ and by BMO the space of functions with bounded mean oscillation on $\R^n$. Let $(p, q)\in \square_{\mathrm{ABCD}}$ and denote by
\begin{align}\label{equ3.6}
L^p_*-L^q_*=
\begin{cases}
L^1-L^{\tau_m,\infty}\, \text{or}\, H^1-L^{\tau_m},&\text{if},\,(p,q)=(1, \tau_m),\\[4pt]
L^{\tau'_m,1}-L^\infty\,  \text{or}\, L^{\tau_m'}-\text{BMO},&\text{if}\,(p,q)=(\tau'_m,\infty),\\[4pt]
L^p-L^q,&\text{otherwise}.
\end{cases}
\end{align}
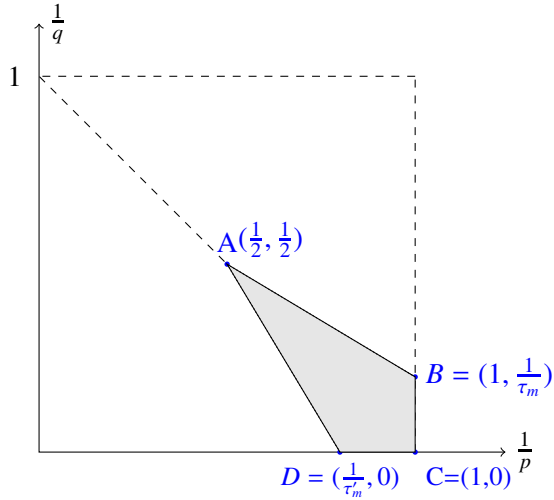
\begin{figure}[h]
    \centering
  \begin{tikzpicture}[scale=5.0]
    \draw [<->,] (0,1.14) node (yaxis) [right] {$\frac{1}{q}$}
        |- (1.24,0) node (xaxis) [right] {$\frac{1}{p}$};


\draw (0.80,0) coordinate (a_1)node[below]{\color{blue}{\small{$D=(\frac{1}{\tau_m'},0)$}}};
\draw[dashed](1,0) coordinate (a_3)node[below=3.5mm] [right=0.1mm] {\small{\color{blue}{C=(1,0)}}} -- (1,1) coordinate (a_4);
\draw[dashed](1,0) coordinate (a_3) -- (0,1) coordinate (a_5)node[left=0.9mm] {$1$};
\draw[dashed](1,1) coordinate (a_3) -- (0,1) coordinate (a_5);
\draw (1,0.2) coordinate (a_5)node[right]{\color{blue}{$B=(1,\frac{1}{\tau_m})$}} -- (0.56,0.444) coordinate (a_2);

\fill[blue]  (0.5,0.5)coordinate (a_6) circle (0.2pt) node [above=3.3mm] [right=0.1mm] {$(\frac12, \frac12)$} node[above=0.3mm]{{A}} ;
\fill[blue]  (0.80,0) circle (0.2pt);
  \fill[blue] (1,0) circle (0.2pt);
   \fill[blue]  (0.56,0.444) circle (0.2pt);
   \fill[blue]  (1,0.2) circle (0.2pt);

\draw[fill=gray!20]  (0.5,0.5)--(0.80,0) -- (1,0) -- (1,0.2) -- cycle;
\draw [densely dotted] (0.5,0.5) -- (1,0.2);
\draw [densely dotted] (0.5,0.5) -- (0.8,0);
\end{tikzpicture}
    \caption{The  $L^p-L^q$ estimates}\label{fig1}
    \end{figure}

\begin{corollary}\label{cor-Lp}
Under Assumption \ref{assum1} with the resonance kind $\k$ satisfying $0\le \k\le m-\frac{n-1}{2}$ when $1\le n<2m$, or $\k=0$ when $2m<n<4m$, we have the following.						
						
\noindent\emph{(\romannumeral1)} For all $(p, q)\in \square_{\mathrm{ABCD}}$ (see Figure \ref{fig1}), we have the  $L^p_*-L^q_*$ (defined by \eqref{equ3.6}) estimates 
\begin{equation}\label{equ-Lp-Lq}
\|e^{-\i tH}P_{ac}(H)\|_{L^p_*-L^q_*}\lesssim|t|^{-\frac{n}{2m}(\frac{1}{p}-\frac{1}{q})}.
\end{equation}
\noindent\emph{(\romannumeral2)} If $\alpha\in(0,n(m-1)]$ and redefining $\beta=\max\left\{4m,\,2n\right\}+4\mathbf{k}+4$ in \eqref{eq:decay assump}, then
\begin{equation}\label{equ-L1-smooth}
\|H^\frac{\alpha}{2m} e^{-\i tH}P_{ac}(H)\|_{L^1-L^{\infty}} \lesssim |t|^{-\frac{n+\alpha}{2m}}.
\end{equation}
\end{corollary}
If we take  $\alpha=n(m-1)$, the above result implies the smoothing dispersive bound
\begin{equation*}
\|H^\frac{n(m-1)}{2m}e^{-\i tH}P_{ac}(H)\|_{L^1-L^\infty}\lesssim|t|^{-\frac n2},
\end{equation*}
which can be seen as a variable coefficient variant of the smoothing effect obtained in Kenig-Vega-Ponce \cite{KPV} and  Erdo\v{g}an-Goldberg-Green \cite[Proposition 2.1]{EGG23} $($see also Huang-Huang-Zheng \cite[Theorem 3.1]{HHZ}$)$  for the free case $V\equiv0$.
It also generalizes the smoothing estimate in Erdo\v{g}an-Green \cite[Corollary 1.4]{EG23} where $n>2m$ and $\mathbf{k}=0$ (i.e. assuming zero to be regular), which was obtained by proving $L^p$ boundedness of the wave operators.

\subsection{Outline of the proof}\label{section1.3}\

We  outline the strategy for proving Theorem \ref{thm1.1} and Theorem \ref{thm1.1-723}.
Our starting point is the Stone's formula
\begin{equation}\label{equ0.2}
	\langle e^{-\i tH}P_{ac}(H)f,\, g\rangle=\frac{1}{2\pi \i}\int_0^{\infty}e^{-\i t\lambda}\langle(R^+(\lambda)-R^-(\lambda))f,g\rangle \d\lambda,\quad f, g\in C_0^{\infty}(\mathbb{R}^n),
\end{equation}
where 
\begin{equation}\label{eq:def of perturbed R^pm}
 R^\pm(\lambda):=(H-\lambda\mp \i0)^{-1}.   
\end{equation}
We divide the proof into low and high energy parts.
Let $\chi\in C_0^{\infty}(\mathbb{R})$ be a  cut-off function such that
\begin{equation}\label{equ4.cutoff}
			\chi(\lambda)=
			\begin{cases}
				1,\quad&|\lambda|\le (\frac{\lambda_0}{2})^\frac{1}{2m},\\
				0,&|\lambda|\ge \lambda_0^\frac{1}{2m},
			\end{cases}
		\end{equation}
for some $\lambda_0>0$,  and denote $\tilde{\chi}(\lambda)=1-\chi(\lambda)$. Using the spectral theorem and Stone's formula, (\romannumeral1) of Assumption \ref{assum1} allows us to split
\begin{align*}
 e^{-\i tH}P_{ac}(H)&=\frac{1}{2\pi i} \int_{0}^{+\infty} e^{-\i  t \lambda}\left(R^{+}(\lambda)-R^{-}(\lambda)\right) (\chi(\lambda)+\tilde{\chi}(\lambda))\, \d \lambda\\
  & :=e^{-\i tH}\chi(H)P_{ac}(H)+e^{-\i tH}\tilde{\chi}(H)P_{ac}(H)
\end{align*}
into low and high energy part. Therefore our main results Theorem \ref{thm1.1} and Theorem \ref{thm1.1-723} follow immediately from the two theorems below.

\begin{theorem}\label{theorem4.0}
	Under Assumption \ref{assum1}, if $\lambda_{0}\in (0, 1)$ is given by Theorem \ref{thm3.4}, then  $e^{-\i tH}\chi(H)P_{ac}(H)$ has integral kernel  $K_{L}(t,x,y)$ satisfying
\begin{equation}\label{equ4.1.2.0}
		\left|K_{L}(t, x, y)\right| \lesssim |t|^{-h(m, n, \mathbf{k})}(1+|t|^{-\frac{n}{2m}})\left(1+|t|^{-\frac{1}{2 m}}|x-y|\right)^{-\frac{n(m-1)}{2 m-1}},\quad t\neq0,\,\,x,y\in\mathbb{R}^n,
	\end{equation}
where $h(m, n, \mathbf{k})$ is given by \eqref{equ4.1.3}.
\end{theorem}

\begin{theorem}\label{theorem4.1}
	Under Assumption \ref{assum1}, $e^{-\i tH}\tilde{\chi}(H)P_{ac}(H)$ (for any fixed $\lambda_0>0$)  has integral kernel $K_{H}(t,x,y)$ satisfying
\begin{equation*}\label{equ4.1.2.11}
		\left|K_{H}(t, x, y)\right| \lesssim |t|^{-\frac{n}{2 m}}\left(1+|t|^{-\frac{1}{2 m}}|x-y|\right)^{-\frac{n(m-1)}{2 m-1}},\quad t\neq0,\,\,x,y\in\mathbb{R}^n.
	\end{equation*}
\end{theorem}

The proofs of Theorems~\ref{theorem4.0} and~\ref{theorem4.1} rely on a careful analysis of the perturbed spectral measure ($R^{+}(\lambda^{2m}) -R^{-}(\lambda^{2m}) $) combined with oscillatory integral estimates. The main technical challenge lies in  the low energy part $e^{-\mathrm{i}tH}\chi(H)P_{ac}(H)$, while the high energy analysis is comparatively routine. Hence, we restrict our exposition to an outline of Theorem \ref{theorem4.0} for $e^{\mathrm{i}tH}\chi(H)P_{ac}(H)$, proceeding in two steps.

\noindent{\bf Step 1. The asymptotic behaviors of $(M^\pm(\lambda))^{-1}$.}\

Our low energy spectral analysis of $H$ starts from the symmetric resolvent identity:
\begin{equation}\label{eq:sym resolven exp}
	R^{\pm}(\lambda^{2m}) = R_0^\pm(\lambda^{2m}) - R_0^\pm(\lambda^{2m})v(M^\pm(\lambda))^{-1}vR_0^\pm(\lambda^{2m}),\quad\lambda>0,
\end{equation}
where $R_0^\pm(\lambda^{2m})=((-\Delta)^m-\lambda^{2m}\mp\mathrm{i}0)^{-1}$ and
\begin{equation}\label{eq:def for M}
	M^{\pm}(\lambda) = U + vR_0^\pm(\lambda^{2m})v, \quad v(x) = |V(x)|^{1/2}, \quad U(x)= \operatorname{sgn} V(x),
\end{equation}
with $\operatorname{sgn} x = 1$ for $x \geq 0$ and $\operatorname{sgn} x = -1$ for $x < 0$. 
Since the properties of $R_0^\pm(\lambda^{2m})$ are well understood, our primary technical challenge reduces to deriving the asymptotic expansions of $(M^\pm(\lambda))^{-1}$ around zero energy. Theorem \ref{thm3.4} provides such expansions in the form:
\begin{equation}\label{equ3.48-1111}
	(M^\pm(\lambda))^{-1} = \sum_{i,j\in J_{\mathbf{k}}} \lambda^{2m-n-i-j} Q_i \big(M^\pm_{i,j} + \Gamma_{i,j}^\pm(\lambda)\big) Q_j, \quad 0<\lambda<\lambda_0,
\end{equation}
where $\{Q_j\}$ are certain projection operators (defined in  \eqref{equ3.52-2}) and $\Gamma_{i,j}^\pm(\lambda)$ are the higher order remainder terms.

\noindent\textbf{An illustrative case: $n<2m$ with zero eigenvalue.}

To demonstrate our approach, we give a overview of the case where $n < 2m$ and zero is an eigenvalue for an example. In this regime, $M^{\pm}(\lambda)$ admit the expansions
\begin{equation}\label{eq:example for m+1}
 \begin{aligned}
M^{\pm}(\lambda)=&\sum\limits_{j=0}^{ m-\frac{n+1}{2}}\lambda^{n-2m+2j}a_j^\pm vG_{2j}v+T_0 \\
&+\sum_{j=m-\frac{n-1}{2}}^{2m- \frac{n+1}{2}}\lambda^{n-2m+2j}a_j^\pm vG_{2j}v+ \lambda^{2m}b_1vG_{4m-n}v+vr_{4m-n+1}^\pm(\lambda)v,	
\end{aligned}   
\end{equation}
where all terms are $L^2$-bounded, as established in \eqref{equ3.52-2}.
Furthermore, the remainder terms $vr_{4m-n+1}^\pm(\lambda)v$ admit favorable $L^2$ bounds with respect to the parameter (see \eqref{equ3.53}).

When zero is the eigenvalue of $H$ (i.e., $\k = m_n+1=m+1$), we define a family of orthogonal projection operators $\{Q_j\}_{j \in J_{m+1}}$ in \eqref{eq-Q_j-odd}, where $J_{m+1}$ denotes the index set
\begin{equation*}
	J_{m+1} = \left\{\mbox{$0, 1, \ldots, m-\frac{n+1}{2}, m-\frac{n}{2}, m-\frac{n-1}{2}, \ldots, 2m-\frac{n+1}{2}, 2m-\frac{n}{2}$}\right\}.
\end{equation*}
These projections satisfy $\sum\limits_{j\in J_{m+1}}Q_j=I$ (see  Proposition \ref{prop3.1}). For  $\lambda>0$,  we introduce the isomorphism $B_{\lambda} := (\lambda^{-j}Q_j)_{j\in J_{m+1}} :$
\begin{equation*}
 \bigoplus_{j\in J_{m+1}} Q_jL^2 \ni (f_j)_{j\in J_{m+1}}\longmapsto  \sum_{j\in J_{m+1}} \lambda^{-j}Q_jf_j \in L^2,
\end{equation*}
while in the matrix form, we can write it and its adjoint $B_{\lambda}^*$ as 
$$
B_{\lambda}= \left(Q_0, \lambda^{-1}Q_1, \ldots, \lambda^{-2m+\frac{n}{2}}Q_{2m-\frac{n}{2}} \right), \qquad 
B_{\lambda}^*= \begin{pmatrix} Q_0 \\ \lambda^{-1}Q_1\\ \vdots \\ \lambda^{-2m+\frac{n}{2}}Q_{2m-\frac{n}{2}} \end{pmatrix}.
 $$
The isomorphism of $B_{\lambda}$ yields the following equivalence: 
 \begin{equation*}
     \mbox{ $M^{\pm}(\lambda)$ are invertible on $L^2$ }\iff \mbox{$B_{\lambda}^*M^{\pm}(\lambda)B_{\lambda}$ are invertible on $\bigoplus_{j\in J_{m+1}} Q_jL^2$},
 \end{equation*}
with the inversion formula
\begin{equation}\label{eq-inverse-A-to-M-0}
	\left(M^{\pm}(\lambda)\right)^{-1}=B_{\lambda} \left(B_{\lambda}^*M^{\pm}(\lambda)B_{\lambda}\right)^{-1}B_{\lambda}^*.
\end{equation}
The above equivalence enables us to transform  the original problem into studying the invertibility  of the operator block matrices $B_{\lambda}^*M^{\pm}(\lambda)B_{\lambda}$  and their asymptotic behaviors near $\lambda=0$.  More importantly, this framework capitalizes on the intrinsic  global structure of the block matrices, which  allows us to systematically analyze $(M^\pm(\lambda))^{-1}$ in a unified manner.  
More precisely, we decompose the matrices as
\begin{equation}\label{eq:decomposition-intro}
	\mathbb{A^\pm} = \lambda^{2m-n}B_{\lambda}^*M^{\pm}(\lambda)B_{\lambda} = D^{\pm} + (r_{i,j}^\pm(\lambda))_{i,j\in J_{m+1}},
\end{equation}
where $D^{\pm}$ represent the leading order terms that are independent of $\lambda$, and the remainder terms satisfy the improved $L^2$ estimates $\|r_{i,j}^\pm(\lambda)\|_{L^2-L^2}=O(\lambda^{1/2})$ for $\lambda \in (0,1)$ (see \eqref{equ3.65.55}).  The cancellation properties of the projections $Q_j$ induce a block structure in $D^{\pm}$ (see \eqref{eq:stru of D-0}):
\begin{equation}\label{eq:D-structure}
	D^\pm = \begin{pmatrix}
		D_{00}^{\pm} & 0 & D_{01}^{\pm} \\[0.3em]
		0 & Q_{m-\frac{n}{2}}T_0Q_{m-\frac{n}{2}} & 0 \\[0.3em]
		D_{10}^{\pm} & 0 & D_{11}^{\pm}
	\end{pmatrix}.
\end{equation}
The proof of the invertibility of $D^\pm$ on the direct sum space $\bigoplus_{j\in J_{m+1}} Q_jL^2$ is mainly based on \textbf{three key observations}:

\begin{itemize}

	\item[($O_1$)] Through the resolvent relation
$$R_0(-\lambda^{2m})(x-y)=R_0^{\pm}((e^{\frac{\pm i\pi}{2m}}\lambda)^{2m})(x-y),$$
where $R_0(-\lambda^{2m})=((-\Delta)^{m}+\lambda^{2m})^{-1}$, one has following identity
$$(-\i)^{i+j}(-1)^j e^{\frac{\pm \i\pi}{2m}(n-2m+i+j)}d_{i,j}^\pm :=d_{i,j},$$
where $d_{i,j}^\pm$ are elements of $D^\pm$, and the right hand side is sign independent. Thus, the invertibility of matrices $D^\pm$ are equivalent to that of another operator matrix $D=(d_{i,j})$ given in \eqref{eq-matrix-2} (i.e. $D^\pm=P^\pm DQ^\pm$ for some invertible operator matrices $P^\pm$ and $Q^\pm$):
\begin{equation*}
	D=\begin{pmatrix}
		\textbf{D}_{00} & 0 & D_{01}\\[0.2cm]
		0 & Q_{m-\frac{n}{2}}T_0Q_{m-\frac{n}{2}} & 0\\[0.2cm]
		D_{10} & 0 &\textbf{D}_{11}
	\end{pmatrix}.
\end{equation*}
Here, $D$ is exactly the leading order term of $\lambda^{2m-n}B_{\lambda}^*(U + vR_0(-\lambda^{2m})v)B_{\lambda}$, and the operator $Q_{m-\frac{n}{2}}T_0Q_{m-\frac{n}{2}}$ is invertible on $Q_{m-\frac{n}{2}}L^2$ by the Riesz-Schauder theory.

\item[($O_2$)] The block $\mathbf{D}_{00}$ is strictly positive on $\bigoplus_{j\in J'_{m+1}}Q_jL^2$,  corresponding to the Gram matrix of vectors 
\[
\left\{\frac{\xi^\alpha}{(2\pi)^{n/2}\alpha!(1+|\xi|^{2m})^{1/2}}\right\}_{|\alpha|\in J'_{m+1}},\quad J'_{m+1} = \{j\in J_{m+1};\, 0 \leq j \leq m-\tfrac{n+1}{2}\}.
\]
	
\item[($O_3$)] The block $\mathbf{D}_{11}$ is strictly negative on $\bigoplus_{j\in J''_{m+1}}Q_{j,1}L^2$ (where $Q_{j,1} \leq Q_j$ is defined in \eqref{eq-decom-for-Qj}), and $-\mathbf{D}_{11}$ corresponds to the Gram matrix of vectors 
\[
\left\{\frac{\xi^\alpha}{(2\pi)^{n/2}\alpha!(1+|\xi|^{2m})^{1/2}|\xi|^m}\right\}_{|\alpha|\in J''_{m+1}},\quad J''_{m+1} = \{j\in J_{m+1};\,m-\tfrac{n-1}{2} \leq j \leq 2m-\tfrac{n+1}{2}\}.
\]
\end{itemize}

Observations  $(O_2)$ and $(O_3)$ are established in Lemma~\ref{lem3.3}. Combining observations $(O_1)$, $(O_2)$ and $(O_3)$ with the abstract Feshbach formula immediately establishes the invertibility of $D$ and consequently of $D^{\pm}$. A straightforward application of the Neumann series expansion and the identity \eqref{eq-inverse-A-to-M-0} then yields all results stated in Theorem~\ref{thm3.4}.

\begin{remark}[\textbf{Comparison with previous approaches}]
We remark that previous studies have established partial results on the asymptotic expansions of $(M^\pm(\lambda))^{-1}$. The case $m=1$ can be found in e.g. \cite{Sch-G, JN}, the case $m=2$ and $1\le n\le 4$ was treated in \cite{EGT, GY, GT19, lsy, SWY}, while the regime $3m < n$ was addressed in  \cite{FSWY}. 
All these results rely on a common iterative framework due to the following abstract inversion formula.
\begin{lemma}[{\cite[Corollary 2.2.]{JN}}]\label{lem:iteration framework}
 Let $F \subset \mathbb{C}$ have zero as an accumulation point. Let $A(z)$, $z \in F$, be a family of bounded operators on the Hilbert space $\mathcal{H}$ of the form
\[
A(z) = A + z B(z),
\]
with $ B(z)$ uniformly bounded as $z \to 0 $. Suppose zero is an isolated point of the spectrum of $A$, and let $S$ be the corresponding Riesz projection. Then for sufficiently small $z \in F$ the operator $A_1(z): S\mathcal{H} \to S\mathcal{H}$ defined by
\[
\begin{aligned}
A_1(z) = \frac{1}{z} \bigl( S - S(A(z) + S)^{-1}S \bigr) 
= \sum_{j=0}^{\infty} (-1)^j z^j S \bigl[ B(z)(A + S)^{-1} \bigr]^{j+1} S
\end{aligned}
\]
is uniformly bounded as $z \to 0 $. The operator $A(z)$ has a bounded inverse in  $\mathcal{H}$ if and only if $A_1(z)$ has a bounded inverse in $S\mathcal{H}$, and in this case
\begin{equation}\label{eq:pull back}
 A(z)^{-1} = (A(z) + S)^{-1} + \frac{1}{z}(A(z) + S)^{-1} S A_1(z)^{-1} S (A(z) + S)^{-1}. 
 \end{equation}
\end{lemma}
However, the existing iterative framework based on Lemma \ref{lem:iteration framework} becomes computationally intractable for large $m$ when expanding $(M^\pm(\lambda))^{-1}$. To illustrate, consider the case $n\ll 2m$ and $\k=0$, and let $S_0L^2=\{v\}^{\perp}$ on $L^2$. It follows that $S_0$ is the Riesz projection of $a_0^{\pm}vG_0v$ at zero. 
Define operators
$$A^{\pm}(\lambda)=\lambda^{2m-n}M^{\pm}(\lambda), \quad  B_0^{\pm}(\lambda)= \lambda^{-2} \left(A^{\pm}(\lambda)-a_0^\pm vG_{0}v\right).$$
\begin{itemize}
\item \textbf{First iteration.} Applying \eqref{eq:example for m+1} and Lemma 1.10, we reduce the problem of inverting $A^{\pm}(\lambda)$ on $L^2$ to studying the invertibility of $A_1^{\pm}(\lambda)$ (restricted to the subspace $S_0L^2$):
\begin{align*}
A_1^{\pm}(\lambda)&=\lambda^{-2}\left(S_0-S_0(A_0^{\pm}(\lambda)+S_0)^{-1}S_0\right)
=\sum\limits_{j=0}^{\infty}(-1)^{j} S_0\bigl[  B_0^{\pm}(\lambda)(a_0^{\pm}vG_0v + S_0)^{-1} \bigr]^{j+1} S_0\\
&=a_1^\pm S_0vG_{1}S_0+ \lambda^{2}B_1^{\pm}(\lambda),
\end{align*}
where $B_1^{\pm}(\lambda)=\sum_{j=2}^{ m-\frac{n+1}{2}}\lambda^{2j-4} S_0B_{1,j}S_0+\cdots$ represents a combination of $L^2$-bounded operators derived from the expansion \eqref{eq:example for m+1}. The simplification in the last equality relies on the fact that $(a_0^\pm vG_{0}v+S_0)^{-1}S_0=S_0$. 

\item \textbf{Subsequent iteration.} We define $S_1L^2=\Ker(a_1^\pm S_0vG_{1}S_0)$. Applying Lemma \ref{lem:iteration framework} again, we reduce the problem of inverting $A_1^{\pm}(\lambda)$ on $S_0L^2$ to  studying the invertibility of $A_2^{\pm}(\lambda)$ (restricted to the subspace  $S_1L^2$):
\begin{align*}
A_2^{\pm}(\lambda)&=\lambda^{-2}\left(S_1-S_1(A_1^{\pm}(\lambda))^{-1}S_1\right)
=\sum\limits_{j=0}^{\infty}(-1)^{j} S_0\bigl[B^{\pm}_1(\lambda)(A_0^{\pm}(\lambda) + S)^{-1} \bigr]^{j+1} \\
&=S_1B^{\pm}_{1,2}S_1+ \lambda^{2}B_2^{\pm}(\lambda),
\end{align*}
where $B_2^{\pm}(\lambda)$ are $L^2$-bounded operators. Further iterations produce  $A_3^{\pm}(\lambda), A_4^{\pm}(\lambda), \ldots$.  
\end{itemize} 

Note  that  the number of required iterations  grows monotonically  with $m$. This leads to two computational challenges when using the iterative method:
(i) The complexity of determining each term in $A_{j}^{\pm}(\lambda)$ grows rapidly with successive iterations; (ii) Reconstructing $(M^{\pm}(\lambda))^{-1}$ via \eqref{eq:pull back} becomes computationally intractable.
\end{remark}
In contrast to the aforementioned iterative method, our matrix-based approach exploits the intrinsic global structure of the matrix operator corresponding to $M^{\pm}(\lambda)$, which avoids the complicated iterative framework and consequently leads to substantial computational simplification.

\noindent{\bf Step 2. Oscillatory integral estimates.}\

Using expansions \eqref{equ3.48-1111} of $(M^{\pm}(\lambda))^{-1}$, we write
\begin{equation*}
R_0^\pm(\lambda^{2m})v(M^\pm(\lambda))^{-1}vR_0^\pm(\lambda^{2m}) = \sum_{i,j\in J_{\k}} \lambda^{2m-n-i-j}R_0^\pm(\lambda^{2m})v\big(M_{i,j}^\pm + \Gamma_{i,j}^\pm(\lambda)\big)vR_0^\pm(\lambda^{2m}).
\end{equation*}

Proposition \ref{lemma4.2} examines the integral kernels of $Q_j v R_0^{\pm}(\lambda^{2m})$ and $Q_j v\left(R_0^{+}(\lambda^{2m})-R_0^{-}(\lambda^{2m})\right)$. This analysis enables us to express the kernel of
$$\Omega^{+, low}_{ r}-\Omega^{-, low}_{r}=e^{-\mathrm{i}tH}\chi(H)P_{ac}(H)-e^{-\mathrm{i}t(-\Delta)^{m}}\chi((-\Delta)^{m})$$
as a linear combination of oscillatory integrals of the following form
\begin{equation*}
\int_{0}^{1}\int_{0}^{1}\int_{0}^{\infty} e^{-\mathrm{i} t \lambda^{2m} + \mathrm{i} \lambda(s_1^p|y|\pm s_2^q|x|)} T^{\pm}(\lambda, x, y, s_1, s_2)\chi(\lambda^{2m}) \d\lambda \d s_1\d s_2,
\end{equation*}
where $p,q\in\{0,1\}$. Then we apply the oscillatory integral estimate provided in  Lemma \ref{lemmaA.2} to derive the low energy  estimate for the fundamental solution.

\vspace{1em}


\subsection{Organization and Notations}\

In  Section~\ref{section2},  we provide some expansions for the kernel of the free resolvents.
In Section~\ref{sect-3}, we investigate the asymptotic expansion of $(M^\pm(\lambda))^{-1}$ near zero energy for all resonance types.
In Section \ref{section-4}, we prove the   low energy estimate  Theorem \ref{theorem4.0} in a unified approach. In Section \ref{section4.2}, we prove  the high energy estimate Theorem \ref{theorem4.1}. In Section \ref{section4.4}, we prove Corollary \ref{cor-Lp}, where  the result for smoothing operators follows from a parallel argument as in Section \ref{section-4} and  \ref{section4.2} with a minor modification.

We  collect some common notations and conventions.
\begin{itemize}
	\item  $\mathbb{N}_+=\{1,2,\ldots\}$, $\mathbb{N}_0=\{0,1,2,\ldots\}$, and $\mathbb{Z}=\{0,\pm1, \pm2,\ldots\}$, $L^2=L^2(\R^n;\mathbb{C})$. 
	\item $[l]$ denotes the greatest integer at most $l$.  $a-$ (resp. $a+$) means $a-\epsilon$ (resp. $a+\epsilon$) for some $\epsilon>0$.
	\item $A\lesssim B$ means $A \leq  CB$, where $C>0$ is an absolute constant whose dependence will be specified whenever necessary, and the value of C may vary from line to line.
\end{itemize}

The following notations will also be frequently used.
\begin{itemize}
    \item $W_s(\R^n)$ (see \eqref{equ1.3});
    
	\item   $m_n$ (see \eqref{equ0.1}), $\mathbf{k}_c$ (see \eqref{equ3.3.1});
    
	\item $R_0(z)$ (see \eqref{equ2.1}), $R_0^{\pm}(\lambda^{2m})$ (see \eqref{equ:limiting operator}), $R^\pm(\lambda^2m)$ (see \eqref{eq:def of perturbed R^pm}), $I^{\pm}$ (see \eqref{equ2.1.1'});
	
	\item  $\mathfrak{S}^{b}_{K}(\Omega)$ (see \eqref{eq:def for symb class}),   $M^{\pm}(\lambda)$ (see \eqref{equ3.1});

	\item $J_{\k}$ (see \eqref{equ3.12}), $J_{\k}'$ and $J_{\k}''$ (see \eqref{equ3.12.11});
	\item  $S_j$ (see \eqref{eq-S_j-odd-1}, \eqref{eq-S_j-odd-2}), $Q_j$ (see \eqref{eq-Q_j-odd});
    \item $\delta(j)$ (see \eqref{CforQ_j-1}).
\end{itemize}

\section{Free resolvent expansions}\label{section2}

For $z\in \mathbb{C}\setminus[0,\infty)$, we set
\begin{equation}\label{equ2.1}
\text{$R_0(z)=((-\Delta)^m-z)^{-1}$ and $\mathfrak{R}_0(z)=(-\Delta-z)^{-1}$.}
\end{equation}
 It is known that $R_0(z)$ has the following expression (see \cite{EG22})
\begin{equation*}
R_0(z)=\frac{1}{mz^{1-\frac{1}{m}}}\sum_{k=0}^{m-1}e^{\i\frac{2\pi k}{m}}\mathfrak{R}_0(e^{\i\frac{2\pi k}{m}}z^{\frac{1}{m}}).
\end{equation*}
For $\lambda > 0$, the limiting absorption principle from \cite{Ag} guarantees that the weak* limits
\begin{equation}\label{equ:limiting operator}
R_0^{\pm}(\lambda^{2m}) := R_0(\lambda^{2m} \pm i0) = \text{w*-}\lim_{\epsilon \downarrow 0} R_0(\lambda^{2m} \pm i\epsilon)
\end{equation}
exist as bounded operators between appropriate weighted $L^2$ spaces. Applying a change of variables to \eqref{equ2.1} yields the decomposition:
\begin{equation}\label{equ2.1'}
	R_0^{\pm}(\lambda^{2m})= \frac{1}{m\lambda^{2m}} \sum_{k \in I^{\pm}} \lambda_k^2 \mathfrak{R}_0^{\pm}(\lambda_k^2),  \tag{2.1'}
\end{equation}
where the summation ranges and complex phases are given by
\begin{equation}\label{equ2.1.1'}
	I^+ = \{0,1,\ldots,m-1\}, \quad I^- = \{1,\ldots,m\}, \quad \lambda_k = \lambda e^{i\frac{\pi k}{m}}.
\end{equation}
The integral kernel of $\mathfrak{R}_0(z)$ has an expression as
$$\mathfrak{R}_0(z)(x-y)=\frac{\i}{4}\left(\frac{z^{\frac{1}{2}}}{2\pi\lvert x-y \rvert}\right)^{\frac{n}{2}-1}\mathcal{H}^{(1)}_{\frac{n}{2}-1}(z^{\frac12}|x-y|).$$
Here $\mathrm{Im}\,z^{\frac{1}{2}}\ge 0$ and $\mathcal{H}^{(1)}_{\frac{n}{2}-1}$ is the first Hankel function. When $n\ge 1$ is odd, it follows from explicit expression of the Hankel function (see \cite{GR02}) that
$$\mathfrak{R}_0(z)(x-y)=\frac{1}{(4\pi)^{\frac{n-1}{2}}}\sum^{\frac{n-3}{2}}_{j=\min\{0,\frac{n-3}{2}\}}{c_j(\i z^{\frac 12}|x-y|)^j},$$
where
\begin{equation}\label{equ1.4,cj}
	c_j=\begin{cases}
		-\frac12,\quad&\mbox{if}\,\, j=-1,\\
		\frac{(-2)^j(n-3-j)!}{j!(\frac{n-3}{2}-j)!},&\mbox{if}\,\, 0\le j\le \frac{n-3}{2}.
	\end{cases}
\end{equation}
Combining this with \eqref{equ2.1'}, we obtain that in old dimensions,
\begin{equation}\label{eq2.8}
\begin{aligned}
R_0^{\pm}(\lambda^{2m})(x-y)&=\frac{1}{(4\pi)^{\frac{n-1}{2}}m\lambda^{2m}|x-y|^{n-2}}\sum_{k\in I^{\pm}} {\lambda_k^2e^{ \i\lambda_k|x-y|}}\sum^{\frac{n-3}{2}}_{j=\min\{0,\frac{n-3}{2}\}}{c_j(\i\lambda_k|x-y|)^j} \\
&=\sum_{k\in I^{\pm}} {e^{ \i\lambda_k|x-y|}}\sum^{\frac{n-3}{2}}_{j=\min\{0,\frac{n-3} {2}\}} D_{j} \lambda_k^{j+2-2m} {|x-y|^{n-2-j}}.
\end{aligned}
\end{equation}

\begin{proposition}\label{pro2.1}
Let $\lambda>0$ and $\theta\in \mathbb{N}_0$. In  each old dimension $n$, we have
\begin{equation}\label{equ2.2.2}	
R_0^{\pm}(\lambda^{2m})(x-y)=\sum_{0\le j\le \frac{\theta-1}{2}}a_j^\pm\lambda^{n-2m+2j}|x-y|^{2j}+\sum_{0\le l<\frac{n-2m+\theta}{2m}}b_l\lambda^{2ml}|x-y|^{2m-n+2ml}+r_{\theta}^{\pm}(\lambda)(x-y),
	\end{equation}
where $a^{\pm}_j\in \mathbb{C}\setminus\R$ and $b_l\in \mathbb{R}\setminus\{0\}$ are defined in \eqref{eq2.10.11}, and
	\begin{equation}\label{equ2.3}
		\big|\partial_\lambda^l r_{\theta}^{\pm}(\lambda)(x-y)\big|\lesssim \lambda^{n-2m+\theta-l}|x-y|^{\theta},\quad l=0,\cdots,\mbox{$\theta+\frac{n-1}{2}$}.
	\end{equation}
\end{proposition}
\begin{proof}
For each $j\in\{\min\{0,\frac{n-3}{2}\},\cdots,\frac{n-3}{2}\}$ in \eqref{eq2.8}, we apply the Taylor formula for $e^{ \i\lambda_k|x|}$ (of order $n-j+\theta-3$) to get
\begin{equation*}\label{equ2.3.000}
  (\i\lambda_k|x|)^je^{\i \lambda_k|x|}=\sum_{l=j}^{n+\theta-3}\frac{(\i\lambda_k|x|)^l}{(l-j)!}+\frac{(\i\lambda_k|x|)^{n+\theta-2}}{(n-j+\theta-2)!} \int_{0}^{1} e^{\i s \lambda_{k}|x|}(1-s)^{n-j+\theta-3} \d s.
\end{equation*}
Plugging this into \eqref{eq2.8}, we have
\begin{equation}\label{eq:exp to infinity}
R^{\pm}_0(\lambda_k^{2m})(x-y)=\frac{1}{(4\pi)^{\frac{n-1}{2}}m} \sum_{l=0}^{n+\theta-3} d_l\sum_{k\in I^{\pm}}\lambda_k^{l+2-2m}|x-y|^{l-n+2}+
r_{\theta}^{\pm}(\lambda)(x-y),
\end{equation}
where $d_l:=\sum^{\min\{l,\, \frac{n-3}{2}\}}_{j=\min\{0,\frac{n-3}{2}\}}{\frac{c_j}{(l-j)!}},$
$c_j$ is given in \eqref{equ1.4,cj}, and
\begin{align}\label{eq2.10}
r_{\theta}^{\pm}(\lambda)(x)= \sum_{j=\min\{0,\frac{n-3}{2}\}}^{\frac{n-3}{2}}\sum_{k\in I^{\pm}} C_{j,\theta}  \lambda_{k}^{n-2 m+\theta}
|x|^{\theta } \int_{0}^{1} e^{\i s \lambda_{k}|x|}(1-s)^{n-j+\theta-3} \d s.
\end{align}
Denoted by
\begin{equation}\label{eq2.10.11}
 a_j^{\pm}=\frac{d_{2j+n-2}}{(4\pi)^{\frac{n-1}{2}}m}\sum_{k\in I^{\pm}}e^{\frac{k\pi\i}{m}(2j+n-2m)},\quad \,\, b_l=(4\pi)^{-\frac{n-1}{2}} d_{2ml+2m-2},
\end{equation}
then \eqref{equ2.2.2} follows by	 using the property
\begin{equation*}\label{eq2.8.1}
  \sum_{k=0}^{m-1}{\lambda_k^{2j}}= \sum_{k=1}^{m}{\lambda_k^{2j}}=0,\qquad \text{when}\,\,\, j\in \mathbb{Z}\setminus m\mathbb{Z},
\end{equation*}
and the fact that $d_l=0$ when $l$ is odd and $1\le l\le n-4$ (see \cite[Lemma 3.3]{J80}).

We next prove \eqref{equ2.3}. We divide the proof into two cases.

\noindent (i) When $\lambda|x-y|\le 1$, by \eqref{eq2.10}, a direct computation shows that 
\begin{align*}
\left|\partial_{\lambda}^l r_{\theta}^{\pm}(\lambda)(x-y)\right| \les \sum_{j=0}^{l}\lambda^{n-2 m+\theta-j}|x-y|^{\theta+l-j}\les \lambda^{n-2m+\theta-l}|x-y|^{\theta},
\end{align*} 
holds for $l=0,\cdots,\theta+\frac{n-1}{2}$. 

\noindent (ii) When $\lambda|x-y|> 1$, by the definition, one has 
$$r_{\theta}^{\pm}(\lambda)(x-y)=R_0^{\pm}(\lambda^{2m})(x-y)-\sum_{0\le j\le \frac{\theta-1}{2}}a_j^\pm\lambda^{n-2m+2j}|x-y|^{2j}-\sum_{0\le l<\frac{n-2m+\theta}{2m}}b_l\lambda^{2ml}|x-y|^{2m-n+2ml}. $$
Moreover, we have
\begin{align*}
	\left|\partial_{\lambda}^lR_0^{\pm}(\lambda^{2m})(x-y)\right| 
	\les& \sum^{\frac{n-3}{2}}_{j=\min\{0,\frac{n-3}{2}\}}\sum_{s=0}^{l} \lambda^{j+2-2m-s}|x-y|^{l-s+j-n+2} \\
	\les& \sum^{\frac{n-3}{2}}_{j=\min\{0,\frac{n-3}{2}\}}\sum_{s=0}^{l} \lambda^{j+2-2m-s}|x-y|^{l-s+j-n+2}(\lambda|x-y|)^{\theta-(l-s+j-n-2)} \\
	\les&  \lambda^{n-2m+\theta-l}|x-y|^{\theta},
\end{align*} 
for $l=0,\cdots,\theta+\frac{n-1}{2}$, where we use the fact that $\theta-(l-s+j-n-2)\ge 0$. Similarly, the sum 
\[
\sum_{0 \leq j \leq \frac{\theta-1}{2}} a_j^{\pm} \lambda^{n-2m+2j} |x-y|^{2j} - \sum_{0 \leq l < \frac{n-2m+\theta}{2m}} b_l \lambda^{2ml} |x-y|^{2m-n+2ml}
\] 
satisfies the same estimate if $\lambda|x-y|> 1$. This completes the proof of \eqref{equ2.3}.
\end{proof}

\begin{remark}
For each $\min\{0,\frac{n-3}{2}\}\le l\le \frac{n-3}{2}$, if we apply the Taylor formula for $e^{ \i\lambda_k|x|}$ of order $\frac{n-5}{2}-l$
then by the same arguments above  we can obtain  expansions which have better decay on $x-y$:
\begin{equation}\label{equ4.17}
	\begin{split}
		R_{0}^{\pm}(\lambda^{2 m})(x-y)=&\sum_{k\in I^{\pm}}\lambda_k^{\frac{n+1}{2}-2m}|x-y|^{-\frac{n-1}{2}} \left(\sum_{l=0}^{\frac{n-5}{2}} C_{l,-\frac{n-1}{2}} \int_{0}^{1} e^{\i s \lambda_{k}|x-y|}(1-s)^{\frac{n-5}{2}-l}\d s\right) \\
		&+\sum_{k\in I^{\pm}}D_{\frac{n-1}{2}}\lambda_k^{\frac{n+1}{2}-2m}|x-y|^{-\frac{n-1}{2}}e^{\i\lambda_{k}|x|}.
	\end{split}
\end{equation}
Here $C_{l,-\frac{n-1}{2}}$ and $D_{l}$ are absolute constants.
\end{remark}

We also need the following expression of $R_{0}(-1)(x-y)$.
\begin{lemma}\label{lemma2.2}
We have
\begin{equation}\label{equ3.34}
R_0(-1)(x-y)=\sum\limits_{|\alpha+\beta|\le 4m-n-1}(-1)^{|\beta|} A_{\alpha,\beta}x^\alpha y^\beta+b_0|x-y|^{2m-n}+b_1|x-y|^{4m-n}+r_{4m-n+1}(1)(x-y),
\end{equation}
where the remainder $r_{4m-n+1}$ satisfies the same estimate in \eqref{equ2.3} with $\theta=4m-n+1$, and
\begin{equation}\label{equ2.12.000}
 \ A_{\alpha,\beta}= \left\{
	\begin{array}{ll}
		\frac{\i^{|\alpha|+|\beta|}}{(2\pi)^n\alpha!\beta!}\int_{\R^n}\frac{\xi^{\alpha+\beta}}{1+|\xi|^{2m}}\,\d\xi,\qquad \,\qquad\,\,\,\,\mbox{if}\,\,\, 0\le |\alpha|, |\beta|<\mathbf{k}_c=\max\{m-\frac{n-1}{2},0\},\\[0.5cm]
		\frac{\i^{|\alpha|+|\beta|}}{(2\pi)^n\alpha!\beta!}\int_{\R^n}\frac{-\xi^{\alpha+\beta}}{(1+|\xi|^{2m})|\xi|^{2m}}\,\d\xi,\qquad\, \,\,\mbox{if}\,\, \max\{m-\frac{n-1}{2},0\}=\mathbf{k}_c \le |\alpha|, |\beta|\le 2m-\frac{n+1}{2}.
	\end{array}
	\right.
\end{equation}
\end{lemma}
\begin{proof}
By \eqref{equ2.1}, it follows that for $\lambda>0$,
\begin{equation*}
  R_0(-\lambda^{2m})(x)=\frac{1}{-m\lambda^{2m}}\sum_{k\in I^{\pm}}(e^{\frac{\pm\pi\i}{m}}\lambda^2_k) \mathfrak{R}_0^{\pm}\left(e^{\frac{\pm\pi\i}{m}}\lambda_k^2\right)(x).
\end{equation*}
Following the proof of Proposition \ref{pro2.1} and choosing $\theta=4m-n+1$ in \eqref{equ2.2.2}, we have
\begin{equation}\label{equ2.2}	
R_0(-\lambda^{2m})(x-y)=\sum_{j=0}^{2m-\frac{n+1}{2}}a_j \lambda^{n-2m+2j}|x-y|^{2j}+b_0|x-y|^{2m-n}+b_1\lambda^{2m}|x-y|^{4m-n}+r_{4m-n+1}(\lambda)(x-y).
	\end{equation}
where
\begin{equation}\label{equ2.7-00}
	a_j=\frac{d_{2j+n-2}}{(4\pi)^{\frac{n-1}{2}}m}e^{\frac{\pm\pi \i}{2m}(2j+n-2m)} \sum_{k\in I^{\pm}}e^{\frac{k\pi\i}{m}(2j+n-2m)}=e^{\frac{\pm\pi \i}{2m}(2j+n-2m)} a_j^\pm,
\end{equation}
and the remainder term $r_{4m-n+1}$ can be expressed as
\begin{align*}
	r_{4m-n+1}(\lambda)(x-y)=& e^{\frac{(2m+1)\pi \i}{2m}} \lambda^{2m+1}\sum_{j=\min\{0,\frac{n-3}{2}\}}^{\frac{n-3}{2}}\sum_{k\in I^{+}} C_{j,4m-n+1}e^{\frac{(2m+1)k\pi \i}{m}}
	|x-y|^{4m-n+1} \\ 
	&\times \int_{0}^{1} \exp(\i s \lambda e^{\frac{(2k+1)\pi \i}{2m}}|x-y|)(1-s)^{4m-j-2} \d s.
\end{align*} 
Hence \eqref{equ3.34} follows immediately by the following identity
\begin{equation}\label{eq:exp for |x-y|}
	|x-y|^{2j}=\sum_{|\alpha|+|\beta|=2j}C_{\alpha,\beta}(-1)^{|\beta|}x^{\alpha}y^{\beta},
\end{equation}
and setting
\begin{equation}\label{equp3.18.alpha}
A_{\alpha,\beta}=a_{\frac{|\alpha|+|\beta|}{2}}C_{\alpha,\beta}.
\end{equation}

Now  we prove \eqref{equ2.12.000}.
Note that when $0\le |\alpha|, |\beta|<\mathbf{k}_c=\max\{m-\frac{n-1}{2},0\}$, it suffices to consider  $1\le n<2m$ since
$\mathbf{k}_c=0$ if $n>2m$.
When $1\le n< 2m$,  by \eqref{equ3.34}, we deduce that
\begin{equation*}\label{equ3.35}\begin{split}
(-1)^{|\beta|} A_{\alpha,\beta}\alpha!\beta!
	&=\lim\limits_{x,y\to0}\partial_x^\alpha\partial_y^\beta R_0(-1)(x-y)
	\\&=\lim\limits_{x,y\to0}\frac{1}{(2\pi)^n}\partial_x^\alpha\partial_y^\beta\left(\int_{\R^n}\frac{e^{\i(x-y)\cdot\xi}}{|\xi|^{2m}+1}\,\d\xi\right)
	\\&=\frac{\i^{|\alpha|+|\beta|}(-1)^{|\beta|}}{(2\pi)^n}\int_{\R^n}\frac{\xi^{\alpha+\beta}}{|\xi|^{2m}+1}\, \d\xi,
	\end{split}
\end{equation*}
where the last two equalities follow from the fact that $\xi^{\alpha+\beta}/{(|\xi|^{2m}+1)}\in L^1$ and the dominated convergence theorem,   since $0\le |\alpha|, |\beta|<\mathbf{k}_c=m-\frac{n-1}{2}$.

On the other hand, if $\mathbf{k}_c \le |\alpha|, |\beta|\le 2m-\frac{n+1}{2}$, it follows from  \eqref{equ3.34} that
\begin{equation}\label{equ3.37}
(-1)^{|\beta|} A_{\alpha,\beta}\alpha!\beta!
	=\lim\limits_{x,y\to0}\partial_x^\alpha\partial_y^\beta \big(R_0(-1)(x-y)-b_0|x-y|^{2m-n}\big).
\end{equation}
Observe that $b_0|\cdot|^{2m-n}$ is the fundamental solution of $(-\Delta)^m$, then
\begin{equation*}\label{equ3.38}
	(-\Delta)^m\partial^{\alpha+\beta}(b_0|\cdot|^{2m-n})=\partial^{\alpha+\beta}(-\Delta)^m(b_0|\cdot|^{2m-n})=\partial^{\alpha+\beta}\delta_0
\end{equation*}
holds in distributional sense. Taking Fourier transform on both sides of \eqref{equ3.38} yields
\begin{equation*}\label{equ3.39}
	\mathcal{F}\Big(\partial^{\alpha+\beta}(b_0|\cdot|^{2m-n})\Big)(\xi)=\frac{\i^{|\alpha|+|\beta|}\xi^{\alpha+\beta}}{|\xi|^{2m}}.
\end{equation*}
 Note that $|\alpha|, |\beta|\ge \mathbf{k}_c$, thus $|\alpha|+|\beta|-2m>-n$,
and  $\xi^{\alpha+\beta}/|\xi|^{2m}$ is a tempered distribution. Using Fourier inversion formula we deduce that
\begin{equation*}\label{equ3.40} \partial_x^\alpha\partial_y^\beta\big(b_0|x-y|^{2m-n}\big)=\frac{\i^{|\alpha|+|\beta|}}{(2\pi)^n}\lim\limits_{\epsilon\to0+}\int_{\R^n}\frac{\xi^{\alpha+\beta}e^{\i\xi(x-y)-\epsilon|\xi|^2}}{|\xi|^{2m}}\d\xi.
\end{equation*}
On the other hand,
\begin{equation*}\label{equ3.41}
	\partial_x^\alpha\partial_y^\beta R_0(-1)(x-y)=\frac{\i^{|\alpha|+|\beta|}}{(2\pi)^n}\lim\limits_{\epsilon\to0+}\int_{\R^n}\frac{\xi^{\alpha+\beta}e^{\i\xi(x-y)-\epsilon|\xi|^2}}{|\xi|^{2m}+1}\d\xi.
\end{equation*}
Combining \eqref{equ3.37} and the above two relations,
we have
\begin{equation*}\label{equ3.42}\begin{split}
		(-1)^{|\beta|} A_{\alpha,\beta}\alpha!\beta!
		&=\lim\limits_{x,y\to0}\lim\limits_{\epsilon\to0+}
		\frac{(-1)^{|\beta|} \i^{|\alpha|+|\beta|}}{(2\pi)^n}
		\int_{\R^n}\frac{-\xi^{\alpha+\beta}e^{\i\xi(x-y)-\epsilon|\xi|^2}}{(1+|\xi|^{2m})|\xi|^{2m}}\d\xi
		\\&=\frac{(-1)^{|\beta|} \i^{|\alpha|+|\beta|}}{(2\pi)^n}
		\int_{\R^n}\frac{-\xi^{\alpha+\beta}}{(1+|\xi|^{2m})|\xi|^{2m}}\d\xi,
	\end{split}
\end{equation*}
where the last equality follows from the fact  $\frac{\xi^{\alpha+\beta}}{(1+|\xi|^{2m})|\xi|^{2m}}\in L^1(\R^n)$ with  $|\alpha|, |\beta|\le 2m-\frac{n+1}{2}$.
\end{proof}

\begin{remark}
The  identity \eqref{equ2.7-00}, namely
\begin{equation}\label{equ2.7}
	a_j=e^{\frac{\pm\pi \i}{2m}(2j+n-2m)} a_j^\pm,
\end{equation}
establishes the precise relationship between the coefficients of $R_0(-\lambda^{2m})$ and $R^{\pm}_0(\lambda^{2m})$.  This relationship is crucial for for the proof of Theorem \ref{thm3.4} presented in the following section.. 
\end{remark}

\section{Asymptotic expansions of the perturbed resolvent around zero energy}\label{sect-3}
	This section  is devoted to deriving the asymptotic expansion of the perturbed resolvent $(M^{\pm}(\lambda))^{-1}$ near zero energy, which is essential for understanding  the low energy behavior of the fundamental solution of \eqref{equ1.1.1}. 

For notational convenience, we say that $T(\lambda) \in \mathfrak{S}^{b}_{K}(\Omega)$ for an open set $\Omega\subset\mathbb{R}$ if $\{T(\lambda)\}_{\lambda \in \Omega}$ is a family of operators in $L^2$ satisfying
\begin{equation}\label{eq:def for symb class}
	\left|\langle \partial_{\lambda}^{j}T(\lambda)f, g \rangle_{L^2 \times L^2}\right| \leq C_j \|f\|_{L^2} \|g\|_{L^2} |\lambda|^{b-j}, \quad \lambda \in \Omega, ~ 0 \leq j \leq K,
\end{equation}
for all $f, g \in L^2$, where the constants $C_j$ are independent of $f$, $g$, and $\lambda$. 

 We emphasize that the asymptotic behavior varies depending on the different spectral conditions at zero energy. To systematically address this dependence, we introduce a family of index sets  $J_{\mathbf{k}}$ parameterized by the resonance kind $\mathbf{k}$.
Specifically, if $1\le n\leq 2m-1$,  then denote
\begin{equation}\label{equ3.12}
	J_{\k}=\begin{cases}\{0, \cdots,\  m-\frac{n+1}{2},\  m-\frac{n}{2}\},\quad&\mbox{if}\,\, \k=0,\\
		\{0, \cdots,\  m-\frac{n+1}{2},\   m-\frac{n}{2},\  m-\frac{n-1}{2}, \ \cdots,\    m-\frac{n-1}{2}+\k-1\},\quad&\mbox{if}\,\, 1\le \k\le m_n=m,\\
		\{0, \cdots,\  m-\frac{n+1}{2},\  m-\frac{n}{2},\   m-\frac{n-1}{2}, \ \cdots,\    2m-\frac{n+1}{2},\ 2m-\frac{n}{2}\},&\mbox{if}\,\, \k=m_n+1.
	\end{cases}
\end{equation}
If $2m+1\leq n\leq 4m-1$,  then denote
\begin{equation}\label{equ3.12-1}
	J_{\k}=\begin{cases}
		\{ m-\frac{n}{2}\},\quad&\mbox{if}\, \, \k=0,\\
		\{ m-\frac{n}{2},\ 0,  \cdots, \ \k-1 \},\quad&\mbox{if}\, \, 1\le \k\le m_n=2m-\frac{n-1}{2},\\
		\{ m-\frac{n}{2},\ 0,  \cdots, \ 2m-\frac{n+1}{2},\  2m-\frac{n}{2} \},&\mbox{if}\,\, \k=m_n+1.
	\end{cases}
\end{equation}
Here $m_n=\min\{m, 2m-\frac{n-1}{2}\}$ is given by \eqref{equ0.1}. Note that by definition of $J_{\k}$ above, one has
\begin{equation}\label{eq:max J_k}
	\max J_{\k}=
	\begin{cases}
		m-\frac n2, &\quad \mbox{if} ~\k=0, \\
		\k_c+\k-1, & \quad \mbox{if}~ 1\le \k\le m_n, \\
		2m-\frac{n}{2},  &\quad \mbox{if}~ \k=m_n+1.
	\end{cases}
\end{equation}
Here  $\k_c=\max\{m-\frac{n-1}{2},0\}$ is given in \eqref{equ3.3.1}.
For integers $j> -n$ we define the operators 
\begin{equation}\label{eq:def of G_j}
	G_{j}f(x)=\int_{\R^n}|x-y|^{j}f(y)\,\d y,\quad \, f\in \mathcal{S}(\R^n),
\end{equation}
and set
\begin{equation}\label{equ3.5}
	T_0=U+b_0vG_{2m-n}v,
\end{equation}
where $b_0$ is given by \eqref{eq2.10.11}. The main result in this section is the following
\begin{theorem}\label{thm3.4}
Under Assumption \ref{assum1}, there exists $\lambda_0 \in (0,1)$ such that $M^\pm(\lambda)$ have bounded inverses in $L^2$ for all $0 < \lambda < \lambda_0$, which admit the expansions
\begin{equation}\label{equ3.48}
	(M^\pm(\lambda))^{-1} = \sum_{i,j\in J_{\mathbf{k}}} \lambda^{2m-n-i-j} Q_i \big(M^\pm_{i,j} + \Gamma_{i,j}^\pm(\lambda)\big) Q_j,
\end{equation}
where $Q_j$ are projection operators defined in \eqref{eq-Q_j-odd}, $M^\pm_{i,j}$ and $\Gamma_{i,j}^\pm(\lambda)$ are bounded operators on $L^2$.  The $\lambda$-dependent operators $\Gamma_{i,j}^\pm(\lambda)$ satisfy
\begin{equation}\label{equ3.47}
	\Gamma_{i,j}^\pm(\lambda) \in \mathfrak{S}^{\frac{1}{2}}_{\frac{n+1}{2}}((0,\lambda_0)).
\end{equation}
In addition,  we have  $$M^\pm_{m-\frac{n}{2},m-\frac{n}{2}}=(Q_{m-\frac{n}{2}}T_0Q_{m-\frac{n}{2}})^{-1}, \,\,\, M^\pm_{2m-\frac{n}{2},2m-\frac{n}{2}}=(b_1Q_{2m-\frac{n}{2}}vG_{4m-n}vQ_{2m-\frac{n}{2}})^{-1},$$ 
and 
$$M^\pm_{i,j}=0,\quad \text{when $i \neq j$ with either $i$ or $j \in\left\{\mbox{$m-\frac{n}{2}, 2m-\frac{n}{2}$}\right\}$.} $$ 
Finally,
\begin{equation}\label{equ3.47.1}
	\begin{cases}
		\Gamma_{m-\frac{n}{2},m-\frac{n}{2}}^\pm(\lambda) \in \mathfrak{S}^{\max\{1, n-2m\}}_{\frac{n+1}{2}}((0,\lambda_0)), & \text{if } \mathbf{k}=0, \\
		\Gamma_{2m-\frac{n}{2},2m-\frac{n}{2}}^\pm(\lambda) \in \mathfrak{S}^{1}_{\frac{n+1}{2}}((0,\lambda_0)), & \text{if } \mathbf{k}=m_n+1.
	\end{cases}
\end{equation}
\end{theorem}

We postpone the proof of Theorem \ref{thm3.4} to Subsection \ref{sec3.3}. In  Subsections \ref{section3.1}-\ref{section3.2}, we  first  establish the following key auxiliary results: (i) Spectral characterizations corresponding to different resonance types, along with a partition of the identity for relevant  operators; (ii) Criteria for determining positive and negative definiteness of certain  operator block matrices.

\begin{remark}\label{rmk-Mij-inverse}
	We mention that in the special case where $m=2$, $1\le n\le 4$, the expansions of $(M^\pm(\lambda))^{-1}$ were obtained in \cite{SWY,lsy,GT19,EGT,GY} in slightly different forms of \eqref{equ3.48}.
\end{remark}

\subsection{Characterizations of zero energy resonances}\label{section3.1}\

In this subsection, we  first introduce the projection operators $Q_j$ featured in Theorem~\ref{thm3.4} and then establish an equivalent spectral characterization of zero energy resonances using these operators.

Recall that 
\begin{equation}\label{equ3.1}
M^{\pm}(\lambda)=U+vR_0^\pm(\lambda^{2m})v,\quad v(x)=|V(x)|^{\frac12}, \,\,U(x)=\sgn V(x),
\end{equation}
where $\sgn x=1$, if $x\ge 0$ and $\sgn x=-1$, if $x< 0$. Applying the expansion \eqref{equ2.2.2} with $\theta=4m-n+1$, we can expand $M^{\pm}(\lambda)$ as follows.

\noindent (i) If $1\le n\le 2m-1$, then 
\begin{equation}\label{eq:exp for M-1}
\begin{aligned}
M^{\pm}(\lambda)=&\sum\limits_{j=0}^{ m-\frac{n+1}{2}}\lambda^{n-2m+2j}a_j^\pm vG_{2j}v+T_0 \\
&+\sum_{j=m-\frac{n-1}{2}}^{2m- \frac{n+1}{2}}\lambda^{n-2m+2j}a_j^\pm vG_{2j}v+ \lambda^{2m}b_1vG_{4m-n}v+vr_{4m-n+1}^\pm(\lambda)v.	
\end{aligned}
\end{equation}
\noindent (ii) If $2m+1\le n\le 4m-1$, then 
\begin{equation}\label{eq:exp for M-2}
	M^{\pm}(\lambda)=T_0+
	\sum_{j=0}^{2m- \frac{n+1}{2}}\lambda^{n-2m+2j}a_j^\pm vG_{2j}v+ \lambda^{2m}b_1vG_{4m-n}v+vr_{4m-n+1}^\pm(\lambda)v.
\end{equation}

 Now we introduce certain subspaces of $L^2$ and their associated orthogonal projections related to $T_0$. In what follows,  $X^\perp$ denotes the orthogonal complement of a subspace $X\subset L^2(\R^n)$,  and ${\rm{Ker}}(T)$ represents the kernel  of an operator $T$ on $L^2(\R^n)$.  
 
 Denoted by $S_{-j}=I$ when $j\in \N_+$. On the other hand,  in all old dimensions, we always define 
 \begin{equation}\label{eq-S_2m-n2} 
 	S_{2m-\frac n2}L^2=\{0\}.
 \end{equation}
 For a fixed integer   $0\le \k\le m_n+1$, the precise definition of $S_j L^2$ depends on whether $n<2m$ or $n>2m$.
 
\textbf{Case 1:  $1\le n \le2m-1$}. We define
$$S_j L^2=\{x^\alpha v;\,\,|\alpha|\le j\}^\perp,\quad \mbox{if}\,\ 0\le j\le m-\mbox{$\frac{n+1}{2}$},$$
and 
 \begin{equation}\label{eq-S_j-odd-1}
 		S_j L^2=\begin{cases}
 			\mbox{Ker}\,(S_{m-\frac{n+1}{2}} T_0 S_{m-\frac{n+1}{2}})\bigcap S_{m-\frac{n+1}{2}} L^2,\,&\mbox{if}\,\ j=m-\frac{n}{2},\\
 			\mbox{Ker}(S_{2m-n-j-1}T_0S_{m-\frac{n}{2}})\bigcap\{x^\alpha v;\,\,\,|\alpha|\le j\}^\perp\bigcap S_{m-\frac{n}{2}} L^2,
 			&\mbox{if}\,\ m-\frac{n-1}{2}\le j\le  \max J_{\k},\\
 			\{0\},&\mbox{if}\,\ j>\max J_{\k}.
 		\end{cases}
 \end{equation}
 
  \textbf{Case 2:  $2m+1\leq n\leq 4m-1$}. In this case, we define
 \begin{equation}\label{eq-S_j-odd-2}
 	S_j L^2=\begin{cases}
 		\mbox{Ker}\,(T_0), \,&\mbox{if}\,\ j=m-\frac{n}{2},\\
 		\{x^\alpha v;\,\,\,|\alpha|\le j\}^\perp
 		\bigcap S_{m-\frac{n}{2}} L^2,
 		&\mbox{if}\,\ 0\le j\le  \max J_{\k},\\
 		\{0\},&\mbox{if}\,\ j>  \max J_{\k}.
 	\end{cases}
 \end{equation}
 We make the following remarks concerning the operators $S_j$:
 \begin{itemize}

	\item[($1$)]  The orthogonal projections $S_j$ are well-defined by (ii) of Assumption \ref{assum1} for a fixed resonance kind $\k$. 

    \item[($2$)]  For each $j\in J_{m_n+1}$, the projections satisfy the following  cancellation property:
\begin{equation}\label{CforQ_j}
	S_j(x^\alpha v)=0, \quad \text{for}~\alpha\in {\mathbb{N}^n_0},~|\alpha|\leq\max\{0,[j]\}.
\end{equation}

    \item[($3$)] When $1\le n\le 2m-1$, the subspaces form the following inclusions:
\begin{equation*}
	S_0L^2\supset S_1L^2\supset\cdots\supset S_{m-\frac{n+1}{2}}L^2\supset \mathbf{S_{m-\frac{n}{2}}L^2} \supset S_{m-\frac{n-1}{2}}L^2\supset\cdots\supset S_{2m-\frac{n+1}{2}}L^2\supset S_{2m-\frac{n}{2}}L^2=\{0\}.
\end{equation*}
When $2m+1\le n\le 4m-1$,  the subspace hierarchy becomes
\begin{equation*}
\mathbf{S_{m-\frac{n}{2}}L^2}\supset S_0L^2\supset S_1L^2\supset\cdots\supset S_{2m-\frac{n+1}{2}}L^2 \supset S_{2m-\frac{n}{2}}L^2=\{0\}.
\end{equation*}

    \item[($4$)]  The Riesz-Schauder theory implies that the space $S_{m-\frac{n}{2}}L^2$ is finite-dimensional.

    \end{itemize}

We define  a second family  of orthogonal projections $\{Q_j\}_{j\in J_{\k}}$ as follows: 
\begin{equation}\label{eq-Q_j-odd}
	Q_j:=
	\begin{cases}   
		I-S_0, &\quad \text{if $1\le n\le 2m-1$ and $j=0$}, \\
		I-S_{m-\frac n2},  &\quad \text{if $2m< n <4m$ and $j=m-\frac n2$}, \\
		S_{j'}-S_j,   &\quad \text{else},     
	\end{cases}
\end{equation}
where $$ j'=\max\{l\in J_{m_n+1};\,l<j\}.$$
By definition of $Q_j$, it follows for all $j\in J_{m_n+1}$ that
\begin{equation}\label{CforQ_j-1}
	Q_j(x^\alpha v)=0, \quad \,\,\,~\alpha\in {\mathbb{N}^n_0},~|\alpha|\leq\max\{0,[j+1/2]\}-1:=\delta(j)-1.
\end{equation}
This property will be used frequently in our proof.

The following remarks on  $\{Q_j\}_{j\in J_{\k}}$ are in order.
\begin{itemize}
\item[($1$)]  Similar projection operators have been defined in \cite{EGT,FSY,FSWY,MWY,SWY}, among others.

\item[($2$)]  When $2m< n\le 4m-1$ and  $0\le j\le 2m-\frac{n+1}{2}$, one has 
\begin{equation}\label{eq:Q_j for large j-2}
	Q_{j}L^2=\{T_0\psi=0\,\mid\, \langle\psi, x^{\alpha}v \rangle=0 ~\text{for all}~|\alpha|\le j-1 ~\text{but}~ \langle\psi, x^{\alpha}v \rangle\neq 0 ~\text{for some $\alpha$ with}~|\alpha|= j \}. 
\end{equation}
    \end{itemize}
    
These projections are motivated by the following key proposition, which  establishes a bijection between  resonance functions (solutions to \eqref{equ1.2}) and  solutions to the equation
\begin{equation}\label{eq:the kernel of T_0}
	T_0\psi = U\psi + b_0vG_{2m-n}v\psi = 0, \quad \psi \in L^2,
\end{equation}
where $T_0$ defined in \eqref{equ3.5} appears in the asymptotic expansions of  $M^{\pm}(\lambda)$.

\begin{proposition}\label{pro:one to one corres of solutions} Under decay assumption \eqref{eq:decay assump} for some fixed resonance kind $\k\in\{1,2,\ldots,m_n+1\}$, if $\kappa\in\{1,2,\ldots,\k+1\}$ and $\kappa\le m_n+1$,  then the following statements are equivalent: 

\noindent (i) $0\ne \phi\in W_{-1/2-m_n+\kappa}$ is a  solution  to the equation \eqref{equ1.2}.

\noindent  (ii) $0\neq \psi=Uv\phi\in S_{j_{\kappa-1}}L^2$, where $j_{\kappa-1}=\max J_{\kappa-1}$.
\end{proposition}

\begin{proof}
We organize the proof according to different dimensional regimes $n<2m$ and $n>2m$.

\textbf{The case $1\le n\le 2m-1$}. 
We first prove  (ii) implies   (i).  Assume that $\psi\in S_{j_{\kappa-1}}L^2$, then one has
$$
\begin{cases}
S_{m-\frac{n+1}{2}}T_0\psi=0,\quad j_{0}=m-\frac n2,  \quad &\text{if $\kappa=1$},	\\
S_{m-\frac{n-1}{2}-\kappa}T_0\psi=0,\quad j_{\kappa-1}=m-\frac{n-1}{2}+\kappa-2, \quad &\text{if $1<\kappa\le m_n+1=m+1$},
\end{cases}
$$
and 
\begin{equation} \label{eq:cancell propery for kappa}
\int x^{\alpha} v(x)\psi(x) dx=0,\quad \text{for all}~ \mbox{$|\alpha|\le [j_{\kappa-1}]=m-\frac{n-1}{2}+\kappa-2$}.
\end{equation}
Thus we have
\begin{align}\label{equ.psi}
	U\psi(x)=
	\begin{cases}
		-b_0vG_{2m-n}v\psi(x)+(I-S_{m-\frac{n-1}{2}-\kappa})T_0\psi(x), \quad &\text{if $1\le \kappa\le  m-\frac{n-1}{2}$},\\
		-b_0vG_{2m-n}v\psi(x), \quad 	&\text{if $m-\frac{n-1}{2}<\kappa\le m+1$}.
	\end{cases}
\end{align}
The definition of $S_{m-\frac{n-1}{2}-\kappa}$ indicates that
$$(I-S_{m-\frac{n-1}{2}-\kappa})T_0\psi(x)=\sum_{|\alpha|\le  m-\frac{n-1}{2}-\kappa}C_\alpha x^{\alpha}v.$$
Set
\begin{equation}\label{eq-phi-729}
   \phi(x)=-b_0G_{2m-n}v\psi(x)+\sum_{|\alpha|\le  m-\frac{n-1}{2}-\kappa}C_\alpha x^{\alpha}, 
\end{equation}
where we note that when $m-\frac{n-1}{2}-\kappa<0$, the sum $\sum_{|\alpha|\le  m-\frac{n-1}{2}-\kappa}C_\alpha x^{\alpha}$ vanishes identically. By comparing \eqref{equ.psi} and \eqref{eq-phi-729},  we obtain $\psi = Uv\phi$. Furthermore,  since
$$(-\Delta)^m b_0G_{2m-n}v\psi = v\psi=V\phi\quad\mbox{and}\quad (-\Delta)^{m}\Big(\sum_{|\alpha|\le  m-\frac{n-1}{2}-\kappa}C_\alpha x^{\alpha}\Big)=0,$$ 
we conclude that $((-\Delta)^m + V)\phi = 0$ holds in distributional sense.

Now we  analyze the  term $-b_0vG_{2m-n}v\psi(x)$. Under the decay assumption \eqref{eq:decay assump} for $V$, we have
$$v\psi \in L^{2}_{2m-\frac n2+2\k+2+},$$ 
where the weight index satisfies
$$\mbox{$2m-\frac n2+2\k+2+>\max\{2m-n, [j_{\kappa-1}]+1\}+\frac n2$}.$$
Applying Lemma \ref{lemmaA.1} and the cancellation property \eqref{eq:cancell propery for kappa}, we obtain the pointwise estimate
\begin{equation}\label{eq-727-2}
    |G_{2m-n}v\psi(x)|=\left| \int|x-y|^{2m-n}v(y)\psi(y)\right|\lesssim\langle x\rangle^{2m-n-[j_{\kappa-1}]-1} = \langle x\rangle^{m-\frac{n-1}{2}-\kappa}.
\end{equation}
This implies
\begin{equation}\label{eq-727-3}
    \phi(x)\in W_{-\frac{1}{2}-m+\kappa}(\R^n).
\end{equation}

Next, we show that  (i) implies   (ii).  Assume that there  exists  $\phi(x)\in W_{-\frac{1}{2}-m+\kappa}(\mathbb{R}^n)$ such that \eqref{equ1.2} holds for $1\le \kappa \le m_n+1=m+1$ and $\kappa \le \k+1$. 
Let $\psi=Uv\phi$.
Choose $\eta(x)\in C_0^{\infty}(\R^n)$  such that $\eta(x)=1$ for $|x|\le 1$ and $\eta(x)=0$ for $|x|>2$. For any $\alpha\in {\N_0}^n$ with $|\alpha|\le  [j_{\kappa-1}]=m-\frac{n-1}{2}+\kappa-2$, we obtain from \eqref{equ1.2} that  
\begin{equation}\label{eq:to obtain cancell}
	\begin{aligned}
		&\left|\int_{\R^n}x^{\alpha}v(x)\psi(x)\eta(\delta x)dx \right|= \left| \int_{\R^n}x^{\alpha}\eta(\delta x)(-\Delta)^m\phi(x)\,\d x\right|\\
		=&\bigg|\sum_{\alpha_i\geq\beta_i\  \text{for}\ i=1,\cdots,n \atop |\beta|+|\gamma|=2m,\  \beta,\gamma \in\N_0^n } C_{\beta,\gamma}\delta^{|\gamma|}\int_{\R^n}x^{\alpha-\beta}\phi(x)(\partial^{\gamma}\eta)(\delta x)\,\d x\bigg| \\
		\lesssim & \sum_{\alpha_i\geq\beta_i\  \text{for}\ i=1,\cdots,n \atop |\beta|+|\gamma|=2m,\  \beta,\gamma \in\N_0^n } \delta^{m-\frac{n+1}{2}+\kappa-|\alpha|-} \left\|\langle x\rangle^{|\alpha|-|\beta|+m+\frac12-\kappa+}\partial^{\gamma}\eta(x) \right\|_{L^2}  \left\|\langle x\rangle^{-(m+\frac12-\kappa+)} \phi(x)\right\|_{L^2} \\
		\rightarrow & 0,\qquad \mbox{as}\,\,\,\delta\rightarrow0,
	\end{aligned}
\end{equation}
 where the last inequality follows from scaling and the fact  $$\mbox{$|\gamma|+|\beta|-|\alpha|-(m+\frac12-\kappa)-\frac{n}{2}-=m-\frac{n+1}{2}+\kappa-|\alpha|-.$}$$
 In particular, we have
\begin{equation}\label{equA.7}
	\int_{\R^n}x^{\alpha}v(x)\psi(x)\d x=	\int_{\R^n}x^{\alpha}V(x)\phi(x)\d x=0,\quad \, \,\text{if}\, \, \mbox{$|\alpha|\le [j_{\kappa-1}]=m-\frac{n-1}{2}+\kappa-2$}.
\end{equation}
Observe that
\begin{align*}
S_{m-\frac{n-1}{2}-\kappa}vG_{2j}V\phi(x) 
&=S_{m-\frac{n-1}{2}-\kappa}v \left( \int_{\R^n}|\cdot-y|^{2j}V(y)\phi(y)\d y\right)(x) \\
&=\sum_{ \alpha,\, \beta \in\N_0^n\atop |\alpha|+|\beta|\le 2j} C_{\alpha,\beta} \,  S_{m-\frac{n-1}{2}-\kappa}(z^{\alpha} v(z)) (x) \int_{\R^n} y^{\beta} V(y)\phi(y)\d y \\
&=0
\end{align*}
holds for $j=0,1,\ldots, m-\frac{n+1}{2}$. Thus, applying the expansion \eqref{equ2.2} and the dominated convergence theorem, it follows that for any Schwartz function $f$, 
\begin{equation*}
	\begin{aligned}
	\big\langle S_{m-\frac{n-1}{2}-\kappa}(v b_0G_{2m-n}^0v)\psi,\, f\big\rangle&=\big\langle S_{m-\frac{n-1}{2}-\kappa}(v b_0G_{2m-n}^0)V\phi,\, f \big\rangle \\
     &=\lim_{\lambda\rightarrow0} \big\langle  S_{m-\frac{n-1}{2}-\kappa} vR_0(-\lambda^{2m})V\phi, \, f \big\rangle \\
	&=\lim_{\lambda\rightarrow0} \big\langle S_{m-\frac{n-1}{2}-\kappa} v \big((-\Delta)^{m}+\lambda^{2m}\big)^{-1}V\phi(x), f \big\rangle \\
	&= \big\langle S_{m-\frac{n-1}{2}-\kappa} v(-\Delta)^{-m}V\phi(x), f \big\rangle.
	\end{aligned}
\end{equation*}
 Noting $\phi=-(-\Delta)^{-m}V\phi(x)$ for $\phi$ satisfies \eqref{equ1.2}, the above identity implies that 
 $$	\big\langle S_{m-\frac{n-1}{2}-\kappa}(v b_0G_{2m-n}^0v)\psi,\, f\big\rangle=-\big\langle S_{m-\frac{n-1}{2}-\kappa} v\phi(x), f \big\rangle=-\big\langle S_{m-\frac{n-1}{2}-\kappa} U\psi(x), f \big\rangle, $$ 
 that is
 $$ S_{m-\frac{n-1}{2}-\kappa}T_0\psi=0.$$
 This combined with \eqref{equA.7} implies that $\psi \in S_{j_{\kappa-1}}L^2$, 
which completes the proof for the case $1\le n < 2m$.

\textbf{The case $2m< n< 4m$}. 
The argument in these dimensions proceeds analogously to the proof for previous case. Indeed, to prove (i) from (ii), observe that by definition \eqref{eq-S_j-odd-2},
 $\psi\in S_{j_{\kappa-1}}L^2$ if and only if
$$T_0\psi=0,$$
and 
\begin{equation} \label{eq:cancell propery for kappa-1}
	\int x^{\alpha} v(x)\psi(x) \d x=0,\quad \text{for all}~ |\alpha|\le [j_{\kappa-1}]=\kappa-2,
\end{equation}
where we use the fact 
$$j_{\kappa-1}=
\begin{cases}
m-\frac n2,\quad &\kappa=1,\\
\kappa-2,\quad &1<\kappa\le m_n+1=2m-\frac{n-1}{2}+1.
\end{cases}
$$
Then letting $\phi=-b_0G_{2m-n}v\psi(x)$, we have $\psi=Uv\phi$. It follows from  Lemma \ref{lemmaA.1} that 
$$|\phi(x)|=
|b_0G_{2m-n}v\psi(x)|=\left| b_0 \int|x-y|^{2m-n}v(y)\psi(y) \d y\right|\lesssim\langle x\rangle^{2m-n-[j_{\kappa-1}]-1} = \langle x\rangle^{2m-n-\kappa+1},
$$
which implies $\phi \in  W_{\frac n2-2m+\kappa-1}=W_{-\frac12-m_n+\kappa}$.

Conversely,  we prove that (i) implies (ii).  Assume the existence of $\phi(x) \in W_{-\frac{1}{2} - m_n + \kappa}(\mathbb{R}^n)$ such that \eqref{equ1.2} holds for some $\kappa$ with $1 \leq \kappa \leq m_n + 1 = m + 1$ and $\kappa \leq \k + 1$. Following the same argument as in \eqref{eq:to obtain cancell},  $\psi$ satisfies \eqref{eq:cancell propery for kappa-1}. Note that by \eqref{equ1.2} we have
$$\phi=-(-\Delta)^{-m}V\phi=-b_0G_{2m-n}v\psi.$$
 This implies that $T_0\psi=0$, and therefore $\psi \in S_{j_{\kappa-1}}L^2$ holds, which completes the proof for the case $2m<n\le 4m-1$.
\end{proof}

\begin{remark}\label{remk:for eigenfunctions}
Following the proof of Proposition~\ref{pro:one to one corres of solutions},  when decay assumption \eqref{eq:decay assump} is valid for $\k = m_n$ or $\k = m_n + 1$, we observe that if $\phi \in L^2$ is a solution of \eqref{equ1.2}, then $\psi = Uv\phi$ belongs to $S_{2m - \frac{n+1}{2}}L^2$ $($as seen from the proof of sufficiency$)$. Since $W_{\frac12}\subset L^2$, this implies that $\phi$ is an eigenfunction of $H$ at zero energy if and only if $0\neq\psi = Uv\phi \in S_{2m - \frac{n+1}{2}}L^2$.
\end{remark}

	
	

\begin{proposition}\label{prop3.1}
Let the decay assumption \eqref{eq:decay assump} hold for some $\k\in\{0,\cdots,m_n+1\}$. Then the following statements hold:

\noindent\emph{(\romannumeral1)} Zero energy is a regular point (i.e., $\k=0$) of $H$ if and only if  $\sum_{j\in J_{0}}Q_j=I$.

\noindent\emph{(\romannumeral2)} Zero energy is a resonance of the $\mathbf{k}$-th kind ($1\leq\k\leq m_n+1$) of $H$ if and only if  $\sum_{j\in J_{\mathbf{k-1}}}Q_j\neq I$ and $\sum_{j\in J_{\mathbf{k}}}Q_j=I$.
\end{proposition}

\begin{proof}
    Note that $\{j_{\k} \}=J_{\k}\setminus J_{\k-1}$ for $\k=1,2,\ldots,m_n+1$.  By the definition of $Q_j$, one has
$$\sum_{j\in J_{\k}}Q_j=I-S_{j_{\k}}.$$
 Thus the results follow from Proposition \ref{pro:one to one corres of solutions} and Remark~\ref{remk:for eigenfunctions}. 
\end{proof}

\subsection{A key observation}\label{section3.2}\

We begin by defining the index sets
\begin{equation}\label{equ3.12.11}
	J'_{\k}=\{j\in J_{\k};\,j<\mbox{$m-\frac{n}{2}$}\},\quad	J''_{\k}=\{j\in J_{\k};\,\mbox{$m-\frac{n}{2}$}<j<2m-\mbox{$\frac n2$}\}.
\end{equation}
In particular, $J'_{\k}=\emptyset$ when $2m< n< 4m$.
We view  the subspace $X_0=\bigoplus_{j\in J'_{\k}} Q_jL^2$ as a vector valued space, on which we define the operator block matrix  
\begin{equation}\label{eq:def of bfD_00}
\textbf{D}_{00}=(d_{l,h})_{l,h\in J'_{\k}}:\, X_0\to X_0,
\end{equation} 
where
\begin{equation}\label{equ3.21}
	d_{l,h}=\begin{cases}
		(-\i)^{l+h}(-1)^ha_{\frac{l+h}{2}} Q_lvG_{l+h}vQ_h,\quad&\mbox{if}\,\, l+h\, \,\,\mbox{is even},\\
		0,  &\mbox{if}\,\, l+h\,\, \mbox{is odd},
	\end{cases}
\end{equation}
and $a_{\frac{l+h}{2}}$ is the coefficient of $R_0(-1)(x-y)$ given in \eqref{equ2.7}. 
We decompose $Q_j$ for $j\in J''_{\k}$ into two orthogonal projections:
\begin{equation}\label{eq-decom-for-Qj}
	Q_j=Q_{j,0}+Q_{j,1},	
\end{equation} 
where 
$$Q_{j,0}L^2=Q_jL^2\cap \{x^\alpha v(x):\, |\alpha|\le j\}^{\perp}.$$
Note that $Q_{j,0}=0$ and $Q_j=Q_{j,1}$ when $2m<n\le 4m-1$ by \eqref{eq:Q_j for large j-2}, and $\dim Q_{j,1}L^2 \le \#\{|\alpha|=j\}$. 
When $1\le n<2m$, denote
 $$\widetilde{\textbf{D}}_{11}=(d_{l,h}^1)_{l,h\in J''_{\k}}:\, \bigoplus_{j\in J''_{\k}} Q_{j,1}L^2\to \bigoplus_{j\in J''_{\k}}Q_{j,1}L^2,$$ 
 where
 \begin{equation*}\label{equ3.21-2}
 	d_{l,h}^1=\begin{cases}
 		(-\i)^{l+h}(-1)^ha_{\frac{l+h}{2}} Q_{l,1}vG_{l+h}vQ_{h,1},\quad&\mbox{if}\,\, l+h\, \,\,\mbox{is even},\\
 		0,  &\mbox{if}\,\, l+h\,\, \mbox{is odd},
 	\end{cases}
 \end{equation*}
 When $2m<n< 4m$, we denote
 $${\textbf{D}}_{11}=(d_{l,h})_{l,h\in J''_{\k}}:\, \bigoplus_{j\in {J}''_{\k}} Q_{j}L^2\to \bigoplus_{j\in {J}''_{\k}} Q_{j}L^2,$$ 
 where
 \begin{equation*}\label{equ3.21-2-1}
 	d_{l,h}=\begin{cases}
 		(-\i)^{l+h}(-1)^ha_{\frac{l+h}{2}} Q_{l}vG_{l+h}vQ_{h},\quad&\mbox{if}\,\, l+h\, \,\,\mbox{is even},\\
 		0,  &\mbox{if}\,\, l+h\,\, \mbox{is odd},
 	\end{cases}
 \end{equation*}

The following  lemma reveals  critical structural properties of the operator block matrices $\mathbf{D}_{00}$,  $\widetilde{\mathbf{D}}_{11}$ and   $\mathbf{D}_{11}$, which play a key role in the proof of Theorem \ref{thm3.4}. 

\begin{lemma}\label{lem3.3}
For dimensions $n < 2m$, the operator block matrix $\mathbf{D}_{00}$ is strictly positive, while $\widetilde{\mathbf{D}}_{11}$ is strictly negative. In the case $2m < n \leq 4m-1$, the operator block matrix $\mathbf{D}_{11}$ is strictly negative.
\end{lemma}
\begin{proof}
The negativity of $\widetilde{\mathbf{D}}_{11}$ and $\mathbf{D}_{11}$ follows from analogous arguments. Therefore, we focus our analysis on $\mathbf{D}_{00}$ and $\widetilde{\mathbf{D}}_{11}$.

It follows from \eqref{eq:exp for |x-y|} and \eqref{equp3.18.alpha} that for any $f\in  L^2$,
\begin{equation}\label{eq:exp for J'}
	a_{\frac{l+h}{2}}Q_lvG_{l+h}vQ_hf(x)=\sum_{|\alpha|=l, |\beta|=h}(-1)^{|\beta|}A_{\alpha,\beta}Q_l(x^\alpha v)\int_{\R^n}y^\beta v(y)Q_hf(y)\d y,\quad\mbox{if}\,\,l, h\in J'_{\k},
\end{equation}
and
\begin{equation}\label{eq:exp for J''}
	a_{\frac{l+h}{2}}Q_{l,1}vG_{l+h}vQ_{h,1}f(x)=\sum_{|\alpha|=l, |\beta|=h}(-1)^{|\beta|}A_{\alpha,\beta}Q_{l,1}(x^\alpha v)\int_{\R^n}y^\beta v(y)Q_{h,1}f(y)\d y,\quad\mbox{if}\,\,l, h\in J''_{\k}.
\end{equation}
Note that $Q_{j}L^2$ for $j\in J'_{\k} \cup J''_{\k}$ is finite dimensional by definition and the Riesz-Schauder theory.
Let $\{u_\alpha\}_{|\alpha|=l}$ be an orthonormal basis of $Q_lL^2$ for each $l \in J'_{\k}$. For $l \in J''_{\k}$, and let $\{u_\alpha\}_{\alpha \in I_l}$ be an orthonormal basis of $Q_{l,1}L^2$ where $I_l\subset\{\alpha;\,|\alpha| = l\}$ parametrizes the basis elements of $Q_{l,1}L^2$.

Now we define two  matrices
\begin{equation*}\label{equ3.30}
	E_0=\left((-\i)^{|\alpha|+|\beta|}(A_{\alpha,\beta})_{|\alpha|=l, |\beta|=h}\right)_{l,h \in J'_{\k}},\qquad E_1=\left((-\i)^{|\alpha|+|\beta|}(A_{\alpha,\beta})_{|\alpha|=l, |\beta|=h}\right)_{l,h \in J''_{\k}}.
\end{equation*}
In the following we abbreviate $E_0=((-\i)^{|\alpha|+|\beta|}A_{\alpha,\beta})_{|\alpha|,|\beta|\in J'_{\k}}$  and so on. Let 
$$\tilde{E}_{0}=\left(\tilde{A}_{\alpha,\beta}=\langle d_{|\alpha||\beta|}u_{\alpha}, u_{\beta}\rangle\right)_{|\alpha|,|\beta|\in J'_{\k}}$$
be the matrix representation of the linear operator $\textbf{D}_{00}$ under the orthogonal basis $\{u_\alpha\}_{|\alpha|\in J'_{\k}}$. Using \eqref{eq:exp for J'}, we have
\begin{align*}
	\tilde{A}_{\alpha,\beta}& =  \langle d_{|\alpha||\beta|}u_{\alpha}, u_{\beta}\rangle\\
	&=(-\i)^{|\alpha|+|\beta|}\sum_{|\alpha'|=|\alpha|, |\beta'|=|\beta|} A_{\alpha,\beta} \langle u_{\alpha}, y^{\alpha'}v\rangle \langle x^{\beta'}v, u_{\beta}\rangle\\
	&=(-\i)^{|\alpha|+|\beta|}\sum_{|\alpha'|=|\alpha|, |\beta'|=|\beta|}\Lambda_{\alpha,\alpha'}A_{\alpha',\beta'}\bar{\Lambda}_{\beta',\beta},
\end{align*}
where $\Lambda_{\alpha,\beta}=\langle Q_{|\alpha|}u_{\alpha},\, Q_{|\beta|}(x^\beta v)\rangle$.
Denoted by 
$$\Lambda_0=(\Lambda_{\alpha,\beta})_{|\alpha|,|\beta|\in J'_{\k}}=\begin{pmatrix}
	(\Lambda_{\alpha,\beta})_{|\alpha|=|\beta|=0} & 0 & \cdots & 0 \\
	0 & (\Lambda_{\alpha,\beta})_{|\alpha|=|\beta|=1} & \cdots & 0 \\
	\vdots & \vdots & \ddots & \vdots \\
	0 & 0 & \cdots & (\Lambda_{\alpha,\beta})_{|\alpha|=|\beta|=m-\frac{n+1}{2}}
\end{pmatrix},
$$ 
we have
\begin{equation*}
	\tilde{E}_{0}=\Lambda_0{E}_{0}(\overline{\Lambda}_0)^{T}, 
\end{equation*}
where $(\overline{\Lambda}_0)^{T}$ is the conjugate transpose of $\Lambda_0$.
Since $\{Q_{j}(x^\beta v)\}_{|\beta|=j}$ is also a basis on $Q_{j}L^2$ for $j\in J'_{\k}$, the diagonal block matrix $\Lambda_0$ is nonsingular. Thus matrix
$\tilde{E}_{0}=(\tilde{A}_{\alpha,\beta})_{|\alpha|,|\beta|\in J'_{\k}}$ is congruent to $E_0$ (through $\Lambda_0$). {\bf Therefore the operator block matrix $\mathbf{D}_{00}$ is strictly positive if and only if  the matrix $E_0$ is positive definite.}

Similarly, let 
$$\tilde{E}_{1}=\left(\tilde{A}_{\alpha,\beta}=\langle d_{|\alpha||\beta|}u_{\alpha}, u_{\beta}\rangle\right)_{\alpha \in I_{l},\beta\in I_{h}, \atop l,h \in J''_{\k}}$$
denote the matrix representation of the operator $\widetilde{\textbf{D}}_{11}$ under the orthogonal basis $\{u_\alpha\}_{\alpha\in I_{l}, l\in J''_{\k}}$.
Using \eqref{eq:exp for J''}, we also have
\begin{align*}
	\tilde{A}_{\alpha,\beta}& =  \langle d_{|\alpha||\beta|}u_{\alpha}, u_{\beta}\rangle\\
	&=(-\i)^{|\alpha|+|\beta|}\sum_{|\alpha'|=|\alpha|, |\beta'|=|\beta|} A_{\alpha,\beta} \langle u_{\alpha}, y^{\alpha'}v\rangle \langle x^{\beta'}v, u_{\beta}\rangle\\
	&=(-\i)^{|\alpha|+|\beta|}\sum_{|\alpha'|=|\alpha|, |\beta'|=|\beta|}\Lambda_{\alpha,\alpha'}A_{\alpha',\beta'}\bar{\Lambda}_{\beta',\beta},
\end{align*}
for $\alpha \in I_{l},\beta\in I_{h}$, where $\Lambda_{\alpha,\beta}=\langle Q_{|\alpha|,1}u_{\alpha},\, Q_{|\beta|,1}(x^\beta v)\rangle$. We claim the matrix $\Lambda_1=(\Lambda_{\alpha,\beta})_{\alpha \in I_{l}, l\in J''_{\k} \atop |\beta|\in J''_{\k}}$ has full row rank. In fact, if  $\Lambda_1$ dose
not have full row rank, then the row vectors of this matrix are linearly dependent, which is to say, there exists a sequence of numubers $z_{\alpha}$ which are not all zeros such that for each $|\beta|\in  J''_{\k}$, the following holds: 
$$0=\sum_{l\in J''_{\k}}\sum_{|\alpha|\in\alpha \in I_{l}}  z_{\alpha}\Lambda_{\alpha,\beta}=\Big \langle \sum z_{\alpha} u_{\alpha},\, x^\beta v\Big\rangle= \Big \langle \sum_{\alpha\in I_{|\beta|}} z_{\alpha} u_{\alpha},\, x^\beta v\Big\rangle.$$
One concludes that $\sum_{\alpha\in I_{|\beta|}} z_{\alpha} u_{\alpha} \in Q_{|\beta|,0}L^2$, which implies the contradiction that $z_\alpha=0$ for all $\alpha$. So we have 
$$\quad \tilde{E}_{1}=\Lambda_1{E}_{1}(\overline{\Lambda}_1)^{T}.$$
 {\bf Therefore the operator block matrix $\widetilde{\textbf{D}}_{11}$ is strictly negative if the matrix $E_1$ is negative  definite.}

It suffices to prove that $E_0$ is strictly positive definite and $E_1$ is strictly negative definite.
First, observe that the vectors
 $$\left\{\frac{\xi^\alpha}{(2\pi)^{\frac{n}{2}}\alpha!(1+|\xi|^{2m})^\frac12}\right\}_{|\alpha|\in J'_{\k}}$$
 are linearly independent. From the expression \eqref{equ2.12.000}, we note  that $E_0=((-\i)^{|\alpha|+|\beta|}A_{\alpha,\beta})_{|\alpha|,|\beta|\in J'_{\k}}$ is the Gram matrix associated with these vectors.  Since the vectors are linearly independent,  $E_0$ is strictly positive definite.


Similarly, \eqref{equ2.12.000} implies that $-E_1$ is the Gram matrix associated with the linearly independent vectors
$$\left\{\frac{\xi^\alpha}{(2\pi)^{\frac{n}{2}}\alpha!(1+|\xi|^{2m})^\frac12|\xi|^m}\right\}_{|\alpha|\in J''_{\k}}.$$
Therefore $E_1$ is strictly negative, and the proof is complete.
\end{proof}

\subsection{Proof of Theorem \ref{thm3.4}}\label{sec3.3}\

Before the proof, we first collect some orthogonal properties of  $Q_{j}$.
\begin{lemma}\label{prop3.2}
	For  $i,j\in J_{\mathbf{k}}$, we have the following orthogonal properties:
	\begin{equation}\label{equ3.14}
		Q_ivG_{2l}vQ_j=0,\quad \text{if~} \mbox{$[i+\frac{1}{2}]+[j+\frac{1}{2}]-1\ge 2l$},~~l\in\mathbb{N}_0,
	\end{equation}
	and
	\begin{equation}\label{equ3.16}
		Q_iT_0Q_j=0,\quad \text{if~} \mbox{$i+j>2m-n$}.
	\end{equation}
\end{lemma}
\begin{proof}
	We first prove \eqref{equ3.14}.
	For any $f, g\in L^2$,
	\begin{equation*}
		\left\langle Q_ivG_{2l}vQ_jf,\, g\right\rangle=\int_{\R^{2n}}|x-y|^{2l}v(x)v(y)Q_jf(x)\overline{Q_ig(y)}\d x\d y.
	\end{equation*}
	Note that
	\begin{equation*}\label{equp3.18}
		|x-y|^{2l}=\sum_{|\alpha|+|\beta|=2l}C_{\alpha\beta}x^{\alpha}(-y)^{\beta},
	\end{equation*}
	and by \eqref{CforQ_j-1}, we have 
    $$\int_{\R^n}v(x)x^{\alpha}Q_jf(x)\d x=0,\quad\mbox{ when}\quad |\alpha| \leq [j+\frac12]-1,$$ and 
    $$\int_{\R^n}v(y)y^{\beta}\overline{Q_ig(y)}\d y=0 ,\quad\mbox{ when}\quad |\beta| \leq [i+\frac12]-1.$$ 
    This implies \eqref{equ3.14}.
	
	In order to prove \eqref{equ3.16}, note that by \eqref{eq-S_j-odd-1} and \eqref{eq-S_j-odd-2} we have
	$$S_iT_0S_j=0,\quad\text{if}\,\, i+j \geq2m-n-1 \,\,\text{and}\,\, \max\{i,j\}>\mbox{$m-\frac{n}{2}$},$$
	and
	$$S_{m-\frac{n}{2}}T_0S_j=S_jT_0S_{m-\frac{n}{2}}=0,\quad  \text{if}\,\, j\geq \mbox{$m-\frac{n+1}{2}$},$$
	therefore \eqref{equ3.16} follows.
\end{proof}
We also need the following abstract Feshbach formula.
\begin{lemma}[{\cite[Lemma 2.3]{JN}}]\label{lemma-Feshbach formula}
	Let $\mathbb{A^\pm}$ be an operator matrix on the orthogaonal direct sum of Hilbert spaces $\mathcal{H}_1\oplus\mathcal{H}_2:$
	\begin{equation}
		\mathbb{A^\pm}=\begin{pmatrix}
			a_{11} & a_{12}\\
			a_{21} & a_{22}
		\end{pmatrix},\ \ \ a_{ij}:\ \mathcal{H}_j\rightarrow\mathcal{H}_i,
	\end{equation}
	where $a_{11},a_{22}$ are closed and $a_{12},a_{21}$ are bounded. Assume that $a_{11}$ has a bounded inverse, then $\mathbb{A^\pm}$ has a bounded inverse if and only if the operator
	\begin{equation}
		d:= a_{22}-a_{21}a_{11}^{-1}a_{12}
	\end{equation}
	has a bounded inverse.  Furthermore, if $d$ has a bounded inverse, we then have 
	\begin{equation}\label{eq-inverse-feshbach}
		\left(\mathbb{A^\pm} \right)^{-1}=
		\begin{pmatrix}
			a_{11}^{-1}a_{12}d^{-1}a_{21}a_{11}^{-1}+a_{11}^{-1} & -a_{11}^{-1}a_{12}d^{-1}\\
			-d^{-1}a_{21}a_{11}^{-1} & d^{-1}
		\end{pmatrix}.
	\end{equation}
\end{lemma}

We outline the \textbf{proof strategy} for Theorem \ref{thm3.4}  as follows. Under Assumption \ref{assum1}, we know from Proposition \ref{prop3.1} that $\sum\limits_{j\in J_{\k}}Q_j=I$. This leads to the definition of the isomorphism (for  $\lambda>0$):
\begin{equation}\label{eq:def fo B_lambda}
B_{\lambda} = (\lambda^{-j}Q_j)_{j\in J_{\k}} : \bigoplus_{j\in J_{\k}} Q_jL^2 \to L^2, \quad (f_j)_{j\in J_{\k}} \mapsto \sum_{j\in J_{\k}} \lambda^{-j}Q_jf_j.
\end{equation}
In matrix form, $B_\lambda$ and its adjoint $B_\lambda^*$ can be written as 
$$
B_{\lambda}= \left(Q_0, \lambda^{-1}Q_1, \ldots, \lambda^{-j_{\k}}Q_{j_{\k}} \right),   \qquad 
B_{\lambda}^*= \begin{pmatrix} Q_0 \\ \lambda^{-1}Q_1\\ \vdots \\ \lambda^{-j_{\k}}Q_{j_{\k}} \end{pmatrix},
\quad \text{if $n<2m$},
$$
and 
$$
B_{\lambda}^*= \left( \lambda^{\frac n2-m}Q_{m-\frac n2}Q_0, Q_0, \ldots, \lambda^{-j_{\k}}Q_{j_{\k}} \right),  \qquad 
B_{\lambda}= \begin{pmatrix} \lambda^{\frac n2-m}Q_{m-\frac n2} \\ Q_0\\ \vdots \\ \lambda^{-j_{\k}}Q_{j_{\k}} \end{pmatrix},
\quad \text{if $2m<n< 4m$},
$$
where $j_{\k}=\max J_{\k}$. The isomorphism of $B_{\lambda}$ yields the following equivalence:  $M^{\pm}(\lambda)$ is invertible on $L^2$ if and only if $B_{\lambda}^*M^{\pm}(\lambda)B_{\lambda}$ is invertible on $\bigoplus_{j\in J_{\k}} Q_jL^2$. In such case, the inverses are given by
\begin{equation}\label{eq-inverse-A-to-M}
	\left(M^{\pm}(\lambda)\right)^{-1}=B_{\lambda} \left(B_{\lambda}^*M^{\pm}(\lambda)B_{\lambda}\right)^{-1}B_{\lambda}^*.
\end{equation}
The key of our method lies in reformulating the  original problem of inverting $M^{\pm}(\lambda)$ as an investigation of the invertibility and asymptotic behavior of the operator block matrices $B_{\lambda}^*M^{\pm}(\lambda)B_{\lambda}$ when $\lambda$ is around zero energy.  This framework allows us to systematically analyze  $\left(M^{\pm}(\lambda)\right)^{-1}$, and our analysis is naturally divided into two distinct scenarios: the case  $1 \leq n < 2m$, and the case $2m < n < 4m$.

\subsubsection{\bf{The scenario $1 \leq n < 2m$.}}\

For a zero resonance of the $\k$-th kind, we apply \eqref{equ2.2.2} with the parameter
 $$\theta=\min\{2m-n+2\mathbf{k}+1,\, 4m-n+1\},$$ 
to yield the following expansions for $M^{\pm}(\lambda)$ as sums of bounded operators in $L^2$: \\
$\bullet$ if $\k=0$ (regular case),
\begin{equation}\label{equ3.52-0}
M^{\pm}(\lambda)=\sum\limits_{j=0}^{ m-\frac{n+1}{2}}\lambda^{n-2m+2j}a_j^\pm vG_{2j}v+T_0 
+vr_{2m-n+1}^\pm(\lambda)v;
\end{equation}
$\bullet$ if $1\le \k\le m_n-1=m-1$,
\begin{equation}\label{equ3.52-1}
M^{\pm}(\lambda)=\sum\limits_{j=0}^{ m-\frac{n+1}{2}}\lambda^{n-2m+2j}a_j^\pm vG_{2j}v+T_0+\sum\limits_{j=m-\frac{n-1}{2}}^{m-\frac{n+1}{2}+\k}\lambda^{n-2m+2j}a_j^\pm vG_{2j}v 
+vr_{2m-n+2\mathbf{k}+1}^\pm(\lambda)v;
\end{equation}
$\bullet$ if $\k= m$ or $m+1$,
\begin{equation}\label{equ3.52-2}
\begin{aligned}
	M^{\pm}(\lambda)=&\sum\limits_{j=0}^{ m-\frac{n+1}{2}}\lambda^{n-2m+2j}a_j^\pm vG_{2j}v+T_0 \\
	&+\sum_{j=m-\frac{n-1}{2}}^{2m- \frac{n+1}{2}}\lambda^{n-2m+2j}a_j^\pm vG_{2j}v+ \lambda^{2m}b_1vG_{4m-n}v+vr_{4m-n+1}^\pm(\lambda)v.	
\end{aligned}
\end{equation}
Here, $G_{j}$ and $T_0$ are defined in \eqref{eq:def of G_j} and \eqref{equ3.5}; the $\lambda$-dependent operators $r_{\theta}^{\pm}(\lambda)$ have integral kernels $r_{\theta}^{\pm}(\lambda)(x-y)$. We note that the $L^2$-boundedness of all terms in the above expansions follows from the decay assumption \eqref{eq:decay assump} combined with the Hilbert-Schmidt property.\footnote{Throughout this paper, our decay assumption \eqref{eq:decay assump} stems fundamentally from this consideration.} Moreover,  it follows from  \eqref{equ2.3}  that	
\begin{equation}\label{equ3.53}
	vr_{\theta}^\pm(\lambda)v\in \mathfrak{S}_{\frac{n+1}{2}}^{n-2m+\theta}\big((0,1)\big).
\end{equation}\

\noindent {\bf Case 1: zero is a regular point ($\k=0$).}

By Lemma \ref{prop3.2} and the expansions \eqref{equ3.52-0}, for any $i,j\in J_{0}=\{0,\cdots, m-\frac{n+1}{2}, m-\frac{n}{2}\}$, we obtain the following.
\begin{itemize}
	\item If  $i,j<m-\frac n2$, then
\begin{equation*}
\begin{split}
    &\lambda^{-(i+j)}Q_i M^{\pm}(\lambda)Q_j\\
    =&\sum\limits_{s=[\frac{i+j+1}{2}]}^{ m-\frac{n+1}{2}}\lambda^{n-2m+2s-(i+j)}a_s^\pm Q_ivG_{2s}vQ_j+
	\lambda^{-(i+j)}\left( Q_iT_0Q_j +Q_ivr_{2m-n+1}^\pm(\lambda)vQ_j\right).
\end{split}
\end{equation*} 
\item If $i< m-\frac{n}{2}$, $j=m-\frac n2$, or if $i= m-\frac{n}{2}$,  $j<m-\frac n2$, then
\begin{equation*}
    \begin{split}
        &\lambda^{-(i+j)}Q_i M^{\pm}(\lambda)Q_j\\
        =&\sum\limits_{s=[\frac{i+j+\frac 12+1}{2}]}^{ m-\frac{n+1}{2}}\lambda^{n-2m+2s-(i+j)}a_s^\pm Q_ivG_{2s}vQ_j+
	\lambda^{-(i+j)}\left( Q_iT_0Q_j +Q_ivr_{2m-n+1}^\pm(\lambda)vQ_j\right).
    \end{split}
\end{equation*} 
\item If $i=j=m-\frac n2$, then
\begin{align*}
	\lambda^{-(i+j)}Q_i M^{\pm}(\lambda)Q_j&=\lambda^{n-2m}Q_iT_0Q_j 
	+\lambda^{n-2m}Q_ivr_{2m-n+1}^\pm(\lambda)vQ_j.
\end{align*} 
\end{itemize} 
Let $$\mathbb{A^{\pm}}:=\Big(a_{i,j}^\pm(\lambda)\Big)_{i,j\in J_{0}}=\lambda^{2m-n}B_{\lambda}^*M^{\pm}(\lambda)B_{\lambda}.$$
Since $\k=0$, using the above identities, we can write $\mathbb{A^\pm}$ in the following form
\begin{equation}\label{equ3.63-0}
	\mathbb{A^{\pm}}=
		D^{\pm}+(r_{i,j}^\pm(\lambda))_{i,j\in J_{0}},
\end{equation}
where  $D^\pm=(d_{i,j}^\pm)_{i,j\in J_{0}}$ are given by
\begin{equation}\label{equ3.64-0}
	d_{i,j}^\pm=\begin{cases}
		a_{\frac{i+j}{2}}^\pm Q_ivG_{i+j}vQ_j,\quad&\mbox{if}\,\,i+j\,\,\mbox{is even},\\
		Q_iT_0Q_j,&\mbox{if}\,\,i=j=m-\frac n2,\\
		0,&\mbox{else},
	\end{cases}
\end{equation}
and $r_{i,j}^\pm(\lambda)$ satisfy
\begin{equation}\label{equ3.65.55-0}
	r_{i,j}^\pm(\lambda)\in
	\begin{cases}
		\mathfrak{S}_{\frac{n+1}{2}}^{1}\big((0,1)\big), &\text{if}\,\,i=j=m-\frac{n}{2},\\
		\mathfrak{S}_{\frac{n+1}{2}}^{\frac12}\big((0,1)\big),  &\text{else}.
	\end{cases}
\end{equation}
Note that the matrices $D^\pm=(d_{i,j}^\pm)_{i,j\in J_{0}}$ have the following block structure
$$D^\pm=\begin{pmatrix}
	\left(d_{i,j}^\pm\right)_{0\le i,j<m-\frac n2}  & 0\\[0.2cm]
	0 & Q_{m-\frac{n}{2}}T_0Q_{m-\frac{n}{2}}
\end{pmatrix} = \begin{pmatrix}
D_{00}^{\pm }  & 0\\[0.2cm]
0 & Q_{m-\frac{n}{2}}T_0Q_{m-\frac{n}{2}}
\end{pmatrix} . $$
Using perturbation arguments, we need to show that the matrices $D^\pm = (d_{i,j}^\pm)_{i,j\in J_{0}}$ are invertible on $\bigoplus_{j\in J_{0}} Q_jL^2$. Note that  by \eqref{equ2.7}, we have
\begin{equation*}
	d_{i,j}=(-\i)^{i+j}(-1)^j e^{\frac{\pm \i\pi}{2m}(n-2m+i+j)}d_{i,j}^\pm,
\end{equation*}
where $d_{i,j}$ is given in \eqref{equ3.21}, so it follows that
\begin{equation*}
	U_0^\pm D^\pm U_0^\pm U_1=
	\begin{pmatrix}
		\textbf{D}_{00}  & 0\\[0.2cm]
		0 & Q_{m-\frac{n}{2}}T_0Q_{m-\frac{n}{2}}
	\end{pmatrix} 
	=D,
\end{equation*}
where $\textbf{D}_{00}$ is defined in \eqref{eq:def of bfD_00} and
$$
U_0^\pm=\text{diag}\{e^{\pm\frac{\i\pi}{2m}(j+\frac{n}{2}-m)}(-\i)^jQ_j\}_{j\in J_{0}}, \qquad U_1=\text{diag}\{(-1)^jQ_j\}_{j\in J_{0}}.$$ 
In the above, we use the fact
$(-\i)^{2m-n}(-1)^{m-\frac n2}=1$ on the main branch.
Thus $D^\pm$ are invertible if and only if $D$ is invertible, in which case
\begin{equation}\label{eq:D to D^pm}
 \left( D^\pm\right)^{-1} =   U_1 U_0^\pm D^{-1} U_0^\pm.
\end{equation}
For the invertibility of $D$, we observe that the operator $Q_{m-\frac{n}{2}}T_0Q_{m-\frac{n}{2}}$ is invertible on $Q_{m-\frac{n}{2}}L^2$ by the Riesz-Schauder theory. Meanwhile,  Lemma \ref{lem3.3} shows that $\textbf{D}_{00}$ is invertible on $\bigoplus_{j=0}^{m-\frac{n+1}{2}} Q_jL^2$. Thus, $D$ is invertible on $\bigoplus_{j\in J_{0}} Q_jL^2$.  

Now we have proved the existence of $ \left( D^\pm\right)^{-1}$. So by \eqref{equ3.63-0} and \eqref{equ3.65.55-0}, there exists a small $\lambda_0$ such that $\mathbb{A}^{\pm}$ are invertible for $0<\lambda<\lambda_0$, and moreover, $\left( \mathbb{A}^{\pm}\right)^{-1}$ can be expressed as 
\begin{equation*}
\begin{aligned}
	(\mathbb{A^{\pm}})^{-1} &= \left(D^\pm + (r_{i,j}^\pm(\lambda))\right)^{-1} \\
	&= (D^\pm)^{-1}\sum_{l=0}^\infty\left(-\big(r_{ij}^\pm(\lambda)\big)_{i,j\in J_{0}}(D^\pm)^{-1}\right)^l \\
	&= (D^\pm)^{-1} + \Gamma^{\pm}(\lambda).	
\end{aligned}
\end{equation*} 
According to the identity \eqref{eq-inverse-A-to-M}, one can obtain 
\begin{equation}\label{eq-inverse-A-to-M-1}
	\left(M^{\pm}(\lambda)\right)^{-1}=\lambda^{2m-n}B_{\lambda}\left(\mathbb{A}^{\pm}\right)^{-1} B_{\lambda}^*,
\end{equation}
which implies the results in \eqref{equ3.48}--\eqref{equ3.47.1}, if we define 
 $$(D^\pm)^{-1}=\big(M_{i,j}^{\pm}\big)_{i,j\in J_{0}},\quad \quad \Gamma^{\pm}(\lambda)=\big(\Gamma_{i,j}^{\pm}(\lambda)\big)_{i,j\in J_{0}}. $$
We remark that the estimate 
\[
\Gamma^{\pm}_{m-\frac{n}{2},m-\frac{n}{2}} \in \mathfrak{S}_{\frac{n+1}{2}}^{1}\big((0,\lambda_0)\big)
\]
follows from two facts: 
\begin{enumerate}
	\item $r^{\pm}_{2m-\frac{n}{2},2m-\frac{n}{2}}(\lambda) \in \mathfrak{S}_{\frac{n+1}{2}}^{1}\big((0,1)\big)$.
	\item All elements of the series 
	\[
	(D^\pm)^{-1}\sum_{l=2}^\infty \left(-\big(r_{ij}^\pm(\lambda)\big)_{i,j\in J_{0}}(D^\pm)^{-1}\right)^l
	\] 
	belong to $\mathfrak{S}_{\frac{n+1}{2}}^{1}\big((0,\lambda_0)\big)$.
\end{enumerate}
This completes the proof for the case $\k = 0$ when $1 \leq n < 2m - 1$. \\

\noindent {\bf Case 2: zero is the $\k$-kind of resonance with $1\le \k\le m_n=m$.}\

Denote 
$$a_{i,j}^\pm(\lambda)=\lambda^{2m-n-i-j}Q_iM^{\pm}(\lambda)Q_j,\qquad \mathbb{A^{\pm}}=\Big(a_{i,j}^\pm(\lambda)\Big)_{i,j\in J_{\k}}=\lambda^{2m-n}B_{\lambda}^*M^{\pm}(\lambda)B_{\lambda}.$$
In this case, for any $i,j\in J_{\k}$ where 
$$J_{\k}=\{0,\cdots,\  m-\frac{n+1}{2},\   m-\frac{n}{2},\  m-\frac{n-1}{2}, \ \cdots,\    m-\frac{n-1}{2}+\k-1\},$$ 
by Lemma \ref{prop3.2} and the expansions \eqref{equ3.52-1} and \eqref{equ3.52-2},   we obtain when $1\le \k<m_n=m$ that
\begin{equation*}\label{equ3.53.111-1}
	\begin{split}
		a_{i,j}^\pm(\lambda)&=\sum\limits_{\tau_{i,j}\le l\le \frac{\theta_0-1}{2}}a_j^\pm\lambda^{2l-i-j}Q_ivG_{2l}vQ_j+\lambda^{2m-n-i-j}Q_{i}T_0Q_j+\lambda^{2m-n-i-j}Q_{i}vr_{\theta_0}^\pm(\lambda)vQ_{j},
	\end{split}
\end{equation*}
with $\theta_0=2m-n+2\k+1$, and when $\k=m_n=m$ that
\begin{equation*}\label{equ3.53.111-2}
	\begin{split}
	a_{i,j}^\pm(\lambda)=&\sum\limits_{\tau_{i,j}\le l\le \frac{\theta_0-1}{2}}a_j^\pm\lambda^{2l-i-j}Q_ivG_{2l}vQ_j+\lambda^{2m-n-i-j}Q_{i}T_0Q_j
	    \\
		&+b_1\lambda^{4m-n-i-j}Q_{i}vG_{4m-n}vQ_{j}+\lambda^{2m-n-i-j}Q_{i}vr_{4m-n+1}^\pm(\lambda)vQ_{j},
	\end{split}
\end{equation*}
where $\tau_{i,j}=[\frac{[i+\frac12]+[j+\frac12]+1}{2}]$.
 Note by \eqref{equ3.16} that
\begin{equation*}
	Q_iT_0Q_j=0,\quad \text{if~} \mbox{$i+j>2m-n$},
\end{equation*} 
as in Case 1, we can write $\mathbb{A^\pm}$ in the following form
\begin{equation*}\label{equ3.63-1}
		\mathbb{A^\pm}=D^{\pm}+(r_{i,j}^\pm(\lambda))_{i,j\in J_{\k}},
\end{equation*}
where  $D^\pm=(d_{i,j}^\pm)_{i,j\in J_{\k}}$ are given by
\begin{equation}\label{equ3.64}
	d_{i,j}^\pm=\begin{cases}
		a_{\frac{i+j}{2}}^\pm Q_ivG_{i+j}vQ_j,\quad&\mbox{if $i,j\in J_{\k}\setminus\{m-\frac n2\}$ and $i+j$ is even},\\
		Q_iT_0Q_j,&\mbox{if}\,\,i+j=2m-n,\\
		0,&\mbox{else},
	\end{cases}
\end{equation}
and by \eqref{equ3.53}, the operators $r_{i,j}^\pm(\lambda)$ satisfy
\begin{equation}\label{equ3.65.55}
	r_{i,j}^\pm(\lambda)\in
	\begin{cases}
	\mathfrak{S}_{\frac{n+1}{2}}^{1}\big((0,1)\big), &\text{if}\,\,i=j=m-\frac{n}{2},\\
\mathfrak{S}_{\frac{n+1}{2}}^{\frac12}\big((0,1)\big),  &\text{else}.
	\end{cases}
\end{equation}
By definition, the operator matrices $D^\pm=(d_{i,j}^\pm)_{i,j\in J_{0}}$ have the following block structure
\begin{equation}\label{eq:stru of D-0}
D^\pm=\begin{pmatrix}
	\left(d_{i,j}^\pm\right)_{i,j\in J'_{\k}}  & 0 & \left(d_{i,j}^\pm\right)_{i\in J'_{\k},~ j\in J''_{\k}} \\[0.2cm]
	0 & Q_{m-\frac{n}{2}}T_0Q_{m-\frac{n}{2}} &0 \\[0.2cm]
	\left(d_{i,j}^\pm\right)_{i\in J''_{\k},~ j\in J'_{\k}} & 0 & \left(d_{i,j}^\pm\right)_{i,j\in J''_{\k}}
\end{pmatrix}: = \begin{pmatrix}
	D_{00}^{\pm }  &0  &D_{01}^{\pm }\\[0.2cm]
	0 & Q_{m-\frac{n}{2}}T_0Q_{m-\frac{n}{2}} & 0\\[0.2cm]
	D_{10}^{\pm } & 0 & D_{11}^{\pm }
\end{pmatrix} , 
\end{equation}
where $J'_{\k}$ and $J''_{\k}$ are defined in \eqref{equ3.12.11}. By \eqref{equ2.7}, we have
\begin{equation}\label{equ3.58}
\begin{split}
d_{i,j}=(-\i)^{i+j}(-1)^j e^{\frac{\pm \i\pi}{2m}(n-2m+i+j)}d_{i,j}^\pm, 
\end{split}
\end{equation}
where $$d_{i,j}=
\begin{cases}
	(-\i)^{i+j}(-1)^j a_{\frac{i+j}{2}} Q_ivG_{i+j}vQ_j,\quad&\mbox{if $i,j\in J_{\k}\setminus\{m-\frac n2\}$ and $i+j$ is even},\\
	Q_iT_0Q_j,&\mbox{if}\,\,i+j=2m-n,\\
	0,&\mbox{else},
\end{cases}	$$
so it follows  that
\begin{equation}\label{equ3.64'}
	U_0^\pm D^\pm U_0^\pm U_1=D,
\end{equation}
where $U_0^\pm=\text{diag}\{e^{\pm\frac{\i\pi}{2m}(j+\frac{n}{2}-m)}(-\i)^jQ_j\}_{j\in J_{\k}}$, $U_1=\text{diag}\{(-1)^jQ_j\}_{j\in J_{\k}}$, and 
\begin{equation}\label{eq-matrix-2}
	D=\begin{pmatrix}
		\textbf{D}_{00} & 0 & D_{01}\\[0.2cm]
		0 & Q_{m-\frac{n}{2}}T_0Q_{m-\frac{n}{2}} & 0\\[0.2cm]
		D_{10} & 0 &\textbf{D}_{11}
	\end{pmatrix}.
\end{equation}
Here,  $\textbf{D}_{00}$ is defined in \eqref{eq:def of bfD_00} and  
$$D_{01}=\left(d_{i,j}\right)_{i\in J'_{\k},~ j\in J''_{\k}},\quad D_{10}=\left(d_{i,j}\right)_{i\in J''_{\k},~ j\in J'_{\k}},\quad \textbf{D}_{11}=\left(d_{i,j}\right)_{i\in J''_{\k},~ j\in J''_{\k}}.  $$
Following the approach in Case 1, the identity \eqref{eq:D to D^pm} remains valid, and we consequently investigate the invertibility of $D$.

Since the operator $Q_{m-\frac{n}{2}}T_0Q_{m-\frac{n}{2}}$ is invertible, we consider the operator block matrix 
$$ \begin{pmatrix}
	\textbf{D}_{00}  & D_{01}\\[0.2cm]
	D_{10} &\textbf{D}_{11} 
\end{pmatrix} \quad \text{on $\bigoplus_{j\in J_{\k}\setminus \{m-\frac n2\}} Q_jL^2$} . $$
Note that $\textbf{D}_{00}$ is invertible on $\bigoplus_{j\in J'_{\k}} Q_jL^2$ by  Lemma \ref{lem3.3}.
According to the abstract Fehsbach formula in Lemma~\ref{lemma-Feshbach formula}, $D$ is invertible on $\bigoplus_{j\in J_{\k}}Q_jL^2$ if and only if
\begin{equation*}\label{equ3.66}
	\mathbf{d}:=\mathbf{D}_{11}-D_{10} \textbf{D}_{00}^{-1}D_{01}=\mathbf{D}_{11}-D_{01}^* \textbf{D}_{00}^{-1}D_{01},
\end{equation*}
is invertible on $\bigoplus_{j\in J_\k''}Q_{j}L^2$, where we used the fact $D_{01}^*=D_{10}$.

To prove the invertibility of $\mathbf{d}$, it suffices to prove that $\mathbf{d}$ is injective on $\bigoplus_{j\in J_\k''}Q_{j}L^2$. As in \eqref{eq-decom-for-Qj}, we decompose
$$Q_j=Q_{j,0}+Q_{j,1},$$
where
\begin{equation*}
	Q_{j,0}L^2=Q_jL^2\bigcap \{x^\alpha v;\,\,|\alpha|\le j\}^\perp.
\end{equation*}
Since $Q_i v G_{i+j} v Q_{j,0} = Q_{i,0} v G_{i+j} v Q_j = 0$ for all $i \in J_\k'$ and $j \in J_\k''$ with $i+j$ even,
we derive
\begin{equation}\label{equ3.66'}
	\mathbf{d} = Q \mathbf{D}_{11} Q - D_{01}^* \mathbf{D}_{00}^{-1} D_{01}
	= Q' \mathbf{D}_{11} Q' - D_{01}^* \mathbf{D}_{00}^{-1} D_{01},
\end{equation}
where $Q'=\text{diag}\{Q_{j,1}\}_{j\in J_\k''}$.
 Assume that $f=(f_j)_{j\in J_\k''}\in  \bigoplus_{j\in J_\k''}Q_{j}L^2$ and
\begin{equation}\label{equ3.68-1}
	0=\langle \mathbf{d}f,f\rangle=\langle Q'\mathbf{D}_{11}Q'f,f\rangle-\langle D_{01}^\ast \mathbf{D}_{00}^{-1}D_{01}f,f\rangle.
\end{equation}
By Lemma \ref{lem3.3}, one has that $Q'\mathbf{D}_{11}Q'$ is strictly negative on $\bigoplus_{j\in J_\k''}Q_{j,1}L^2$ and $\mathbf{D}_{00}^{-1}$ is strictly positive $\bigoplus_{j\in J_\k'}Q_{j}L^2$. This fact,  together with \eqref{equ3.68-1}, shows that 
\begin{equation*}
	Q'f=0\quad\mbox{ and }\quad D_{01}f=0.
\end{equation*}
The fact $Q'f=0$ implies $f_j\in Q_{j,0}L^2$, i.e.,  $f_j\in \{x^\alpha v;\,\,|\alpha|\le j\}^\perp$, $j\in J_\k''$. Thus
\begin{equation*}
	D_{01}f=\begin{pmatrix}
		(-\i)^{2m-n}(-1)^{m-\frac{n+1}{2}+\k}Q_{m-\frac{n-1}{2}-\k}T_0Q_{m-\frac{n+1}{2}+\k}f_{m-\frac{n+1}{2}+\k}\\[0.2cm]
		\vdots\\[0.2cm]
		(-\i)^{2m-n}(-1)^{m-\frac{n-1}{2}}Q_{m-\frac{n+1}{2}}T_0Q_{m-\frac{n-1}{2}}f_{m-\frac{n-1}{2}}
	\end{pmatrix},
\end{equation*}
which implies
\begin{align*}\label{equ.Qj}
	Q_{2m-n-j}T_0f_{j}=0.
\end{align*}
Combining this with the definition of $Q_{j}$ for $j \in J''_{\k}$ in \eqref{eq-Q_j-odd}, we conclude that 
$$ f_j \in S_{j}L^2, $$
which yields $f_j=0$ for all $j \in J''_{\k}$. Consequently, the operator $\mathbf{d}$ is injective and therefore invertible.

We have shown that $D$ is invertible, and consequently, so are $D^{\pm}$. Using the Neumann series expansion, there exists a small $\lambda_0$ such that for any $\lambda\in (0,\, \lambda_0)$, 
\begin{equation}\label{eq:inverse of A-for k}
\begin{aligned}
(\mathbb{A^{\pm}})^{-1} &= \left(D^\pm + (r_{i,j}^\pm(\lambda))\right)^{-1} \\
&= (D^\pm)^{-1}\sum_{l=0}^\infty\left(-\big(r_{ij}^\pm(\lambda)\big)_{i,j\in J_{\k}}(D^\pm)^{-1}\right)^l \\
&= (D^\pm)^{-1} + \Gamma^{\pm}(\lambda).	
\end{aligned}
\end{equation}
Applying \eqref{eq:inverse of A-for k} together with the estimates \eqref{equ3.65.55} and the identity \eqref{eq-inverse-A-to-M-1}, we obtain all results in Theorem~\ref{thm3.4} for the case $1 \leq \k \leq m_n$ in dimensions $n < 2m$, if we define
$$(D^\pm)^{-1}=\big(M_{i,j}^{\pm}\big)_{i,j\in J_{\k}},\quad \quad \Gamma^{\pm}(\lambda)=\big(\Gamma_{i,j}^{\pm}(\lambda)\big)_{i,j\in J_{\k}}. $$ \\

\noindent {\bf Case 3: zero is an eigenvalue ($\k= m_n+1=m+1$).}\

We also denote 
$$a_{i,j}^\pm(\lambda)=\lambda^{2m-n-i-j}Q_iM^{\pm}(\lambda)Q_j,\qquad \mathbb{A^{\pm}}=\Big(a_{i,j}^\pm(\lambda)\Big)_{i,j\in J_{m+1}}=\lambda^{2m-n}B_{\lambda}^*M^{\pm}(\lambda)B_{\lambda}.$$
In this case, by Lemma \ref{prop3.2} and the expansions \eqref{equ3.52-2},  for any $i,j\in J_{m+1}$, we obtain that 
\begin{equation*}\label{equ3.53.111-3}
	\begin{split}
		a_{i,j}^\pm(\lambda)=&\sum\limits_{\tau_{i,j}\le l\le \frac{\theta_0-1}{2}}a_j^\pm\lambda^{2l-i-j}Q_ivG_{2l}vQ_j+\lambda^{2m-n-i-j}Q_{i}T_0Q_j
		\\
		&+b_1\lambda^{4m-n-i-j}Q_{i}vG_{4m-n}vQ_{j}+\lambda^{2m-n-i-j}Q_{i}vr_{\theta_0}^\pm(\lambda)vQ_{j},
	\end{split}
\end{equation*}
where $\tau_{i,j}=[\frac{[i+\frac12]+[j+\frac12]+1}{2}]$ and $\theta_0=4m-n+1$. As in Case 2, we can write $\mathbb{A^\pm}$ in the following form
\begin{equation*}\label{equ3.63-1-0}
	\mathbb{A^\pm}=D^{\pm}+(r_{i,j}^\pm(\lambda))_{i,j\in J_{\k}},
\end{equation*}
where  $D^\pm=(d_{i,j}^\pm)_{i,j\in J_{m+1}}$ are given by
\begin{equation}\label{equ3.64-3}
	d_{i,j}^\pm=\begin{cases}
		a_{\frac{i+j}{2}}^\pm Q_ivG_{i+j}vQ_j,\quad&\mbox{if $i,j\in J_{m+1}\setminus\{m-\frac n2, 2m-\frac n2\}$ and $i+j$ is even},\\
		b_1 Q_ivG_{4m-n}vQ_j,\quad&\mbox{if}\,i=j=2m-\frac n2,\\
		Q_iT_0Q_j,&\mbox{if}\,\,i+j=2m-n,\\
		0,&\mbox{else},
	\end{cases}
\end{equation}
and by \eqref{equ3.53}, the operators $r_{i,j}^\pm(\lambda)$ satisfy
\begin{equation}\label{equ3.65.55-3}
	r_{i,j}^\pm(\lambda)\in
	\begin{cases}
		\mathfrak{S}_{\frac{n+1}{2}}^{1}\big((0,1)\big), &\text{if}\,\,i=j=m-\frac{n}{2},\\
		\mathfrak{S}_{\frac{n+1}{2}}^{1}\big((0,1)\big),   &\text{if}\,\,i=j=2m-\frac{n}{2},\\
		\mathfrak{S}_{\frac{n+1}{2}}^{\frac12}\big((0,1)\big),  &\text{else}.
	\end{cases}
\end{equation}

When zero is an eigenvalue, the block operator matrices $D^\pm$ have the following structure:
\begin{equation}\label{eq-matrix-1}
	D^\pm = \begin{pmatrix}
		D_{m}^\pm & 0 \\[0.2cm]
		0 & b_1Q_{2m-\frac{n}{2}}vG_{4m-n}vQ_{2m-\frac{n}{2}}
	\end{pmatrix},
\end{equation}
where $D_{m}^\pm$ are precisely the matrices $D^\pm$ appearing in Case 2 when $\k=m$. We have obtained that $D_{m}^\pm$ are invertible in Case 2. We next prove the invertibility of $D^\pm$ by showing that $b_1Q_{2m-\frac{n}{2}}vG_{4m-n}vQ_{2m-\frac{n}{2}}$ is invertible on $Q_{2m-\frac{n}{2}}L^2$. 
By the expansion \eqref{equ2.2}, for any $\psi\in Q_{2m-\frac{n}{2}}L^2$, we have
\begin{equation}\label{eq:invertility of b_1 term}
\begin{aligned}
\Big\langle b_1Q_{2m-\frac{n}{2}}vG_{4m-n}vQ_{2m-\frac{n}{2}} \psi,\, \psi \Big\rangle
&=\lim_{\lambda\downarrow 0} \lambda^{-2m} \Big\langle v\left( R(-\lambda^{2m})-b_0G_{2m-n}\right)v  Q_{2m-\frac{n}{2}}\psi,\, Q_{2m-\frac{n}{2}}\psi \Big\rangle \\
&=\lim_{\lambda\downarrow 0} \lambda^{-2m} \Big\langle v\left[\left( (-\Delta)^{m}+\lambda^{2m}\right)^{-1} -(-\Delta)^{-m}\right]v  Q_{2m-\frac{n}{2}}\psi,\, Q_{2m-\frac{n}{2}}\psi \Big\rangle \\
&=\left\|(-\Delta)^{-m} v\psi\right\|_{L^2}^2. 
\end{aligned}
\end{equation} 
Thus, $b_1Q_{2m-\frac{n}{2}}vG_{4m-n}vQ_{2m-\frac{n}{2}} \psi=0$ implies $(-\Delta)^{-m} v\psi=0$ and
 $$\psi=-vb_0G_{2m-n}v\psi=-v(-\Delta)^{-m} v\psi=0.$$
This yields the invertbility of $b_1Q_{2m-\frac{n}{2}}vG_{4m-n}vQ_{2m-\frac{n}{2}}$ on $Q_{2m-\frac{n}{2}}L^2$, and therefore, $D^\pm$ are invertible.

We can now apply the Neumann series expansion \eqref{eq:inverse of A-for k} with $\k=m+1$, \eqref{equ3.65.55-3}, and the identity \eqref{eq-inverse-A-to-M-1} to establish all results for $\k=m+1$ when $\lambda\in (0, \lambda_0)$ with a sufficiently small $\lambda_0>0$.

This completes the proof for $1\le n<2m$.

\subsubsection{\bf{The scenario where $2m< n < 4m$.}}\

In these dimensions, when zero is the  $\k$-th kind of resonance, by using \eqref{equ2.2.2}
$M^{\pm}(\lambda)$ admits the following expansions as bounded operators in $L^2$.
\begin{itemize}
    \item If $\k=0$, then
\begin{equation}\label{equ3.52-0-1}
	M^{\pm}(\lambda)=T_0 
	+\lambda^{n-2m}a_0^\pm vG_{0}v+vr_{2}^\pm(\lambda)v.
\end{equation}

\item If $1\le \k\le m_n-1=2m-\frac{n+1}{2}$, then
\begin{equation}\label{equ3.52-1-1}
	M^{\pm}(\lambda)=T_0+\sum\limits_{j=0}^{\k-1}\lambda^{n-2m+2j}a_j^\pm vG_{2j}v 
	+vr_{2\mathbf{k}}^\pm(\lambda)v.
\end{equation}

\item If $\k= m$ or $m+1$, then
\begin{equation}\label{equ3.52-2-1}
	\begin{aligned}
		M^{\pm}(\lambda)=T_0+ \sum_{j=0}^{2m- \frac{n+1}{2}}\lambda^{n-2m+2j}a_j^\pm vG_{2j}v+ \lambda^{2m}b_1vG_{4m-n}v+vr_{4m-n+1}^\pm(\lambda)v.	
	\end{aligned}
\end{equation}
\end{itemize}
Here, the $\lambda$-dependent operators $r_{\theta}^{\pm}(\lambda)$ have integral kernels $r_{\theta}^{\pm}(\lambda)(x-y)$, and all terms in the expansions \eqref{equ3.52-0-1}--\eqref{equ3.52-2-1} are $L^2$ bounded by the decay assumption \eqref{eq:decay assump}. Moreover, it follows from \eqref{equ2.3} that $r_{\theta}^{\pm}(\lambda)$ satisfy \eqref{equ3.53}.

With the identity \eqref{eq-inverse-A-to-M}, we need to study the invertibility of $B_{\lambda}^*M^{\pm}(\lambda)B_{\lambda}$ and the asymptotic behavior of their inverses as $\lambda\to 0$. The proof is divided into three cases according to the zero energy resonance kinds. \\
	
\noindent {\bf Case 1: zero is a regular point ($\k=0$).}\

In this case, note that $J_0=\{m-\frac n2\}$ by definition of $J_{0}$ in \eqref{equ3.12-1}, and $Q_{m-\frac n2}=I$, we then have by \eqref{equ3.52-0-1} that 
$$B_{\lambda}^*M^{\pm}(\lambda)B_{\lambda}=\lambda^{n-2m}\left(T_0 
+\lambda^{n-2m}a_0^\pm vG_{0}v+vr_{2}^\pm(\lambda)v \right).  $$
It is clear that $T_0 =Q_{m-\frac n2}T_0Q_{m-\frac n2}$ is invertible on $L^2$ by the Riesz-Schauder theory. Thus, by the Neumann series expansion, there exists a small $\lambda_0$ such that Theorem \ref{thm3.4} is valid for $\k=0$. We remark that the estimate \eqref{equ3.47.1} follows from
$$\lambda^{n-2m}a_0^\pm vG_{0}v+vr_{2}^\pm(\lambda)v \in \mathfrak{S}_{\frac{n+1}{2}}^{n-2m}\big((0,1)\big).$$\\
	
\noindent {\bf Case 2: zero is the $\k$-kind of resonance with $1\le \k\le m_n=2m-\frac{n-1}{2}$.}\

Denote 
$$a_{i,j}^\pm(\lambda)=\lambda^{2m-n-i-j}Q_iM^{\pm}(\lambda)Q_j,\qquad \mathbb{A^{\pm}}=\Big(a_{i,j}^\pm(\lambda)\Big)_{i,j\in J_{\k}}=\lambda^{2m-n}B_{\lambda}^*M^{\pm}(\lambda)B_{\lambda}.$$
In this case, by Lemma \ref{prop3.2} and the expansions \eqref{equ3.52-1-1} and \eqref{equ3.52-2-1},  for any $i,j\in J_{\k}$, we obtain that $a_{i,j}^\pm(\lambda)$ can be written as follows: when $1\le \k<m_n=2m-\frac{n-1}{2}$,
\begin{equation*}
	\begin{split}
a_{i,j}^\pm(\lambda)=\lambda^{2m-n-i-j}Q_{i}T_0Q_j+\sum\limits_{\tau_{i,j}\le l\le\k-1}
a_j^\pm\lambda^{2l-i-j}Q_ivG_{2l}vQ_j
+\lambda^{2m-n-i-j}Q_{i}vr_{2\k}^\pm(\lambda)vQ_{j}, 
	\end{split}
\end{equation*}
and when $\k=m_n=2m-\frac{n-1}{2}$,
\begin{equation*}
	\begin{split}
		a_{i,j}^\pm(\lambda)=&\lambda^{2m-n-i-j}Q_{i}T_0Q_j+\sum\limits_{\tau_{i,j}\le l\le 2m-\frac{n+1}{2}}a_j^\pm\lambda^{2l-i-j}Q_ivG_{2l}vQ_j
		\\
		&+b_1\lambda^{4m-n-i-j}Q_{i}vG_{4m-n}vQ_{j}+\lambda^{2m-n-i-j}Q_{i}vr_{4m-n+1}^\pm(\lambda)vQ_{j},
	\end{split}
\end{equation*}
where 
\begin{equation}\label{eq:def for tau_ij}
	\tau_{i,j}=[\frac{\delta(i)+\delta(j)+1}{2}], \qquad \delta(i)=\max\{0, [i+ 1/2]\}. 
\end{equation}
Note that 
$$Q_{i}T_0Q_j=0,\quad\text{if}\,\,(i,j)\neq (\mbox{$m-\frac n2, m-\frac n2$}),$$
as in the case $n<2m$, we can write $\mathbb{A^\pm}$ in the following form
\begin{equation*}
	\mathbb{A^\pm}=D^{\pm}+(r_{i,j}^\pm(\lambda))_{i,j\in J_{\k}},
\end{equation*}
where  $D^\pm=(d_{i,j}^\pm)_{i,j\in J_{\k}}$ are given by
\begin{equation}\label{equ3.64-1}
	d_{i,j}^\pm=\begin{cases}
		a_{\frac{i+j}{2}}^\pm Q_ivG_{i+j}vQ_j,\quad&\mbox{if $i,j\in J_{\k}\setminus\{m-\frac n2\}$ and $i+j$ is even},\\
		Q_iT_0Q_j,&\mbox{if $i=j=m-\frac n2$},\\
		0,&\mbox{else},
	\end{cases}
\end{equation}
and by \eqref{equ3.53}, the operators $r_{i,j}^\pm(\lambda)$ satisfy
\begin{equation*}
	r_{i,j}^\pm(\lambda)\in
	\begin{cases}
		\mathfrak{S}_{\frac{n+1}{2}}^{2m-n}\big((0,1)\big), &\text{if $i=j=m-\frac n2$},\\
		\mathfrak{S}_{\frac{n+1}{2}}^{m-\frac n2}\big((0,1)\big), &\text{if $i=m-\frac n2$ with $j\neq m-\frac n2$, or $i\neq m-\frac n2$ with $j= m-\frac n2$},\\
		\mathfrak{S}_{\frac{n+1}{2}}^{\frac12}\big((0,1)\big),  &\text{else}.
	\end{cases}
\end{equation*}
By definition, the operator matrices $D^\pm=(d_{i,j}^\pm)_{i,j\in J_{0}}$ have the following block structure:
$$D^\pm=\begin{pmatrix}
	 Q_{m-\frac{n}{2}}T_0Q_{m-\frac{n}{2}} &0 \\[0.2cm]
	 0 & \left(d_{i,j}^\pm\right)_{i,j\in J''_{\k}}
\end{pmatrix} = \begin{pmatrix}
 Q_{m-\frac{n}{2}}T_0Q_{m-\frac{n}{2}} & 0\\[0.2cm]
 0 & D_{11}^{\pm }
\end{pmatrix} , $$
where  $J''_{\k}$ is defined in \eqref{equ3.12.11}. It follows from \eqref{equ2.7} that
\begin{equation*}
		(-\i)^{i+j}(-1)^j a_{\frac{i+j}{2}} Q_ivG_{i+j}vQ_j=(-\i)^{i+j}(-1)^j e^{\frac{\pm \i\pi}{2m}(n-2m+i+j)}d_{i,j}^\pm :=d_{i,j}.
\end{equation*}
Let  $U_0^\pm=\text{diag}\{e^{\pm\frac{\i\pi}{2m}(j+\frac{n}{2}-m)}(-\i)^jQ_j\}_{j\in J_{\k}}$, $U_1=\text{diag}\{(-1)^jQ_j\}_{j\in J_{\k}}$.  It follows  that
\begin{equation*}
	U_0^\pm D^\pm U_0^\pm U_1=D,
\end{equation*}
where 
\begin{equation}\label{eq-matrix-2-1}
	D=\begin{pmatrix}
	    Q_{m-\frac{n}{2}}T_0Q_{m-\frac{n}{2}} & 0\\[0.2cm]
		0 &{\textbf{D}}_{11}
	\end{pmatrix}.
\end{equation}
Here,  $\textbf{D}_{11}=\left(d_{i,j}\right)_{i,j\in J''_{\k}}$ is defined in \eqref{equ3.21-2-1}. 
Following the approach in dimensions $n<2m$, the identity \eqref{eq:D to D^pm} remains valid, and we need to study the invertibility of $D$. The operator $Q_{m-\frac{n}{2}}T_0Q_{m-\frac{n}{2}}$ is obviously invertible and ${\textbf{D}}_{11}$ is invertible on $\bigoplus_{j\in {J}''_{\k}} Q_{j}L^2$ by Lemma \ref{lem3.3}, which implies that $D$ is invertible. Now, we can apply the identity \eqref{eq-inverse-A-to-M-1} and  the Neumann series expansion \eqref{eq:inverse of A-for k} to establish all results for $1\le \k\le m_n$ when $\lambda\in (0, \lambda_0)$ with a sufficiently small $\lambda_0>0$. \\

\noindent {\bf Case 3: zero is an eigenvalue ($\k= m_n+1=2m-\frac{n-1}{2}+1$).}\

We also denote 
$$a_{i,j}^\pm(\lambda)=\lambda^{2m-n-i-j}Q_iM^{\pm}(\lambda)Q_j,\qquad \mathbb{A^{\pm}}=\Big(a_{i,j}^\pm(\lambda)\Big)_{i,j\in J_{m+1}}=\lambda^{2m-n}B_{\lambda}^*M^{\pm}(\lambda)B_{\lambda}.$$
In this case, by Lemma \ref{prop3.2} and the expansion \eqref{equ3.52-2-1},  for any $i,j\in J_{m_n+1}$, we obtain
\begin{equation*}
	\begin{split}
		a_{i,j}^\pm(\lambda)=&\lambda^{2m-n-i-j}Q_{i}T_0Q_j+\sum\limits_{\tau_{i,j}\le l\le 2m-\frac{n+1}{2}}a_j^\pm\lambda^{2l-i-j}Q_ivG_{2l}vQ_j
		\\
		&+b_1\lambda^{4m-n-i-j}Q_{i}vG_{4m-n}vQ_{j}+\lambda^{2m-n-i-j}Q_{i}vr_{4m-n+1}^\pm(\lambda)vQ_{j}.
	\end{split}
\end{equation*}
where $\tau_{i,j}$ is given in \eqref{eq:def for tau_ij}. Then $\mathbb{A^\pm}$ can be written as the following form
\begin{equation*}
	\mathbb{A^\pm}=D^{\pm}+(r_{i,j}^\pm(\lambda))_{i,j\in J_{m_n+1}},
\end{equation*}
where  $D^\pm=(d_{i,j}^\pm)_{i,j\in J_{m_n+1}}$ have the structure
\begin{equation*}
	D^\pm = \begin{pmatrix}
		D_{m_n}^\pm & 0 \\[0.2cm]
		0 & b_1Q_{2m-\frac{n}{2}}vG_{4m-n}vQ_{2m-\frac{n}{2}}
	\end{pmatrix},
\end{equation*}
and $D_{m_n}^\pm$ are precisely the matrices $D^\pm$ appearing in Case 2 when $\k=m_n$.
Moreover, by \eqref{equ3.53}, the operators $r_{i,j}^\pm(\lambda)$ satisfy
\begin{equation}\label{equ3.65.55-1}
	r_{i,j}^\pm(\lambda)\in
	\begin{cases}
		\mathfrak{S}_{\frac{n+1}{2}}^{2m-n}\big((0,1)\big), &\text{if $i=j=m-\frac n2$},\\
		\mathfrak{S}_{\frac{n+1}{2}}^{1}\big((0,1)\big), &\text{if $i=j=2m-\frac n2$},\\
		\mathfrak{S}_{\frac{n+1}{2}}^{m-\frac n2}\big((0,1)\big), &\text{if $i=m-\frac n2$ with $j\neq m-\frac n2$, or $i\neq m-\frac n2$ with $j= m-\frac n2$},\\
		\mathfrak{S}_{\frac{n+1}{2}}^{\frac12}\big((0,1)\big),  &\text{else}.
	\end{cases}
\end{equation}
Since $D_{m_n}^\pm$ are invertible by Case 2 and $b_1Q_{2m-\frac{n}{2}}vG_{4m-n}vQ_{2m-\frac{n}{2}}$ is invertible on $Q_{2m-\frac{n}{2}}L^2$  by \eqref{eq:invertility of b_1 term}, one has the invertibility of  $D^\pm$ on $\bigoplus_{j\in {J}_{m_n+1}} Q_{j}L^2$. Then applying the identity \eqref{eq-inverse-A-to-M-1}, we establish all the results in this theorem. 
The estimates
$$\Gamma^{\pm}_{m-\frac n2,m-\frac n2}\in \mathfrak{S}_{\frac{n+1}{2}}^{1}\big((0,\lambda_0)\big) $$
are given by the facts that $r^{\pm}_{2m-\frac n2,2m-\frac n2}(\lambda)\in \mathfrak{S}_{\frac{n+1}{2}}^{1}\big((0,1)\big)$ and all the elements of $$(D^\pm)^{-1}\sum\limits_{l=2}^\infty\left(-\big(r_{ij}^\pm(\lambda)\big)_{i,j\in J_{m_n+1}}(D^\pm)^{-1}\right )^l$$ 
are in $\mathfrak{S}_{\frac{n+1}{2}}^{1}\big((0,\lambda_0)\big)$.

The proof of Theorem \ref{thm3.4} is complete.


\section{Proof of Theorem \ref{theorem4.0} (low energy part)}\label{section-4}
In this section, let $\lambda_{0}\in (0, 1)$ be given by Theorem \ref{thm3.4}.
By  resolvent identity and \eqref{equ0.2}, we have 
\begin{align}\label{eq4.54}
	e^{-\i tH}\chi( H)P_{ac}(H)=&
	\frac{m}{\pi \i}\, \int_{0}^{+\infty}e^{-\i t\lambda^{2m}}\left(R_{0}^{+}(\lambda^{2m})-R_{0}^{-}(\lambda^{2m})\right)\lambda^{2m-1}\chi(\lambda^{2m})\d\lambda\nonumber\\
	&-\frac{m}{\pi \i}\, \int_{0}^{\infty}e^{-\i t\lambda^{2m}}R_{0}^{+}vM^{+}(\lambda)^{-1}vR_{0}^{+}(\lambda^{2m})\lambda^{2m-1}\chi(\lambda^{2m})\d\lambda\nonumber\\	
&+\frac{m}{\pi \i}\, \int_{0}^{\infty}e^{-\i t\lambda^{2m}}R_{0}^{-}vM^{-}(\lambda)^{-1}vR_{0}^{-}(\lambda^{2m})\lambda^{2m-1}\chi(\lambda^{2m})\d\lambda\nonumber\\
	:=&\Omega^{low}_{0}-\left(\Omega^{+, low}_{ r}-\Omega^{-, low}_{r}\right).
\end{align}
Since
$\Omega^{low}_{0}=e^{-\i tH_0}\chi(H_0)$,
we have (see e.g. \cite{Mi,HHZ})
\begin{equation} \label{eq4.1.1}
	\left|\Omega^{ low}_{0}(t, x, y)\right|\lesssim |t|^{-\frac{n}{2 m}} \left(1+|t|^{-\frac{1}{2 m}}|x-y|\right)^{-\frac{n(m-1)}{2 m-1}},\quad t\neq0,\,\,x,y\in\mathbb{R}^n.
\end{equation}
 It suffices to estimate the kernel of the remainder term
$\Omega^{+, low}_{r}-\Omega^{-, low}_{r}$.
Applying the expansion \eqref{equ3.48} in Theorem \ref{thm3.4}, we have
\begin{equation}\label{equ4.19}
	\begin{split}
		\Omega_{r}^{\pm, low}(t, x, y)=\sum_{l, h\in J_{\k} }\frac{m}{\pi {i}} \int_{0}^{\infty}& e^{-\i t \lambda^{2 m}}\big{\langle}( M_{ l,h}^{\pm} +\Gamma_{l,h} ^ { \pm } (\lambda ) ) Q_{l}vR_{0} ^ { \pm }(\lambda^{2 m})(\cdot-y), \\
		&Q_{h} v R_{0}^{\mp}(\lambda^{2 m})(\cdot-x) \big{\rangle}\lambda^{4 m-n-l-h-1} \chi(\lambda^{2 m}) \d \lambda.
	\end{split}
\end{equation}
We mention that in order to deal with the oscillatory integral \eqref{equ4.19},  a key step involves analyzing  possible cancellations of $\Omega^{+, low}_{ r}-\Omega^{-, low}_{r}$.

Here, we present the kernel representations of  $Q_j v R_0^{\pm}(\lambda^{2m})$ and $Q_j v(R_0^+(\lambda^{2m}) - R_0^-(\lambda^{2m}))$,  in which we separate out appropriate oscillating factors. The proof relies on  Theorem \ref{thm3.4} and  exploits   the cancellation property of $Q_j$. Due to the technical and lengthy nature of the argument, we provide the detailed proof in Appendix \ref{app-002-8-9}.
\begin{proposition}\label{lemma4.2}
Let $Q_j$ be the orthogonal projections defined in \eqref{eq-Q_j-odd} for $j \in J_{\mathbf{k}}$, and let the exponent $\delta(j) = \max\left\{0, \left[j+\frac{1}{2}\right]\right\}$ as given in \eqref{CforQ_j-1}.
Then, we have
\begin{equation}\label{equ4.2.3}
		\begin{aligned}
			\left(Q_{j}vR_{0}^{\pm}(\lambda^{2 m})(x-\cdot) \right)(y)=\int_{0}^{1}e^{ \pm \i \lambda|x|}k_{j, 1,0}^{\pm}(\lambda,s,y,x) \d s+\int_{0}^{1} e^{\pm\i\lambda s |x|} k_{j,1,1}^{\pm}(\lambda, s,y, x) \d s,
		\end{aligned}
	\end{equation}
where $k_{j,1,i}^{\pm}$ ($i=0,1$) satisfy the following estimates  for $0\le \ell \le [\frac{n}{2m}]+1$ and $\lambda \in (0,1)$: 
\begin{equation}\label{equ4.3}
    \sup_{x,s} 
    \left\| \partial_{\lambda}^\ell  k_{j,1,i}^{\pm}(\lambda, s,\cdot, x) \right\|_{L^2} 
    \les \lambda^{\min\left\{n-2m+\delta(j), 0\right\}-l}, 
\end{equation}
and
\begin{equation}\label{equ4.4}
    \sup_{s} 
    \left\| \partial_{\lambda}^\ell  k_{j,1,i}^{\pm}(\lambda, s,\cdot, x) \right\|_{L^2} 
    \les \lambda^{\min\left\{\frac{n+1}{2}-2m+\delta(j), 0\right\}-l} \langle x \rangle^{-\frac{n-1}{2}}.
\end{equation}
Moreover, when   $\lambda\langle x\rangle \ge 1/2$, we have the following better estimates:
\begin{equation}\label{equ4.4-1-1}
    \sup_{s} 
    \left\| \partial_{\lambda}^\ell  k_{j,1,i}^{\pm}(\lambda, s,\cdot, x) \right\|_{L^2} 
    \les \lambda^{\frac{n+1}{2}-2m+\delta(j)-l} \langle x \rangle^{-\frac{n-1}{2}}.
\end{equation}
For $j\in \{m-\frac{n}{2}, 2m-\frac{n}{2}\}$, we have
\begin{equation}\label{equ4.4.3}
\begin{aligned}
	\Big(Q_{j}v\big(R_{0}^{+}(\lambda^{2 m})(x-\cdot)-&R_{0}^{-}(\lambda^{2 m})(x-\cdot)\big)\Big)(y)\\
	=&\sum_{q=0,1}\left( \int_{0}^{1} e^{\i \lambda s^q|x|} k_{ j,2,i}^{+}(\lambda, s, y, x) \d s-\int_{0}^{1} e^{-\i \lambda s^q|x|} k_{ j,2,i}^{-}(\lambda, s, y, x) \d s\right),
\end{aligned}
\end{equation}
where $k_{j,2,i}^{\pm}(\lambda, s, \cdot, x)$ (for $i=0, 1$) satisfy the following uniform estimates
for $0\le \ell \le [\frac{n}{2m}]+1$ and $\lambda \in (0,1)$:
 \par
\begin{equation}\label{equ4.4.4}
	\sup_{x,s} 
    \left\| \partial_{\lambda}^\ell  k_{2m-\frac n2,2,i}^{\pm}(\lambda, s,\cdot, x) \right\|_{L^2} 
    \les \lambda^{\frac{n+1}{2}-l}, 
\end{equation}
and
\begin{equation}\label{equ4.4.5}
	\sup_{s\in (0, 1)} 
    \left\| \partial_{\lambda}^\ell  k_{2m-\frac n2,2,i}^{\pm}(\lambda, s,\cdot, x) \right\|_{L^2} 
    \les \lambda^{-l} \langle x \rangle^{-\frac{n-1}{2}}.
\end{equation}
Furthermore, when $2m< n< 4m$, $k_{m-\frac{n}{2},2,i}^{\pm}$ ($i=0,1$) satisfy
\begin{equation}\label{equ4.34.2}
\sup_{x,s} 
    \left\| \partial_{\lambda}^\ell  k_{m-\frac n2,2,i}^{\pm}(\lambda, s,\cdot, x) \right\|_{L^2} 
    \les \lambda^{n-2m-l}, 
\end{equation}
 and	
\begin{equation}\label{equ4.34.3} 
    \sup_{s\in (0, 1)} 
    \left\| \partial_{\lambda}^\ell  k_{m-\frac n2,2,i}^{\pm}(\lambda, s,\cdot, x) \right\|_{L^2} 
    \les \lambda^{\frac{n+1}{2}-2m-l} \langle x \rangle^{-\frac{n-1}{2}}; 
\end{equation}
\end{proposition}

The following (oscillatory) integral estimates, which are more or less standard, will be frequently used in our analysis.

\begin{lemma}\label{lemmaA.2}
Let $\chi(\lambda)$ be given by \eqref{equ4.cutoff}, and consider the oscillatory integral
$$
I(t,x)=\int_{0}^{+\infty}e^{-\i t\lambda^{2m}+\i \lambda x}f(\lambda)\chi(\lambda)\, \d\lambda,
$$
where  $f(\lambda)\in C^{K}((0,\lambda_{0}))$ and satisfy that 
$$\left|\frac{d^j}{d\lambda^j} f(\lambda)\right|\le C_j|\lambda|^{b-j} \quad \text{ for $j=0,1,\ldots,K$}. $$
Denoted by $\mu_b=\frac{m-1-b}{2m-1}$, we have

\noindent (1)  If $b\in[-\frac{1}{2},\, 2Km-1)$, then
\begin{equation*}\label{eqA.4}
|I(t,x)|\lesssim
	|t|^{-\frac{1+b}{2m}}(|t|^{-\frac{1}{2m}}|x|)^{-\mu_b},\quad |t|^{-\frac{1}{2m}}|x|\geq1.
\end{equation*}
\noindent (2) If $b\in(-1,\, 2Km-1)$, then
\begin{equation}\label{eqA.5}
			|I(t,x)|\lesssim(1+|t|^{\frac{1}{2m}})^{-(1+b)},\quad |t|^{-\frac{1}{2m}}|x|<1.
\end{equation}
\end{lemma}
\begin{proof}
   For the proof, we refer to \cite[Lemma 2.2]{HHZ} or  \cite[Lemma 2.5]{CHZ}.
\end{proof}

\begin{lemma}\label{lem3.10}
	Let $n\ge 1$. Then there is some absolute constant $C>0$ such that
	\begin{equation*}\label{eq2.20}
		\int_{\mathbb{R}^n}|x-y|^{-k}\langle y\rangle^{-l}\,\d y\leq C\langle x\rangle^{-\min\{k,\, k+l-n\}},
	\end{equation*}
	provided  $l\ge 0$, $0\le k<n$ and $k+l>n$.
\end{lemma}
\begin{proof}
  See e.g. in \cite[Lemma 3.8]{GV} or \cite[Lemma 6.3]{EG10}
\end{proof}

The analysis of Theorem~\ref{thm1.1} proceeds differently depending on whether $|t|^{-\frac{1}{2m}}(|x| + |y|) \geq 1$ or $|t|^{-\frac{1}{2m}}(|x| + |y|) < 1$, which we address separately in the following two subsections.

\subsection{The region $|t|^{-\frac{1}{2m}}(|x| + |y|) \geq 1$}
\label{sec4.2.1}\

Without loss of generality,	we assume $|x| \geq|y|$  and thus $|t|^{-\frac{1}{2m}}|x| \ge \frac12$. The proof of this range naturally splits into two cases based on the value of $\k$: $\k \le m_n$ and $\k=m_n+1$.

We use Proposition \ref{lemma4.2} to rewrite \eqref{equ4.19} as
\begin{equation}\label{equ4.20}
	\begin{split}
		\Omega_{r}^{\pm, low}(t, x, y)= 
		\sum_{p, q \in\{0,1\}}\sum_{l, h\in J_{\k}}\int_{0}^{1} \int_{0}^{1} \int_{0}^{+\infty} e^{-\i t \lambda^{2 m} \pm \i \lambda\left(s_{1}^{p}|y|+s_{2}^{q}|x|\right)} T_{l, h}^{\pm}\left(\lambda, x, y, s_{1}, s_{2}\right)\chi(\lambda^{2 m})\d \lambda \d s_{1} \d s_{2},
	\end{split}\tag{\ref{equ4.19}'}
\end{equation}
where
\begin{equation}\label{equ4.20.00}
	\begin{array}{l}
		T_{l, h}^{\pm}:=\left\langle\left(M_{l, h}^{ \pm}+\Gamma_{l, h}^{\pm}(\lambda)\right)  k_{l,1,p}^{\pm}\left(\lambda, s_{1}, \cdot, y\right), k_{h,1,q}^{\mp}\left(\lambda, s_{2},\cdot,x \right)\right\rangle\lambda^{4 m-n-l-h-1}.
	\end{array}
\end{equation} 
\noindent\textbf{Case 1: $0\le \k \le m_n$.}\

We begin by estimating $T_{l, h}^{\pm}$ in the support of $\chi(\lambda^{2m})$. Specifically, we apply \eqref{equ4.3} to control the term $k_{l,1,p}^{\pm}(\lambda, s_1, \cdot, y)$, and \eqref{equ4.4} to handle $k_{h,1,q}^{\mp}(\lambda, s_2, \cdot, x)$, both of which appear in \eqref{equ4.20.00}. These estimates follow directly from Proposition \ref{lemma4.2}.

Let $0\le \k\le \mathbf{k}_c=\max\{m-\frac{n-1}{2},0\}$. For $l, h\in J_{\mathbf{k}}$, one has $l,h\le \max\{2m-n,\, 0\}$ according to \eqref{equ3.12} and \eqref{equ3.12-1}. Recall that $\delta(l)=\max\{0,[l+1/2]\}$,  thus we obtain
$$
\min\{n-2m+\delta(l),\,0\}=n-2m+\delta(l),
$$
and
$$
 \min\{\mbox{$\frac{n+1}{2} -2m+\delta(h)$}, 0\}=\mbox{$\frac{n+1}{2}-2m+\delta(h)$}.
$$
Using\eqref{equ4.3} (applied to  $k_{l,1,p}^{\pm}\left(\lambda, s_{1}, \cdot, y\right)$), \eqref{equ4.4} (applied to  $k_{h,1,q}^{\mp}\left(\lambda, s_{2},\cdot,x \right)$) and Theorem \ref{thm3.4}, we deduce that for $0<\lambda<\lambda_0$,
\begin{equation}\label{equ4.23}
|\partial_{\lambda}^{\gamma} T_{l, h}^{\pm}(\lambda, x, y, s_{1}, s_{2})|\lesssim\lambda^{\frac{n-1}{2}-\gamma}\langle x\rangle^{-\frac{n-1}{2}} ,\quad 0\le \gamma\le\mbox{$[\frac{n}{2m}]+1$},
\end{equation}
where the inequality holds uniformly for all $x,y\in \R^n$ and $s_{1}, s_{2}\in (0, 1)$.

Now let $\max\{m-\frac{n-1}{2},0\}=\k_c < \k \leq m_n$. This implies that
\begin{equation}\label{eq-8-15-1}
    l \leq \max J_k = \k_c + \k-1 \leq 2m - \tfrac{n+1}{2}.
\end{equation}
 If $l < \max\{2m - n, 0\}$, following the arguments leading to \eqref{equ4.23}, we establish the uniform estimate
\begin{equation}\label{equ4.23-1-1}
|\partial_\lambda^\gamma T_{l,h}^\pm(\lambda,x,y,s_1,s_2)| \lesssim \lambda^{\frac{n-1}{2}-\gamma}\langle x \rangle^{-\frac{n-1}{2}}, \quad 0\le \gamma\le\mbox{$[\frac{n}{2m}]+1$}.
\end{equation}
We next consider $l \geq \max\{2m - n, 0\}$. In this regime, we have 
$$
\min\{n - 2m + \delta(l), 0\} = 0,$$
and
$$
\min\left\{\tfrac{n+1}{2} - 2m + \delta(h), 0\right\} = \tfrac{n+1}{2} - 2m + \delta(h).
$$
By reapplying \eqref{equ4.3} to $k_{l,1,p}^\pm(\lambda,s_1,\cdot,y)$,  \eqref{equ4.4} to $k_{h,1,q}^\mp(\lambda,s_2,\cdot,x)$, and invoking  Theorem \ref{thm3.4}, we conclude, together with  \eqref{equ4.23-1-1}, that  for  $0<\lambda<\lambda_0$, and $0\le \gamma\le [\frac{n}{2m}]+1$,
\begin{equation}\label{equ4.23.2-0}
|\partial_\lambda^\gamma T_{l,h}^\pm(\lambda,x,y,s_1,s_2)| \lesssim 
\begin{cases}
\lambda^{\frac{n-1}{2}-\gamma}\langle x \rangle^{-\frac{n-1}{2}}, & l < \max\{2m - n, 0\}, \\
\lambda^{2m - \frac{n+1}{2} - l - \gamma}\langle x \rangle^{-\frac{n-1}{2}}, & l \geq \max\{2m - n, 0\}.
\end{cases}
\end{equation}
These estimates hold uniformly for all  $x,y\in \R^n$ and $s_{1}, s_{2}\in (0, 1)$.
Furthermore, observe that by \eqref{eq-8-15-1}, the following relationship holds:
$$ \min\{ 2m - \frac{n+1}{2} - l, \frac{n-1}{2}\} \ge 2m - \frac{n+1}{2} - \max J_{\k}=m_n-\k.$$
This,  combined with \eqref{equ4.23.2-0}, yields  for $0<\lambda<\lambda_0$,
\begin{equation}\label{equ4.23.2}
|\partial_{\lambda}^{\gamma} T_{l, h}^{\pm}(\lambda, x, y, s_{1}, s_{2})|\lesssim\lambda^{m_n-\k-\gamma}\langle x\rangle^{-\frac{n-1}{2}} ,\quad 0\le \gamma\le\mbox{$[\frac{n}{2m}]+1$},
\end{equation}
which holds uniformly  for all  $x,y\in \R^n$ and $s_{1}, s_{2}\in (0, 1)$.

To proceed, we use a cut-off function to split the integral region $(s_1, s_2)\in [0, 1]\times [0,1]$ of \eqref{equ4.20} into
 \begin{equation*}
D_1:=\{(s_1, s_2); |t|^{-\frac{1}{2 m}}\left(s_{1}^{p}|y|+s_{2}^{q}|x|\right) \le  1\},\quad\mbox{and}\quad D_2:=\{(s_1, s_2); |t|^{-\frac{1}{2 m}}\left(s_{1}^{p}|y|+s_{2}^{q}|x|\right) \geq1\}.
 \end{equation*}
If $0\le \k\le \mathbf{k}_c$, by \eqref{equ4.23}, we apply Lemma \ref{lemmaA.2} with $b=\frac{n-1}{2}$ to deduce
\begin{equation}\label{equ4.23.0}
	\begin{aligned}
		\left| \eqref{equ4.20}\right|  \leq& \int\int_{D_1} |t|^{-\frac{n+1}{4m}}\langle x\rangle^{-\frac{n-1}{2}} \d s_{1} \d s_{2} \\
		&+\int \int_{D_2} |t|^{-\frac{n+1}{4 m}}\left(|t|^{-\frac{1}{2 m}}\left(s_{1}^{p}|y|+s_{2}^{q}|x|\right)\right)^{-\mu_{\frac{n-1}{2}}} \langle x\rangle^{-\frac{n-1}{2}} \d s_{1} \d s_{2}\\
		 \lesssim& |t|^{-\frac{n}{2 m}}\left(1+|t|^{-\frac{1}{2 m}}|x-y|\right)^{-\frac{n(m-1)}{2m-1}},
	\end{aligned}
\end{equation}
where $\mu_{\frac{n-1}{2}}=\frac{m-1-\frac{n-1}{2}}{2 m-1}$. In the last inequality of \eqref{equ4.23.0}, we use the assumption  $|x-y| \le |x|+|y| \le  2|x|,\ |t|^{-\frac{1}{2 m}}|x|  \gtrsim  1$, the estimates
\begin{equation*}\label{equ4.mub}
	\int\int_{D_1} \d s_{1} \d s_{2} \leq \int\int_{D_1} \left(|t|^{-\frac{1}{2 m}}(s_{1}^{p}|y|+s_{2}^{q}|x|)\right)^{-\max\{{\mu_{\frac{n-1}{2}}}, 0\}}\d s_{1} \d s_{2}\leq (|t|^{-\frac{1}{2 m}}|x|)^{-\max\{{\mu_{\frac{n-1}{2}}}, 0\}},
\end{equation*}
and
\begin{equation*}
			\left(|t|^{-\frac{1}{2 m}}\left(s_{1}^{p}|y|+s_{2}^{q}|x|\right)\right)^{-\mu} \lesssim
			\begin{cases}
				(|t|^{-\frac{1}{2 m}}|x|)^{-\mu},&\text{if}\,\,\mu<0,\\
				(|t|^{-\frac{1}{2 m}}|x|)^{-\mu}\cdot  s_2^{-q\mu},&\text{if}\,\,\mu\ge 0,
			\end{cases}
		\end{equation*}
as well as the following identity
\begin{equation}\label{equ4.23.00}
	\mbox{$\frac{n-1}{2}+\frac{m-1-\frac{n-1}{2}}{2 m-1}=\frac{n(m-1)}{2m-1}$}.
\end{equation}
Similarly, if $\mathbf{k}_c<\k\le m_n$, 
by \eqref{equ4.23.2}, we apply Lemma \ref{lemmaA.2} with $b=m_n-\k$ to obtain
	\begin{align}\label{equ4.23.0000}
		\left| \eqref{equ4.20}\right|  \leq& \int\int_{D_1} |t|^{-\frac{m_n-\k+1}{2m}}\langle x\rangle^{-\frac{n-1}{2}} \d s_{1} \d s_{2} \nonumber\\
		&+\int \int_{D_2} |t|^{-\frac{m_n-\k+1}{2m}}\left(|t|^{-\frac{1}{2 m}}\left(s_{1}^{p}|y|+s_{2}^{q}|x|\right)\right)^{-\mu_{m_n-\k}} \langle x\rangle^{-\frac{n-1}{2}} \d s_{1} \d s_{2}\nonumber\\
		\lesssim& |t|^{-\frac{m_n-\k+\frac{n+1}{2}}{2m}}\left(1+|t|^{-\frac{1}{2 m}}|x-y|\right)^{-\frac{n(m-1)}{2m-1}} \\
        \les& |t|^{-\frac{2m_n-2\k+1}{2m}}(1+|t|^{-\frac{n}{2m}})\left(1+|t|^{-\frac{1}{2 m}}|x-y|\right)^{-\frac{n(m-1)}{2m-1}}. 
	\end{align}
    where we used the facts 
$$\mu_{m_n-\k}= \frac{m-1-(m_n-\k)}{2 m-1}=\frac{m-1-\frac{n-1}{2}+\frac{n-1}{2}-(m_n-\k)}{2 m-1} \ge \mu_{\frac{n-1}{2}}, $$
    and 
$$2m-2\k+1 \le m_n-\k+\frac{n+1}{2}  \le n . $$

\vspace{1em}
\noindent\textbf{Case 2: $\k=m_n+1$.}\ 

Now we turn to the case that zero is an eigenvalue. Observe  that in \eqref{equ4.19},  if $l\neq2m-\frac{n}{2}$ or $h\neq2m-\frac{n}{2}$,
		by \eqref{equ3.47},  \eqref{equ4.3} and \eqref{equ4.4}, we have
		\begin{equation}\label{equ4.23.3}
			 \sup_{y,s_1,s_2}|\partial_{\lambda}^{\gamma} T_{l, h}^{\pm}(\lambda, x, y, s_{1}, s_{2})|\lesssim\lambda^{-\gamma}\langle x\rangle^{-\frac{n-1}{2}} ,\quad \gamma=0,\cdots,\mbox{$[\frac{n}{2m}]+1$}.
		\end{equation}
Consequently, the estimates for integrals involving these terms coincide precisely with those obtained in the case $\k = m_n$ (see \eqref{equ4.23.2}).

We are left to estimate the kernel of
\begin{equation}\label{equ4.ker}
\begin{aligned}
 \int_{0}^{\infty}e^{-\i t \lambda^{2 m}}  \lambda^{-1}\chi(\lambda^{2 m}) \Psi(\lambda) \d \lambda,   
\end{aligned}
\end{equation}
where
\begin{equation}\label{eq:def for Psi(lambda)}
\Psi(\lambda)=R_{0}^{+}(\lambda^{2m}) v Q_{l} \big(M_{l, h}^{+}+\Gamma_{l,h}^+(\lambda)\big) Q_{h}v R_{0}^{+}(\lambda^{2m})
-R_{0}^{-}(\lambda^{2m}) v Q_{l} \big(M_{l, h}^{-}+\Gamma_{l,h}^{-}(\lambda)\big) Q_{h} v R_{0}^{-}(\lambda^{2m}),
\end{equation}
with $l=h=2m-\frac n2$.
The analysis of this integral shall be divided into two cases based on the dimension: $n < 2m$ and $2m < n < 4m$.\\

\noindent\emph{Subcase 2.1: $n<2m$.}\

In this subcase, we require more precise information about $\Psi(\lambda)$.
Select a smooth cutoff function $\phi \in C_c^\infty(\mathbb{R})$ satisfying
\begin{equation*}
\phi(t) = \begin{cases} 
1, &|t| \leq \frac{1}{2}, \\
0. &|t| \geq 1.
\end{cases}
\end{equation*}
We introduce the partition for the kernel of $\Psi(\lambda)$
\begin{equation}\label{eq:dec of Psi}
 \Psi(\lambda)(x,y)=(1-\phi(\lambda\langle x\rangle)\phi(\lambda\langle y\rangle))\Psi(\lambda)(x,y)+\phi(\lambda\langle x\rangle)\phi(\lambda\langle y\rangle)\Psi(\lambda)(x,y).   
\end{equation}
As in the previous argument, the kernel of \eqref{equ4.ker} associated with the first term on the RHS of \eqref{eq:dec of Psi} admits a decomposition
\begin{equation}\label{eq:firt two terms}
\begin{aligned}
\int_{0}^{1} \int_{0}^{1}\int_{0}^{+\infty} & e^{-\i t\lambda^{2m}\pm \i\lambda( s_{1}^{p}|y|+ s_{2}^q|x|)} T_{l, h}^{\pm}(\lambda, x, y, s_{1}, s_{2}) \left[1-\phi(\lambda\langle x\rangle)\phi(\lambda\langle y\rangle)\right]\chi (\lambda^{2 m})\d \lambda\d s_{1} \d s_{2},
\end{aligned}
\end{equation}
for $l=h=2m-\frac n2$, where $T_{l, h}^{\pm}$ are defined in \eqref{equ4.20.00}.
By \eqref{equ4.4-1-1},  one has  
		\begin{equation}\label{equ4.4.3.3}
		\sup_{s}\left\|\partial_{\lambda}^{\gamma} \Big[ (1-\phi(\lambda\langle x\rangle)) k^{\pm}_{2m-\frac{n}{2},1,1}(\lambda,s,\cdot,x)\Big]\right\|_{L^2}\les\lambda^{1-l} \langle x\rangle^{-\frac{n-1}{2}}, 
		\end{equation}
for $\lambda \in (0,1)$ and $l=0,\cdots,[\frac{n}{2m}]+1$. Noting  
$$(1-\phi(\lambda\langle x\rangle)\phi(\lambda\langle y\rangle))=1-\phi(\lambda\langle x\rangle)+(1-\phi(\lambda\langle y\rangle)\phi(\lambda\langle x\rangle)), $$
this, together with \eqref{equ4.4}, shows that 
$$\bigg|\partial_{\lambda}^{\gamma}\left[(1-\phi(\lambda\langle x\rangle)\phi(\lambda\langle y\rangle)) T_{l, h}^{\pm}(\lambda, x, y, s_{1}, s_{2})\right]\bigg|\lesssim\lambda^{-\gamma}\langle x\rangle^{-\frac{n-1}{2}} ,\quad \gamma=0,\mbox{$[\frac{n}{2m}]+1$},$$ where we note that $[\frac{n}{2m}]+1=1$. Then, we apply  Lemma \ref{lemmaA.2}  with $b=0$ to yield
		\begin{equation}\label{equ4.4.4.3}
			\begin{aligned}
				\left| \eqref{eq:firt two terms} \right|
				&\lesssim\sum_{p,q\in\{0, 1\}}\int_{0}^{1} \int_{0}^{1}
				|t|^{-\frac{1}{2m}}\left(|t|^{-\frac{1}{2 m}}\left|s_{1}^{p}|y|+ s_{2}^{q}|x|\right| \right)^{-\frac{m-1}{2 m-1}}
				\langle x\rangle^{-\frac{n-1}{2}}\d s_{1} \d s_{2} \\
				&\lesssim\sum_{p,q\in\{0, 1\}}\int_{0}^{1} \int_{0}^{1}
				|t|^{-\frac{1}{2m}}\left(|t|^{-\frac{1}{2 m}} |x|\right)^{-\frac{m-1}{2 m-1}}
				\langle x\rangle^{-\frac{n-1}{2}} s_2^{-q\frac{m-1}{2 m-1}}\d s_{1} \d s_{2} \\
				&\lesssim |t|^{-\frac{n+1}{4m}}\left(1+|t|^{-\frac{1}{2 m}}|x-y|\right)^{-\frac{n(m-1)}{2 m-1}} \\
                & \les |t|^{-\frac{1}{2m}}(1+|t|^{-\frac{n}{2m}})\left(1+|t|^{-\frac{1}{2 m}}|x-y|\right)^{-\frac{n(m-1)}{2 m-1}},
			\end{aligned}
		\end{equation}
		where in the third inequality we use $|x|\geqslant \frac12 |x-y|$ and the fact $\frac{m-1}{2 m-1}+\frac{n-1}{2}\ge\frac{n(m-1)}{2 m-1}$.

Next, we estimate the part involved with $\phi(\lambda\langle x\rangle)\phi(\lambda\langle y\rangle) \Psi(\lambda, x, y)$ in \eqref{eq:dec of Psi}, i.e.
		\begin{equation}\label{shit}
			\int_{0}^{+\infty}e^{-\i t\lambda^{2m}}\phi(\lambda\langle x\rangle)\phi(\lambda\langle y\rangle) \Psi(\lambda, x, y)
			\lambda^{-1} \chi(\lambda^{2 m} ) \d \lambda.
		\end{equation}
Following by \eqref{eq:def for Psi(lambda)}, we can decompose $\Psi(\lambda)$ as
$$\Psi(\lambda)= \Psi_M(\lambda)+\Psi_{\Gamma}(\lambda) $$
where 
$$\Psi_M(\lambda)
=R_{0}^{+} v Q_{2 m-\frac{n}{2}} M_{2m-\frac{n}{2}, 2m-\frac{n}{2}}^{+} Q_{2 m-\frac{n}{2}} v R_{0}^{+}
-R_{0}^{-} v Q_{2 m-\frac{n}{2}} M_{2m-\frac{n}{2}, 2m-\frac{n}{2}}^{-} Q_{2 m-\frac{n}{2}} v R_{0}^{-},
$$
and 
$$\Psi_{\Gamma}(\lambda)
=R_{0}^{+} v Q_{2 m-\frac{n}{2}} \Gamma^{+}_{2m-\frac{n}{2}, 2m-\frac{n}{2}}(\lambda) Q_{2 m-\frac{n}{2}} v R_{0}^{+}
-R_{0}^{-} v Q_{2 m-\frac{n}{2}} \Gamma^{-}_{2m-\frac{n}{2}, 2m-\frac{n}{2}}(\lambda) Q_{2 m-\frac{n}{2}} v R_{0}^{-}.
$$
Note that by \eqref{equ3.47.1}, \eqref{equ4.3} and \eqref{equ4.4}, one has
\begin{equation}\label{equ4.23.3-111111}
			\left|\partial_{\lambda}^{\gamma}\left[\lambda^{4 m-n-l-h-1}\left\langle\Gamma_{l, h}^{ \pm}(\lambda) Q_{l} k_{l,1,p}^{ \pm}\left(\lambda, s_{1}, \cdot, y\right), \,Q_{h}k_{h,1,q}^{\mp}\left(\lambda, s_{2},\cdot,x \right)\right\rangle \right]\right| \les \lambda^{-\gamma} \langle x\rangle^{\frac{n-1}{2}},
\end{equation}
for $\gamma=0,\cdots,[\frac{n}{2m}]+1$ and $l=h=2m-\frac n2$, thus the integral related to $\Psi_{\Gamma}(\lambda)$ also satisfies \eqref{equ4.4.4.3}. 

Therefore, we only need to consider
\begin{equation*}
\int_{0}^{+\infty}e^{-\i t\lambda^{2m}}\phi(\lambda\langle x\rangle)\phi(\lambda\langle y\rangle) \Psi_{M}(\lambda, x, y)
\lambda^{-1} \chi(\lambda^{2 m} ) \d \lambda.
\end{equation*}
Since $M_{2m-\frac{n}{2}, 2m-\frac{n}{2}}^{\pm}=(b_1Q_{2m-\frac{n}{2}}vG_{4m-n}vGQ_{2m-\frac{n}{2}})^{-1}$ is independent of the sign $\pm$, we exploit the cancellation  by writing
\begin{equation*}
\begin{aligned}
\Psi_M(\lambda)
=&\left(R_{0}^{+}(\lambda^{2 m})-R_{0}^{-}(\lambda^{2 m})\right) v Q_{2 m-\frac{n}{2}} M_{2m-\frac{n}{2}, 2m-\frac{n}{2}}^{+} Q_{2 m-\frac{n}{2}} v R_{0}^{+}(\lambda^{2 m}) \\
&+R_{0}^{-}(\lambda^{2 m}) v Q_{2 m-\frac{n}{2}} M_{2m-\frac{n}{2}, 2m-\frac{n}{2}}^{-} Q_{2 m-\frac{n}{2}} v\left(R_{0}^{+}(\lambda^{2 m})-R_{0}^{-}(\lambda^{2 m})\right),\\
:=&\Psi_{M,1}(\lambda)+\Psi_{M,2}(\lambda).
\end{aligned}
\end{equation*}
We only consider
\begin{equation}\label{equ4.24.3-11}
\int_{0}^{+\infty}e^{-\i t\lambda^{2m}}\phi(\lambda\langle x\rangle)\phi(\lambda\langle y\rangle) \Psi_{M,1}(\lambda, x, y)
\lambda^{-1} \chi(\lambda^{2m} ) \d \lambda,
\end{equation}
and the argument for the other term follows similarly. By using \eqref{equ4.3} and \eqref{equ4.4.3}, the integral \eqref{equ4.24.3-11} can be written as a linear combination of
\begin{align}\label{equ4.23.12.00}
\int_{0}^{1} \int_{0}^{1}\int_{0}^{+\infty}e^{-\i t\lambda^{2m}+\i \lambda s_{1}^{p}|y|\mp \i \lambda s_{2}^q|x|} \phi(\lambda\langle x\rangle)\phi(\lambda\langle y\rangle)
T_{M,1,2m-\frac{n}{2},2m-\frac{n}{2}}^{ \pm}\left(\lambda, s_{1}, s_{2}, x, y\right)\chi(\lambda^{2 m})  \d \lambda \d s_{1} \d s_{2},
\end{align}
where $p, q\in \{0, 1\}$, and
$$
T_{M,1,2m-\frac{n}{2},2m-\frac{n}{2}}^{ \pm}\left(\lambda, s_{1}, s_{2}, x, y\right)=\left\langle M_{2m-\frac{n}{2}, 2m-\frac{n}{2}}^{+} k_{2 m-\frac{n}{2},1,p}^{+}\left(\lambda, s_{1},\cdot, y\right), \, 
k_{2m-\frac{n}{2},2,q,}^{ \pm }\left(\lambda, s_{2},\cdot, x\right)\right\rangle\lambda^{-1}.$$
By \eqref{equ4.3} and \eqref{equ4.4.5}, we have
\begin{equation*}
\begin{aligned}
\sup_{y,s_1,s_2}\left|\partial_{\lambda}^{\gamma}\big[\phi(\lambda\langle x\rangle)
T_{M,1,2m-\frac{n}{2},2m-\frac{n}{2}}^{ \pm}\left(\lambda, s_{1}, s_{2}, x, y\right)\big]\right|	
&\les \lambda^{-\gamma}  \langle x\rangle^{-\frac{n-1}{2}} \\
&\les \lambda^{-\frac{m-\frac{n+1}{2}}{2m-1}-\gamma}  \langle x\rangle^{-\frac{n(m-1)}{2m-1}},    
\end{aligned}
 \end{equation*}
 for $\gamma=0,1.$
		Note that  $-\frac{m-\frac{n+1}{2}}{2m-1}>-\frac12$,
		then applying Lemma \ref{lemmaA.2} with $b=-\frac{m-\frac{n+1}{2}}{2m-1}$, we obtain
		\begin{equation}\label{equ4.23.6}
			\begin{aligned}
				\left| \eqref{equ4.23.12.00} \right|
				\lesssim&\int_{0}^{1} \int_{0}^{1}|t|^{-\frac{1+b}{2m}}
				\left(1+|t|^{-\frac{1}{2 m}}\left|s_{1}^{p} |y|\mp s_{2}^{q}|x|\right| \right)^{-\frac{m-1-b}{2 m-1}}
				\langle x\rangle^{-\frac{n(m-1)}{2m-1}}\d s_{1} \d s_{2} \\
				\le & |t|^{-\frac{1+b}{2m}}\langle x\rangle^{-\frac{n(m-1)}{2m-1}}\\
				\lesssim& |t|^{-\frac{n+1}{4m}} \left(1+|t|^{-\frac{1}{2 m}}|x-y|\right)^{-\frac{n(m-1)}{2m-1}} \\
                \les& |t|^{-\frac{1}{2m}}(1+|t|^{-\frac{n}{2m}})\left(1+|t|^{-\frac{1}{2 m}}|x-y|\right)^{-\frac{n(m-1)}{2 m-1}},
			\end{aligned}
		\end{equation}
		where in the above inequalities, we have used the identity \eqref{equ4.23.00} and the facts that $|x|\gtrsim |x-y|$ and $\frac{n+1}{4m} \ge \frac{1}{2m}$.\\

\noindent\emph{Subcase 2.2: $2m< n<4m$.}\

In this subcase, we estimate \eqref{equ4.ker} in a more direct way. Indeed, we can write
\begin{equation*}
\begin{split}
\Psi(\lambda)=
 \Psi_{1}(\lambda)+\Psi_{2}(\lambda)+\Psi_{3}(\lambda),
\end{split}
\end{equation*}
where 
\begin{equation}\label{equ-psi-8-15}
\begin{aligned}
\Psi_{1}(\lambda)=& \big(R_{0}^{+}(\lambda^{2m})-R_{0}^{+}(\lambda^{2m})\big) v Q_{l} \big(M_{l, h}^{+}+\Gamma_{l,h}^+(\lambda)\big) Q_{h}v R_{0}^{+}(\lambda^{2m}), \\  
\Psi_{2}(\lambda)=& R_{0}^{-}(\lambda^{2m})vQ_l\big(\Gamma_{l, h}^{+}(\lambda)-\Gamma_{l, h}^{-}(\lambda)\big) Q_{h}v R_{0}^{+}(\lambda^{2m}), \\
\Psi_{3}(\lambda)=& R_{0}^{-}(\lambda^{2m})vQ_l\big(M_{l, h}^{-}+\Gamma_{l,h}^-(\lambda)\big) Q_{h}v\big(R_{0}^{+}(\lambda^{2m})-R_{0}^{-}(\lambda^{2m})\big),  
	\end{aligned}
		\end{equation}
for $l=h=2m-\frac n2$. Then \eqref{equ4.ker} is decomposed into 
\begin{align}\label{eq: exp for I_1}
\int_{0}^{+\infty}e^{-\i t\lambda^{2m}}\big(\Psi_{1}(\lambda)+\Psi_{2}(\lambda)+\Psi_{3}(\lambda)\big)
\lambda^{-1} \chi(\lambda^{2m} ) \d \lambda=I_1+I_2+I_3.
\end{align}
By using \eqref{equ4.3} and \eqref{equ4.4.3}, the kernel of $I_1$ can be written as a linear combination of
\begin{align}\label{equ4.23.12.00-11}
\int_{0}^{1} \int_{0}^{1}\int_{0}^{+\infty}e^{-\i t\lambda^{2m}+\i \lambda s_{1}^{p}|y|\mp \i \lambda s_{2}^q|x|} 
T_{1,2m-\frac{n}{2},2m-\frac{n}{2}}^{ \pm}\left(\lambda, s_{1}, s_{2}, x, y\right) \chi(\lambda^{2m})  \d \lambda \d s_{1} \d s_{2},
\end{align}
where 
$$T_{1,l,h}^{\pm}(\lambda, s_{1}, s_{2}, x, y)=\left\langle (M_{l,h}^{+}+\Gamma_{l,h}^{+}(\lambda)) k_{l,1,p}^{+}\left(\lambda, s_{1},\cdot, y\right), \, 
k_{h,2,q,}^{ \pm }\left(\lambda, s_{2},\cdot, x\right)\right\rangle\lambda^{-1},$$
with $l=h=2m-\frac n2$.
By \eqref{equ4.3} and \eqref{equ4.4.5}, we have
\begin{equation}\label{eq:est for T_1}
\sup_{y,s_1,s_2}\left|\partial_{\lambda}^{\gamma}T_{1,2m-\frac{n}{2},2m-\frac{n}{2}}^{ \pm}\left(\lambda, s_{1}, s_{2}, x, y\right)\right|	\les \lambda^{-\gamma}  \langle x\rangle^{-\frac{n-1}{2}}, \quad\mbox{$\gamma=0,\cdots, [\frac{n}{2m}]+1$}. 
\end{equation}
Applying Lemma \ref{lemmaA.2} with $b=0$, we have
		\begin{align}\label{equ4.50}
			\left|I_1(t,x,y)\right|
			&\lesssim\int_{0}^{1} \int_{0}^{1}|t|^{-\frac{1}{2m}}\left(1+|t|^{-\frac{1}{2 m}}\left|s_{1}^{p} |y|\mp s_{2}^{q}|x|\right| \right)^{-\frac{m-1}{2 m-1}}
			\langle x\rangle^{-\frac{n-1}{2}}\d s_{1} \d s_{2} \nonumber\\
			&\lesssim |t|^{-\frac{1}{2m}}
			\langle x\rangle^{-\frac{n-1}{2}}
			\lesssim |t|^{-\frac{n+1}{4m}}\left(1+t^{-\frac{1}{2 m}}|x-y|\right)^{-\frac{n(m-1)}{2 m-1}},
		\end{align}
where we have used the fact that $\frac{n-1}{2}\ge \frac{n(m-1)}{2 m-1}$ when $2m<n<4m$.
We remark that the analysis of $I_2$ follows similarly to that of the integral \eqref{shit} where \eqref{equ4.23.3-111111} was used, while $I_3$ can be treated analogously to $I_1$. This completes the proof.

To conclude, by \textbf{Case 1} and \textbf{Case 2}, we have
		\begin{equation*}
			\begin{aligned}
				\left|\Omega_{r}^{+, low}(t, x, y)-\Omega_{r}^{-, low}(t, x, y)\right| \lesssim |t|^{-h(m, n, \mathbf{k})}(1+|t|^{-\frac{n}{2m}})
				\left(1+t^{-\frac{1}{2 m}}|x-y|\right)^{-\frac{n(m-1)}{2 m-1}},
			\end{aligned}
\end{equation*}
when $|t|^{-\frac{1}{2m}}(|x| + |y|) > 1$.

\subsection{The region $|t|^{-\frac{1}{2m}}(|x| + |y|) < 1$} \label{sec4.2.2}\ 

The proof of this range also splits into two cases based on the value of $\k$: the range $\k \le m_n$ and $\k=m_n+1$.\\

\noindent\textbf{Case 1: $0\le \k \le m_n$.}\

We divide this case into two subcases based on the dimension.\\

\noindent\emph{Subcase 1.1: $n<2m$.}\

In this case, we still use the representation \eqref{equ4.20}. By Theorem \ref{thm3.4} and \eqref{equ4.3}, it follows for  $l, h\in J_{\k}$ that, for $\k\le \k_c=m-\frac{n-1}{2}$ we have
\begin{equation*}
	\sup_{x,y,s_1,s_2} \left|\partial_{\lambda}^{\gamma}T_{l, h}^{\pm }\left(\lambda, x, y, s_{1}, s_{2}\right)\right| \lesssim \lambda^{n-1-\gamma},\quad\, \mbox{$\gamma=0,[\frac{n}{2m}]+1$},
\end{equation*}
while for $\mathbf{k}_c<\k\le m_n$ we have  
\begin{equation*}
	\sup_{x,y,s_1,s_2}\left|\partial_{\lambda}^{\gamma}T_{l,h}^{\pm}(\lambda,s_{1},s_{2},x,y) \right|\les \lambda^{2m-2\k-\gamma}, \quad \gamma=0,\mbox{$[\frac{n}{2m}]+1$},
\end{equation*}
where we note that $[\frac{n}{2m}]+1=1$. Applying  \eqref{eqA.5}  in Lemma \ref{lemmaA.2} with $b=n-1$ for $\k \le k_c$ and $b=2m-2\k$ for $\k_c<\k\le m_n$  yields
\begin{equation}\label{equ4.22.111-0}
|\eqref{equ4.20}| \lesssim |t|^{-\frac{n}{2 m}} \int_{0}^{1} \int_{0}^{1}  \d s_{1} \d s_{2}
\lesssim |t|^{-\frac{n}{2 m}}\left(1+|t|^{-\frac{1}{2 m}}|x-y|\right)^{-\frac{n(m-1)}{2 m-1}}, \quad \k\le \k_c,
\end{equation}
and
\begin{equation}\label{equ4.22.111-1}
|\eqref{equ4.20}| \lesssim |t|^{-\frac{2m+1-2\k}{2 m}} \int_{0}^{1} \int_{0}^{1}  \d s_{1} \d s_{2}
	\lesssim |t|^{-\frac{2m+1-2\k}{2 m}}\left(1+|t|^{-\frac{1}{2 m}}|x-y|\right)^{-\frac{n(m-1)}{2 m-1}}, \quad \k_c<\k\le m_n,
\end{equation}
where we have used  $|t|^{-\frac{1}{2m}}|x-y|< 1$.\\

\noindent\emph{Subcase1.2: $2m<n<4m$.}\

We now turn to the analysis for dimensions in the range $2m < n < 4m$.
In this case, when zero is an eigenvalue, i.e. $\k=0$, we need to take advantage of the cancellation in $\Omega_{r}^{+, low}-\Omega_{r}^{-, low}$. Indeed, we write
\begin{equation*}
\begin{split}
&R_{0}^{+}(\lambda^{2 m}) v M^{+}(\lambda)^{-1} v R_{0}^{+}(\lambda^{2 m})-R_{0}^{-}(\lambda ^{2 m}) v M^{-}(\lambda)^{-1} v R_{0}^{-}(\lambda^{2 m})\\
=&\left(R_{0}^{+}(\lambda^{2 m})-R_{0}^{-}(\lambda^{2 m})\right) v M^{+}(\lambda)^{-1} v R_{0}^{+}(\lambda^{2 m})+R_{0}^{-}(\lambda^{2 m}) v\left(M^{+}(\lambda)^{-1}-M^-(\lambda)^{-1}\right) v R_{0}^{+}(\lambda^{2 m})\\
&+R_{0}^{-}(\lambda^{2 m}) v M^{-}(\lambda)^{-1} v\left(R_{0}^{+}(\lambda^{2 m})-R_{0}^{-}(\lambda^{2 m})\right)\\
:= & \Phi_{1}(\lambda)+\Phi_{2}(\lambda)+\Phi_{3}(\lambda),
\end{split}
\end{equation*}
and we deduce from \eqref{eq4.54} that
\begin{equation*}
\Omega^{+, low}_{ r}-\Omega^{-, low}_{r}=\frac{m}{\pi {i}} \int_{0}^{\infty} e^{-\i t \lambda^{2 m}}\left(\Phi_{1}(\lambda)+\Phi_{2}(\lambda)+\Phi_{3}(\lambda)\right)\lambda^{2m-1} \chi(\lambda^{2 m}) \d \lambda.
\end{equation*}
By  Theorem \ref{thm3.4} and Proposition \ref{lemma4.2},  the kernel of the integral involving $\Phi_{1}(\lambda)$ is a linear combination of
\begin{equation}\label{eq:int inv PHI_1}
\int_{0}^{1} \int_{0}^{1} \int_{0}^{\infty} e^{-\i t \lambda^{2 m}+\i \lambda s_{1}^{p}|y| \mp \i \lambda s_{2}^q|x|} T_{1, m-\frac{n}{2},m-\frac{n}{2}}^{\pm}(\lambda,s_{1},s_{2},x,y) \chi(\lambda^{2m}) \d\lambda\d s_{1} \d s_{2},
\end{equation}
where $p, q\in\{0, 1\}$  and
\begin{equation*}
\begin{split}
&T_{1, m-\frac{n}{2},m-\frac{n}{2}}^{\pm}(\lambda,s_{1},s_{2},x,y)\\
=&\left\langle\left(M^\pm_{m-\frac{n}{2},m-\frac{n}{2}}+\Gamma^{+}_{m-\frac{n}{2},m-\frac{n}{2}}(\lambda)\right) k_{m-\frac{n}{2},1,p}^{+}(\lambda, s_{1}, \cdot, y), k_{m-\frac{n}{2},2,q}^{\pm}\left(\lambda, s_{2}, \cdot, x\right)\right\rangle\lambda^{2m-1}.
\end{split}
\end{equation*}
 Since zero is regular, we apply \eqref{equ3.47.1} with $\k=0$ to deduce
		\begin{equation}\label{equ4.39}
			M^\pm_{m-\frac{n}{2},m-\frac{n}{2}}+\Gamma^{+}_{m-\frac{n}{2},m-\frac{n}{2}}(\lambda) \in  \mathfrak{S}_{2}^{0}((0,\lambda_{0})).
		\end{equation}		
Then \eqref{equ4.39}, \eqref{equ4.3} and \eqref{equ4.34.2} imply
\begin{equation*}\label{equ4.39.1}
\sup_{x,y,s_1,s_2}\left|\partial_{\lambda}^{\gamma}T_{1,m-\frac{n}{2},m-\frac{n}{2}}^{\pm}(\lambda,s_{1},s_{2},x,y) \right|\les \lambda^{n-1-\gamma}, \quad \gamma=0,\cdots,\mbox{$[\frac{n}{2m}]+1$},
\end{equation*}
where we note that $[\frac{n}{2m}]+1=2$
Applying  \eqref{eqA.5} in Lemma \ref{lemmaA.2} with $b=n-1$ yields
\begin{align}\label{equ4.40}
\left|\eqref{eq:int inv PHI_1}\right| 
\lesssim (1+|t|)^{-\frac{n}{2 m}} \lesssim (1+|t|)^{-\frac{n}{2 m}}\left(1+|t|^{-\frac{1}{2 m}}|x-y|\right)^{-\frac{n(m-1)}{2 m-1}}.
\end{align}	
	
Note that $M^\pm_{m-\frac{n}{2},m-\frac{n}{2}}=(Q_{m-\frac{n}{2}}T_0Q_{m-\frac{n}{2}})^{-1}$  are independent of the sign $\pm$.	
To consider the integral involving $\Phi_{2}$,  by using \eqref{equ4.2.3},   its kernel is a linear combination of 
\begin{equation}\label{equ4.41}
\frac{\pi \i}{m}\int_{0}^{1} \int_{0}^{1}\int_{0}^{\infty} e^{-\i t \lambda^{2m}+\i \lambda s_{1}^{p}|y|-\i \lambda s_{2}^{q}|x|} T_{2, m-\frac{n}{2},m-\frac{n}{2}}(\lambda, s_{1}, s_{2}, x, y) \chi(\lambda)\d\lambda  \d s_{1} \d s_{2},
\end{equation}
where $p,q\in\{0, 1\}$ and 
\begin{equation*}
\begin{split}
&T_{2, m-\frac{n}{2},m-\frac{n}{2}}(\lambda, s_{1}, s_{2}, x, y)\\
=&\left\langle(\Gamma_{ m-\frac{n}{2},  m-\frac{n}{2}}^{+}(\lambda)-\Gamma_{ m-\frac{n}{2},  m-\frac{n}{2}}^{-}(\lambda)) k_{m-\frac{n}{2},1,p}^{+}\left(\lambda, s_{1}, \cdot, y\right) , \,k_{m-\frac{n}{2},1,q}^{+}\left(\lambda, s_{2},\cdot, x\right)\right\rangle\lambda^{2 m-1}.
\end{split}
\end{equation*}	
Using  \eqref{equ3.47.1}, one has 
$$ \Gamma_{ m-\frac{n}{2},  m-\frac{n}{2}}^{+}(\lambda)-\Gamma_{ m-\frac{n}{2},  m-\frac{n}{2}}^{-}(\lambda) \in  \mathfrak{S}_{2}^{n-2m}((0,\lambda_{0})).$$
Combining this with \eqref{equ4.3}, we have  
\begin{equation*}
\sup_{x,y,s_1,s_2} \left|\partial_{\lambda}^{\gamma}T_{2,m-\frac{n}{2},m-\frac{n}{2}}^{\pm}(\lambda,s_{1},s_{2},x,y) \right|\les \lambda^{n-1-\gamma}, \quad \gamma=0,1,\mbox{$[\frac{n}{2m}]+1$}.
\end{equation*}
Since $|t|^{-\frac{1}{2m}}\big|s_1^p|y|-s_2^q|x|\big|\leq1$ for $p,q\in\{0,1\}$,
Lemma \ref{lemmaA.2} gives
\begin{equation}\label{equ4.42}
\begin{aligned}
\left| \eqref{equ4.41} \right|  
\lesssim(1+|t|)^{-\frac{n}{2 m}} \lesssim (1+|t|)^{-\frac{n}{2 m}}\left(1+|t|^{-\frac{1}{2 m}}|x-y|\right)^{-\frac{n(m-1)}{2 m-1}}.
\end{aligned}
\end{equation}

The estimate for the integral associated with $\Phi_{3}(\lambda)$ is the same as that associated with $\Phi_{1}(\lambda)$, and we omit the details.
Combining with \eqref{equ4.40} and \eqref{equ4.42}, we have
\begin{equation}\label{equ4.43}
\left|\Omega_{r}^{+, low}(t, x, y)-\Omega_{r}^{-, low}(t, x, y)\right| \lesssim (1+|t|)^{-\frac{n}{2 m}}\left(1+|t|^{-\frac{1}{2 m}}|x-y|\right)^{-\frac{n(m-1)}{2 m-1}},
\end{equation}
which completes the case of $2m< n<4m$ and $\k=0$.

When $2m< n<4m$ and $1\le \k\le m_n=2m-\frac{n-1}{2}$, similar to the subcase $n<2m$, we use \eqref{equ4.20}. Furthermore, we can obtain 
\begin{equation*}
	\sup_{x,y,s_1,s_2}\left|\partial_{\lambda}^{\gamma}T_{l,h}^{\pm}(\lambda,s_{1},s_{2},x,y) \right|\les \lambda^{2m_n-2\k-\gamma}, \quad \gamma=0,\cdots,\mbox{$[\frac{n}{2m}]+1$},
\end{equation*}
where $[\frac{n}{2m}]+1=2$, which shows that 
\begin{equation}\label{equ4.43-01}
\left|\Omega_{r}^{+, low}(t, x, y)-\Omega_{r}^{-, low}(t, x, y)\right| \lesssim (1+|t|)^{-\frac{2m_n+1-2\k}{2 m}} \left(1+|t|^{-\frac{1}{2 m}}|x-y|\right)^{-\frac{n(m-1)}{2 m-1}}.
\end{equation}\\

\noindent\textbf{Case 2: $\k=m_n+1$.}\

Note that by \eqref{equ4.20}, \eqref{equ4.23.3} and the proof for \textbf{Case 1}, we are also left to estimate the kernel of 
\begin{equation}\label{equ4.ker-01}
\begin{aligned}
 \int_{0}^{\infty}e^{-\i t \lambda^{2 m}}  \Psi(\lambda)\lambda^{-1}\chi(\lambda^{2 m})  \d \lambda,  
\end{aligned}
\end{equation}
where $ \Psi(\lambda)$  is defined in  \eqref{eq:def for Psi(lambda)}.
Following the strategy employed in \emph{Subcase 2.2} of Subsection~\ref{sec4.2.1}, we decompose  \eqref{equ4.ker-01} into
\begin{align}\label{eq:psi-decomposition}
\int_{0}^{\infty} e^{-\mathrm{i} t\lambda^{2m}} \big(\Psi_{1}(\lambda) + \Psi_{2}(\lambda) + \Psi_{3}(\lambda)\big) \lambda^{-1} \chi(\lambda^{2m}) \d\lambda = I_1 + I_2 + I_3,
\end{align}
with  $\Psi_1(\lambda),\Psi_2(\lambda),\Psi_3(\lambda)$  as specified in \eqref{equ-psi-8-15}. 

For the term $I_1(t,x,y)$, it suffices to estimate \eqref{equ4.23.12.00-11}, we  have
\begin{equation}\label{eq:est for T_1-1}
\sup_{y,s_1,s_2}\left|\partial_{\lambda}^{\gamma}T_{1,2m-\frac{n}{2},2m-\frac{n}{2}}^{ \pm}\left(\lambda, s_{1}, s_{2}, x, y\right)\right|	\les \lambda^{-\gamma}, \quad \gamma=0,\cdots,\mbox{$[\frac{n}{2m}]+1$},
\end{equation}
which yields
\begin{align}\label{equ4.50-01}
			\left|I_1(t,x,y)\right| \lesssim |t|^{-\frac{1}{2m}}
			\lesssim |t|^{-\frac{1}{2m}}\left(1+t^{-\frac{1}{2 m}}|x-y|\right)^{-\frac{n(m-1)}{2 m-1}}.
\end{align}

$I_2(t,x,y)$ can be written as a linear combination of 
$$\int_{0}^{1} \int_{0}^{1}\int_{0}^{+\infty}e^{-\i t\lambda^{2m}+\i \lambda s_{1}^{p}|y|- \i \lambda s_{2}^q|x|} 
T_{2,2m-\frac{n}{2},2m-\frac{n}{2}}^{ \pm}(\lambda, s_{1}, s_{2}, x, y) \chi(\lambda^{2m}) \d \lambda \d s_{1} \d s_{2},$$
where $T_{2, 2m-\frac{n}{2},2m-\frac{n}{2}}(\lambda, s_{1}, s_{2}, x, y)$ is defined by
$$
\left\langle\left(\Gamma_{ 2m-\frac{n}{2},  2m-\frac{n}{2}}^{+}(\lambda)-\Gamma_{ 2m-\frac{n}{2},  2m-\frac{n}{2}}^{-}(\lambda)\right) k_{2m-\frac{n}{2},1,p}^{+}\left(\lambda, s_{1}, \cdot, y\right) , \,k_{2m-\frac{n}{2},1,q}^{+}\left(\lambda, s_{2},\cdot, x\right)\right\rangle\lambda^{-1}.$$
It follows from \eqref{equ4.3} that  
\begin{equation*}
\sup_{x,y,s_1,s_2} \left|\partial_{\lambda}^{\gamma}T_{2,2m-\frac{n}{2},2m-\frac{n}{2}}^{\pm}(\lambda,s_{1},s_{2},x,y) \right|\les \lambda^{-\gamma}, \quad \gamma=0,\cdots,\mbox{$[\frac{n}{2m}]+1$}.
\end{equation*}
This immediately yields the decay estimate 
\begin{align}\label{equ4.50-02}
			\left|I_2(t,x,y)\right| \lesssim |t|^{-\frac{1}{2m}}
			\lesssim |t|^{-\frac{1}{2m}}\left(1+t^{-\frac{1}{2 m}}|x-y|\right)^{-\frac{n(m-1)}{2 m-1}}.
\end{align}
Finally, the term  $I_3$ can be handled using the same method as for  $I_1$, which completes the argument.

 In summary, the proof of Theorem \ref{theorem4.0} is completed by considering the different resonance types:
\begin{enumerate}
    \item For $0 \leq \k \leq \mathbf{k}_c$, the estimate \eqref{equ4.1.2.0}  is established through the combination of \eqref{equ4.23.0}, \eqref{equ4.22.111-0}, and \eqref{equ4.43}.
    \item For $\mathbf{k}_c < \k \leq m_n$,  the estimate \eqref{equ4.1.2.0}   follows from \eqref{equ4.23.0000}, \eqref{equ4.22.111-1} and \eqref{equ4.43-01}.
    \item For $\k= m_n+1$, the analysis in Case 2 in Subsections \ref{sec4.2.1}--\ref{sec4.2.2} yields the  decay estimate 
\begin{equation*}
\left|\Omega_{r}^{+, low}(t, x, y)-\Omega_{r}^{-, low}(t, x, y)\right| \lesssim |t|^{-\frac{1}{2 m}} \left(1+|t|^{-\frac{1}{2 m}}|x-y|\right)^{-\frac{n(m-1)}{2 m-1}},
\end{equation*} 
which implies the estimate \eqref{equ4.1.2.0}.
\end{enumerate}

\begin{remark}\label{rmk4.1}
			During the proof,  we have taken advantage of the possible cancellation properties of  $\Omega^{+, low}_{ r}-\Omega^{-, low}_{r}$ when zero is regular in dimensions $n\ge 2m+1$, or when zero is an eigenvalue of $H$, while in  other cases, it suffices to prove results for $\Omega^{\pm, low}_{ r}$.
			We summarize the cases in the following.
			\begin{center}
			\begin{threeparttable}
				\scalebox{0.9}{
					\begin{tabular}{|c|c|c|c|}\hline
						\diagbox{dimension}{resonance } & $0\le k\le \mathbf{k}_c$ & $\mathbf{k}_c< k\le m_n$ & $k=m_n+1$ \\ \hline
						$1\le n\le 2m-1$ & {$\Omega_{r}^{\pm, low}$} & {$\Omega_{r}^{\pm, low}$}  & {$\Omega_{r}^{+, low}-\Omega_{r}^{-, low}$}  \\ \hline
						$2m+1\le n\le 4m-1$ & {$\Omega_{r}^{+, low}-\Omega_{r}^{-, low}$} & {$\Omega_{r}^{\pm, low}$}& {$\Omega_{r}^{+, low}-\Omega_{r}^{-, low}$} \\ \hline
					\end{tabular}
				}
	
		\end{threeparttable}
        \end{center}
	\end{remark}

\section{Proof of Theorem \ref{theorem4.1} (high energy part)}\label{section4.2}
 
We begin by applying the resolvent identity
\begin{equation*}
	R^\pm(\lambda)=R_0^\pm(\lambda)-R_0^\pm(\lambda)VR_0^\pm(\lambda)+R_0^\pm(\lambda)VR_0^\pm(\lambda)VR_0^\pm(\lambda)-R_0^\pm(\lambda) VR_0^\pm(\lambda) VR^\pm(\lambda)VR_0^\pm(\lambda)
\end{equation*}
to the Stone's formula of $e^{-\mathrm{i}tH}P_{ac}(H)\tilde{\chi}(H)$. This yields the spectral decomposition
\begin{equation*}
e^{-\mathrm{i}tH}P_{ac}(H)\tilde{\chi}(H)=\Omega_0^{high}+\sum_{j=1}^2(-1)^j(\Omega_j^{+,high}-\Omega_j^{-,high})-\Omega_{r}^{+,high}+\Omega_{r}^{-,high},
\end{equation*}
where
\begin{equation*}\label{Omegak}
	\Omega_0^{high}=\frac{1}{2\pi\i}\int_0^{+\infty}e^{-\mathrm{i}t\lambda}\tilde{\chi}(\lambda)(R_0^+(\lambda)-R_0^-(\lambda))\d\lambda, 
\end{equation*}
and
\begin{equation}\label{OmegaKr}
\begin{aligned}
\Omega_{j}^{\pm,high}&=\frac{1}{2\pi\i}\int_0^{+\infty}e^{-\mathrm{i}t\lambda}\tilde{\chi}(\lambda)R_0^\pm(\lambda)(VR_0^\pm(\lambda))^j\d\lambda, \\
\Omega_{r}^{\pm,high}&=\frac{1}{2\pi\i}\int_0^{+\infty}e^{-\mathrm{i}t\lambda}\tilde{\chi}(\lambda)R_0^\pm(\lambda)VR_0^\pm(\lambda)VR^\pm(\lambda)VR_0^\pm(\lambda)\d\lambda.  
\end{aligned}
\end{equation}
These integrals converge in the weak $*$ sense when $V$ satisfies (ii) of Assumption \ref{assum1}.

The distribution kernel of $\Omega_0^{high}$ can be written as the oscillatory integral
\begin{equation}\label{eq:Omega-high}
\Omega_0^{high}(t,x,y) = \mathscr{F}^{-1}\left(\widetilde{\chi}(|\cdot|^{2m})e^{-\mathrm{i}t|\cdot|^{2m}}\right)(x-y),
\end{equation}
which follows from the Fourier representation of the spectral measure for $(-\Delta)^m$. As shown in \cite[Lemma 2.1]{HHZ}, this  kernel satisfies the pointwise estimate
\begin{equation}\label{eq:Omega-est}
\left|\Omega_0^{high}(t,x,y)\right| \lesssim |t|^{-\frac{n}{2m}}\left(1 + t^{-\frac{1}{2m}}|x-y|\right)^{-\frac{n(m-1)}{2m}}.
\end{equation}
 We are left to consider $\Omega_{j}^{\pm,high}(t,x,y)$, $j=1,2$ and $\Omega_{r}^{\pm,high}(t,x,y)$. In what follows, we assume that
\begin{equation}\label{eq-V-815-1}
    |V(x)|\lesssim\langle x\rangle^{-(n+\frac 32)-}, \qquad x\in \R^n,
\end{equation}
which is weaker than (ii) in Assumption \ref{assum1}.
 Theorem \ref{theorem4.1} follows from  \eqref{eq:Omega-est}, together with the following  two propositions.
\begin{proposition}\label{pro-est-for-omega-j}
The integral kernels of $\Omega_{j}^{\pm,high}$ $(j=1,2)$ satisfy 
\begin{equation}\label{eq-est-for-omega-j}
		|\Omega_{j}^{\pm,high}(t,x,y)|\lesssim|t|^{-\frac{n}{2m}}(1+|t|^{-\frac{1}{2m}}|x-y|)^{-\frac{n(m-1)}{2m-1}},\quad|t|>0,\,x,y\in\mathbb{R}^n.
\end{equation}
\end{proposition}
\begin{proposition}\label{pro-5-1-815}
The integral kernels of $\Omega_{r}^{\pm,high}$ satisfy  
\begin{equation}\label{estOmegaK}
		|\Omega_{r}^{\pm,high}(t,x,y)|\lesssim|t|^{-\frac{n}{2m}}(1+|t|^{-\frac{1}{2m}}|x-y|)^{-\frac{n(m-1)}{2m-1}},\quad|t|>0,\,x,y\in\mathbb{R}^n.
\end{equation}
\end{proposition}
In order to prove the above propositions, we need the following explicit    representation for the free resolvent kernel  $R_0^\pm(\lambda)(x-y)$. 
\begin{lemma}
\begin{equation}\label{eq:repre for R}
R_0^\pm(\lambda)(x-y)= e^{\pm \i \lambda |x-y|}\lambda^{\frac{n+1}{2}-2m} \frac{\mathcal{K}^{\pm}(\lambda|x-y|)}{|x-y|^{\frac{n-1}{2}}},    
\end{equation}
where  the functions $\mathcal{K}^{\pm}(z)$ satisfy 
\begin{equation}\label{eq:est for mathcalK}
  \left|\frac{d^l}{dz^l}\mathcal{K}^{\pm}(z) \right|\les_{l}|z|^{-l} \quad  \text{for}~~ l\in \N_0 ~\text{and}~ z\in\mathbb{R}\setminus\{0\}.
\end{equation}
\end{lemma}
\begin{proof}
The results follow immediately by applying \eqref{eq2.8} when $\lambda|x-y| \le  1/2$ and \eqref{eq:exp to infinity} when $\lambda|x-y| >1/2$.
\end{proof}

The following oscillatory  integral estimates (see \cite[Lemma 2.1]{HHZ}) will also be used.

\begin{lemma}\label{lemmaA.222}
	Suppose $K\in\mathbb{N}_+$,  and consider the oscillatory integral
	\begin{equation*}
		I(t,x)=\int_0^{+\infty}e^{-\i (t\lambda^{2m}+x\lambda)}\psi(\lambda)\d\lambda,\quad t\neq0,\,x\in\mathbb{R},
	\end{equation*}
 where $\psi\in C^K((\frac{\lambda_0}{2},+\infty))$ such that 
\begin{equation}\label{eq:bound for osci lemma}
 \left|\frac{d^j}{d\lambda^j} \psi(\lambda)\right|\le C_j|\lambda|^{b-j} \quad \text{ for $j=0,1,\cdots,K$}.   
\end{equation}
Denoted by $\mu_b=\frac{m-1-b}{2m-1}$, we have the following estimates.

\noindent 1) If $b\in[m-1-K(2m-1),2Km-1)$, then
\begin{equation}\label{longime}
	|I(t,x)|\lesssim\begin{cases}
		|t|^{-\frac 12+\mu_{b}}|x|^{-\mu_{b}},\quad&|t|\gtrsim1,\,|t|^{-1}|x|\gtrsim1,\\
		|t|^{-K},&|t|\gtrsim1,\,|t|^{-1}|x|\ll 1.
	\end{cases}
\end{equation}

\noindent 2) If $b\in[-\frac 12,2Km-1)$, then
\begin{equation}\label{shorttime}
	|I(t,x)|\lesssim|t|^{-\frac{1+b}{2m}}\left(1+|t|^{-\frac{1}{2m}}|x|\right)^{-\mu_{b}},\quad0<|t|\lesssim1,\,x\in\mathbb{R}.
\end{equation}
\end{lemma}

We now prove Proposition \ref{pro-est-for-omega-j}. 
\begin{proof}[Proof of Proposition \ref{pro-est-for-omega-j}]
We present the proof for $\Omega_{1}^{\pm,\text{high}}$ only, as the case for $\Omega_{2}^{\pm,\text{high}}$ follows analogously.
We start by using \eqref{eq:repre for R} to express $\Omega_{j}^{\pm,high}(t,x,y)$ in  integral form
$$\Omega_{j}^{\pm,high}(t,x,y)=\frac{1}{2\pi\i}
\int_{\R^n}\int_0^{+\infty}e^{-\mathrm{i}t\lambda\pm \i \lambda\rho}\tilde{\chi}(\lambda) \frac{\mathcal{K}^{\pm}(\lambda|x-z|)\, V(z)\,\mathcal{K}^{\pm}(\lambda|z-y|)}{|x-z|^{\frac{n-1}{2}}|z-y|^{\frac{n-1}{2}}}\lambda^{n-2m}\d\lambda \d z,$$
where $\rho=|x-z|+|z-y|$.
This representation allows us to analyze the problem by estimating an oscillatory integral of the form treated in Lemma \ref{lemmaA.222}.

Since $n - 2m \leq \frac{n-1}{2}$, by \eqref{eq:est for mathcalK},  the functions $\mathcal{K}^{\pm}(\lambda|x-z|)\mathcal{K}^{\pm}(\lambda|z-y|)\lambda^{n-2m}$ satisfy the uniform bound 
$$
\sup_{x,y,z}\left|\partial_{\lambda}^l\mathcal{K}^{\pm}(\lambda|x-z|)\mathcal{K}^{\pm}(\lambda|z-y|)\lambda^{n-2m}\right| \les_l \lambda^{\frac{n-1}{2}-l}, \quad \lambda>\lambda_0/2, ~ l=0,1\ldots. 
$$
Next, we shall apply Lemma \ref{lemmaA.222} with $b=\frac{n-1}{2}$ to establish pointwise decay estimates. 

For the long-time decay estimates, we decompose $\mathbb{R}^n$ into two regions $D_1=\{z; \rho \gtrsim|t|\}$ and $D_2=\{z; \rho \les|t|\}$ with  $\rho=|x-z|+|z-y|$.
If $|t|\gtrsim1,\,z\in D_1$, then by Lemma \ref{lemmaA.222}, 
\begin{align*}
|\Omega_{1}^{\pm,high}(t,x,y)|\lesssim & |t|^{-\frac12+\mu_{\frac{n-1}{2}}}\int_{D_1}   \frac{\rho^{-\mu_{\frac{n-1}{2}}}\, |V(z)|}{|x-z|^{\frac{n-1}{2}}|z-y|^{\frac{n-1}{2}}}\d z \\
\les &|t|^{-\frac{n}{2m}}(1+|t|^{-\frac{1}{2m}}|x-y|)^{-\frac{n(m-1)}{2m-1}}  \int_{\R^n}   \frac{|V(z)|}{|x-z|^{\frac{n-1}{2}}} +  \frac{|V(z)|}{|z-y|^{\frac{n-1}{2}}}\d z  \\
\les & |t|^{-\frac{n}{2m}}(1+|t|^{-\frac{1}{2m}}|x-y|)^{-\frac{n(m-1)}{2m-1}}.
\end{align*}
In the above inequalities, we have used Lemma \ref{lem3.10} and the facts 
$$\frac{\rho^{-\mu_{\frac{n-1}{2}}}}{|x-z|^{\frac{n-1}{2}}|z-y|^{\frac{n-1}{2}}} \les \rho ^{-\frac{n(m-1)}{2m-1}} \left(|x-z|^{-\frac{n-1}{2}}+|z-y|^{-\frac{n-1}{2}} \right)$$
and 
$$|t|^{-\frac12+\mu_{\frac{n-1}{2}}} \rho ^{-\frac{n(m-1)}{2m-1}} \les |t|^{-\frac{n}{2m}} (1+|t|^{-\frac{1}{2m}}\rho)^{-\frac{n(m-1)}{2m-1}} \les  |t|^{-\frac{n}{2m}}(1+|t|^{-\frac{1}{2m}}|x-y|)^{-\frac{n(m-1)}{2m-1}},$$
which in turn follows from  the triangle inequality and the condition $|t|^{-\frac{1}{2m}}\rho \gtrsim 1$ in $D_1$.

If $|t|\gtrsim1,\,z\in D_2$, then by Lemma \ref{lemmaA.222},
\begin{align*}
|\Omega_{1}^{\pm,high}(t,x,y)|\lesssim & |t|^{-[\frac{n}{2m}]-1}\int_{D_2}   \frac{\, |V(z)|}{|x-z|^{\frac{n-1}{2}}|z-y|^{\frac{n-1}{2}}}\d z \\
\les &|t|^{-\frac{n}{2m}}(1+|t|^{-\frac{1}{2m}}|x-y|)^{-\frac{n(m-1)}{2m-1}}  \int_{\R^n}  \left( \frac{|V(z)|}{|x-z|^{\frac{n-1}{2}}} +  \frac{|V(z)|}{|z-y|^{\frac{n-1}{2}}}\right)\d z  \\
+& |t|^{-\frac{n}{2m}}(1+|t|^{-\frac{1}{2m}}|x-y|)^{-\frac{n(m-1)}{2m-1}}  \int_{\R^n} 
\frac{|V(z)|}{|x-z|^{\frac{n-1}{2}}|z-y|^{\frac{n-1}{2}}} \d z \\
\les & |t|^{-\frac{n}{2m}}(1+|t|^{-\frac{1}{2m}}|x-y|)^{-\frac{n(m-1)}{2m-1}},
\end{align*}
where we have used the following two  facts: (i) If $|t|^{-\frac{1}{2m}}|x-y| \geq 1$, then
\begin{align*}
|t|^{-\left[\frac{n}{2m}\right]-1}{|x-z|^{-\frac{n-1}{2}}|z-y|^{-\frac{n-1}{2}}} 
\lesssim|t|^{-\frac{n}{2m}}\big(1+|t|^{-\frac{1}{2m}}|x-y|\big)^{-\frac{n(m-1)}{2m-1}} \left(|x-z|^{-\frac{n-1}{2}} + |z-y|^{-\frac{n-1}{2}}\right),
\end{align*}
and (ii) if $|t|^{-\frac{1}{2m}}|x-y| < 1$, then
\begin{align*}
|t|^{-\left[\frac{n}{2m}\right]-1} 
&\lesssim |t|^{-\frac{n}{2m}}\big(1+|t|^{-\frac{1}{2m}}|x-y|\big)^{-\frac{n(m-1)}{2m-1}}.
\end{align*}

For the short-time decay estimates, we decompose $\mathbb{R}^n$ into two regions $D_3=\{z; \rho \gtrsim|t|^{\frac{1}{2m}}\}$ and $D_4=\{z; \rho \les |t|^{\frac{1}{2m}}\}$ with $\rho=|x-z|+|z-y|$.
For $|t|\les 1$, we apply Lemma \ref{lemmaA.222} with $b=\frac{n-1}{2}$ and obtain 
\begin{align*}
|\Omega_{1}^{\pm,high}(t,x,y)|\lesssim & |t|^{-\frac{n+1}{4m}}\int_{\R^n} (1+|t|^{-\frac{1}{2m}}\rho)^{-\mu_{\frac{n-1}{2}}}  \frac{|V(z)|}{|x-z|^{\frac{n-1}{2}}|z-y|^{\frac{n-1}{2}}}\d z  \\
\les &|t|^{-\frac{n}{2m}}(1+|t|^{-\frac{1}{2m}}|x-y|)^{-\frac{n(m-1)}{2m-1}}  \int_{D_3}   \frac{|V(z)|}{|x-z|^{\frac{n-1}{2}}} +  \frac{|V(z)|}{|z-y|^{\frac{n-1}{2}}}\d z  \\
+& |t|^{-\frac{n}{2m}}(1+|t|^{-\frac{1}{2m}}|x-y|)^{-\frac{n(m-1)}{2m-1}}  \int_{D_4}   \frac{|V(z)|}{|x-z|^{\frac{n-1}{2}}|z-y|^{\frac{n-1}{2}}} \d z \\
\les & |t|^{-\frac{n}{2m}}(1+|t|^{-\frac{1}{2m}}|x-y|)^{-\frac{n(m-1)}{2m-1}}.
\end{align*}

Therefore, we complete the proof of  Proposition \ref{pro-est-for-omega-j}.
\end{proof}

The final step involves estimating the kernel of $\Omega_{r}^{\pm,high}$, which admits the oscillatory integral representation
\begin{equation}\label{omegaKr}
\Omega_{r}^{\pm,high}(t,x,y) = \frac{m}{\pi\mathrm{i}}\int_0^\infty e^{-\mathrm{i}t\lambda^{2m} \pm \lambda(|x|+|y|)} \widetilde{\chi}(\lambda^{2m})  T_\pm(\lambda,x,y)\d\lambda,
\end{equation}
where $T_\pm(\lambda,x,y)$ is the sesquilinear form 
\begin{equation*}
T_\pm(\lambda,x,y)=\Big\langle R_0^\pm(\lambda^{2m})VR^\pm(\lambda^{2m})V\big(R_0^\pm(\lambda^{2m})(\cdot-y)e^{\mp\mathrm{i}\lambda|y|}\big), VR_0^\mp(\lambda^{2m}(\cdot-x)e^{\pm\mathrm{i}\lambda|x|}\big) \Big\rangle \lambda^{2m-1},
\end{equation*}
defined through the duality between $L^2_{-\frac12-}$ and $L^2_{\frac12+}$.
In order to prove \eqref{estOmegaK}, we need to first establish pointwise estimates for $T_\pm(\lambda,x,y)$. To this end, we need the following lemmas, the first of which is on the limiting absorption principle.

\begin{lemma}\label{lmdR}
For a fixed $s\in\mathbb{N}_0$, we have 
\begin{equation}\label{eq1.1.1.1}
\big\|\partial_\lambda^s (R_0^\pm(\lambda^{2m}))\big\|_{L_{s+\frac12+}^{2}\rightarrow L_{-(s+\frac12)-}^{2}}\lesssim\lambda^{-(2m-1)},\quad\lambda\gtrsim1,
\end{equation}
and if $|V(x)|\lesssim\langle x\rangle^{-(s+1)-}$, we also have 
\begin{equation}\label{eq1.1.1.1-1}
\big\|\partial_\lambda^s (R^\pm(\lambda^{2m}))\big\|_{L_{s+\frac12+}^{2}\rightarrow L_{-(s+\frac12)-}^{2}}\lesssim\lambda^{-(2m-1)},\quad\lambda\gtrsim1.
\end{equation}
\end{lemma}
\begin{proof}
It is well known (see \cite{KK}) that
\begin{equation*}
\|\partial_z^s\mathfrak{R}_0(z)\|_{L_{s+\frac12+}^{2}\rightarrow L_{-(s+\frac12)-}^{2}}\lesssim z^{\frac{-1-s}{2}},\quad|z|\gtrsim1,\,z\in\mathbb{C}\setminus [0,+\infty).
\end{equation*}
Thus, the first statement follows by the limiting absorption principle and \eqref{equ2.1}.
		
Next, by the resolvent identity
$$R^\pm(\lambda^{2m})=(I+R_0^\pm(\lambda^{2m})V)^{-1}R_0^\pm(\lambda^{2m}),$$
$\partial_\lambda^s(R^\pm(\lambda^{2m}))$ can be written as a linear combination of
$$\prod_{j=1}^{p}\Big(I+R_0^\pm(\lambda^{2m})V)^{-1}\partial_\lambda^{s_j}(R_0^\pm(\lambda^{2m}))V\Big)(I+R_0^\pm(\lambda^{2m})V)^{-1}\partial_\lambda^{s_{p+1}}(R_0^\pm(\lambda^{2m})),$$
where $0\leq p\leq s$ and $\sum_{j=1}^{p+1}s_j=s$. Note that $(I+R_0^\pm(\lambda^{2m})V)^{-1}$ is uniformly bounded for $\lambda$ in $L^2_{-(s_j+\frac12)-}$ if $|V(x)|\lesssim\langle x\rangle^{-(s_j+1)-},$ and
$$\|\partial_\lambda^{s_j}(R_0^\pm(\lambda^{2m}))V\|_{L^2_{-(s_{j+1}+\frac12)-}\rightarrow L^2_{-(s_{j}+\frac12)-}}
\lesssim\lambda^{-(2m-1)},\quad\lambda\gtrsim 1,$$
by \eqref{eq1.1.1.1} if $|V(x)|\lesssim\langle x\rangle^{-(s_j+s_{j+1}+1)-}$. Then we have \eqref{eq1.1.1.1-1}.
\end{proof}

\begin{lemma}\label{lmdTlambda}
 If $|V(x)|\lesssim\langle x\rangle^{-(n+\frac 32)-}$, then
 \begin{equation*}
 \left|\partial_\lambda^l T_\pm(\lambda,x,y)\right|\lesssim\lambda^{\frac{n-1}{2}-l}\langle x\rangle^{-\frac{n-1}{2}}\langle y\rangle^{-\frac{n-1}{2}},\quad x,y\in\mathbb{R}^n,\,\lambda\gtrsim1.
 \end{equation*}
 for $l=0,1,\ldots,\frac{n+1}{2}$.
\end{lemma}

\begin{proof}
By using  the expression \eqref{eq:repre for R}, we  express  $T_\pm(\lambda,x,y)$ in the following form 
\begin{equation*}
 T_\pm(\lambda,x,y)=\lambda^{n-2m}\Big\langle VR_0^\pm(\lambda^{2m})VR^\pm(\lambda^{2m})V \frac{e^{\pm\mathrm{i}\lambda(|y-\cdot|-|y|)} \mathcal{K}^{\pm}(\lambda|y-\cdot|)}{|y-\cdot|^{\frac{n-1}{2}}},~  \frac{e^{\mp\mathrm{i}\lambda(|x-\cdot|-|x|)} \mathcal{K}^{\mp}(\lambda|y-\cdot|)}{|x-\cdot|^{\frac{n-1}{2}}} \Big\rangle.
\end{equation*}

On the one hand, it follows from \eqref{eq:est for mathcalK} and Lemma \ref{lem3.10} that for $\lambda \ \gtrsim1$,
\begin{align*}
\left\|\partial_\lambda^l \frac{e^{\pm\mathrm{i}\lambda(|y-\cdot|-|y|)} \mathcal{K}^{\pm}(\lambda|y-\cdot|)}{|y-\cdot|^{\frac{n-1}{2}}}\right\|_{L^2_{-\frac{n+1}{2}-l}} 
\les & \left\| \frac{\langle \cdot \rangle^l}{|y-\cdot|^{\frac{n-1}{2}}}\right\|_{L^2_{-\frac{n+1}{2}-l}} \\ 
\les & \langle y \rangle^{-\frac{n-1}{2}}.
\end{align*}

On the other hand, the estimate \eqref{eq1.1.1.1-1} shows that, for $\lambda\gtrsim1$,
$$ \Big\| \partial_\lambda^s \big(VR_0^\pm(\lambda^{2m})VR^\pm(\lambda^{2m})V\big) \Big\|_{L_{-\frac{n+1}{2}-l}^{2}\rightarrow L_{\frac{n+1}{2}+l}^{2}}\lesssim\lambda^{-(4m-2)},\quad~ l=0,1,2,$$
    holds provided $|V(x)|\les \langle x\rangle^{-(\frac n2+s+l+1)-}$. Thus applying Leibniz's rule, we derive that, for $ \lambda\gtrsim1$, 
\begin{align*}
\left|\partial_\lambda^l T_\pm(\lambda,x,y)\right|\lesssim &
\lambda^{n-2m}
\sum_{l_1+l_2+l_3\le l} \left\|\partial_\lambda^{l_1} \frac{e^{\pm\mathrm{i}\lambda(|y-\cdot|-|y|)} \mathcal{K}^{\pm}(\lambda|y-\cdot|)}{|y-\cdot|^{\frac{n-1}{2}}}\right\|_{L^2_{-\frac{n+1}{2}-l_1}} 
 \\
\times\Big\|\partial_\lambda^{l_2} \big(VR_0^\pm&(\lambda^{2m})VR^\pm(\lambda^{2m})V\big)\Big\|_{L_{-\frac{n+1}{2}-l_1}^{2}\rightarrow L_{\frac{n+1}{2}+l_3}^{2}}
\left\|\partial_\lambda^{l_3} \frac{e^{\mp\mathrm{i}\lambda(|x-\cdot|-|x|)} \mathcal{K}^{\pm}(\lambda|x-\cdot|)}{|x-\cdot|^{\frac{n-1}{2}}}\right\|_{L^2_{-\frac{n+1}{2}-l_3}} \\
\lesssim & \lambda^{\frac{n-1}{2}-l} \langle x\rangle^{-\frac{n-1}{2}}\langle y\rangle^{-\frac{n-1}{2}},
\end{align*}
where we have used the fact $n-2m-(4m-2)\le 1-2m\le  \frac{n-1}{2}-l$ for all $l\le \frac{n+1}{2}$. 
This completes the proof.
\end{proof}

Now we are in the position  to prove Proposition \ref{pro-5-1-815}.

\begin{proof}[Proof of Proposition \ref{pro-5-1-815}]

Building upon Lemma \ref{lmdTlambda}, we apply estimate \eqref{longime} from Lemma \ref{lemmaA.222} to \eqref{omegaKr} with the parameter  $b=\frac{n-1}{2}$ to derive the following long-time decay estimates
\begin{equation*}
\begin{split}
|\Omega_{r}^{\pm,high}(t,x,y)|\lesssim&\begin{cases}
|t|^{-\frac12+\mu_{\frac{n-1}{2}}}(|x|+|y|)^{-\mu_{\frac{n-1}{2}}}\langle x\rangle^{-\frac{n-1}{2}}\langle y\rangle^{-\frac{n-1}{2}},\quad&|t|\gtrsim1,\,|x|+|y|\gtrsim|t|,\\
|t|^{-[\frac{n}{2m}]-1}\langle x\rangle^{-\frac{n-1}{2}}\langle y\rangle^{-\frac{n-1}{2}},&|t|\gtrsim1,\,|x|+|y|\lesssim|t|.
\end{cases}
\end{split}
\end{equation*}
A direct computation shows that for $|t| \gtrsim 1$ and $|x| + |y| \gtrsim |t|$, which implies $|x| + |y| \gtrsim |t|^{\frac{1}{2m}}$, we have
\begin{align*}
|t|^{-\frac12+\mu_{\frac{n-1}{2}}}(|x|+|y|)^{-\mu_{\frac{n-1}{2}}}\langle x\rangle^{-\frac{n-1}{2}}\langle y\rangle^{-\frac{n-1}{2}} 
\les & |t|^{-\frac12+\mu_{\frac{n-1}{2}}}(|x|+|y|)^{-\mu_{\frac{n-1}{2}}-\frac{n-1}{2}} \\ 
=& |t|^{-\frac{n}{2m}} \left( |t|^{-\frac{1}{2m}}(|x|+|y|)\right)^{-\frac{n(m-1)}{2m-1}} \\
\les& |t|^{-\frac{n}{2m}} \left(1+|t|^{-\frac{1}{2m}}(|x|+|y|)\right)^{-\frac{n(m-1)}{2m-1}} \\
\le & |t|^{-\frac{n}{2m}} \left(1+|t|^{-\frac{1}{2m}}|x-y|\right)^{-\frac{n(m-1)}{2m-1}}.
\end{align*}
Similarly, if  $|t|\gtrsim1$ and $|x|+|y|\les |t|$, then 
\begin{align*}
|t|^{-[\frac{n}{2m}]-1}\langle x\rangle^{-\frac{n-1}{2}}\langle y\rangle^{-\frac{n-1}{2}} \les &
\begin{cases}
 |t|^{-1}\big(|t|^{-1}(|x|+|y|)\big)^{-\max\{\mu_{\frac{n-1}{2}}, 0\}}\langle x\rangle^{-\frac{n-1}{2}}\langle y\rangle^{-\frac{n-1}{2}} \quad &\text{if}~ |x|+|y|> |t|^{\frac {1}{2m}} \\ 
 |t|^{-\frac{n}{2m}} \quad &\text{if}~ |x|+|y|\le |t|^{\frac {1}{2m}} \end{cases} \\
 \les& |t|^{-\frac{n}{2m}} \left(1+|t|^{-\frac{1}{2m}}(|x|+|y|)\right)^{-\frac{n(m-1)}{2m-1}} \\
\le & |t|^{-\frac{n}{2m}} \left(1+|t|^{-\frac{1}{2m}}|x-y|\right)^{-\frac{n(m-1)}{2m-1}}.
\end{align*}
Combining the above three inequalities gives 
\begin{equation*}
\begin{split}
|\Omega_{r}^{\pm,high}(t,x,y)|\lesssim|t|^{-\frac{n}{2m}}(1+|t|^{-\frac{1}{2m}}|x-y|)^{-\frac{n(m-1)}{2m-1}},\quad|t|\gtrsim1,\,x,y\in\mathbb{R}^n.
\end{split}
\end{equation*}

Moreover, by applying Lemma \ref{lmdTlambda} and invoking estimate \eqref{shorttime} from  Lemma \ref{lemmaA.222} to \eqref{omegaKr} with  $b=\frac{n-1}{2}$,  we obtain the following short time estimate 
\begin{align*}\label{eq-Ome-8-15-2}
	\begin{split}
		|\Omega_{r}^{\pm,high}(t,x,y)|\lesssim&|t|^{-\frac{1+\frac{n-1}{2}}{2m}}\left(1+|t|^{-\frac{1}{2m}}(|x|+|y|)\right)^{-\mu_{\frac{n-1}{2}}}\langle x\rangle^{-\frac{n-1}{2}}\langle y\rangle^{-\frac{n-1}{2}}\\
		\lesssim&|t|^{-\frac{n}{2m}}(1+|t|^{-\frac{1}{2m}}|x-y|)^{-\frac{n(m-1)}{2m-1}},\quad0<|t|\lesssim1,\,x,y\in\mathbb{R}^n.
	\end{split}
\end{align*}
Combining the above two estimates yields \eqref{estOmegaK}. Therefore, the proof of Proposition \ref{pro-5-1-815} is complete.
\end{proof}

\section{Proof of Corollary \ref{cor-Lp}}\label{section4.4}

\noindent\emph{Proof of  statement \emph{(\romannumeral1)}.}
We first consider the case that $({1\over p},{1\over q})$ lies in the line segment BC but $(p,q)\ne(1, \tau_m)$, i.e., $p=1$ and $q\in(\tau_m, +\infty]$ with $\tau_m=\frac{2m-1}{m-1}$. In this case one checks that $\frac{n(m-1)}{2 m-1}q>n$.
If we set
\begin{equation*}
  I(t,x-y)=|t|^{-\frac{n}{2 m}}\left(1+|t|^{-\frac{1}{2 m}}|x-y|\right)^{-\frac{n(m-1)}{2 m-1}},
\end{equation*}
then by \eqref{equ1.2.1-231} of Theorem \ref{thm1.1},  \eqref{equ1.2.1-233} of Theorem \ref{thm1.1-723}, and by the Young's inequality, we have
\begin{align*}
\|e^{-\i tH}P_{ac}(H)\|_{L^1-L^q}&\lesssim\|I(t,\cdot)\|_{L^q}\\
&\lesssim Ct^{-\frac{n}{2m}}\left(\int_{\mathbb{R}^n}{(1+|t|^{-1/{2m}}|x|)^{-\frac{n(m-1)}{2 m-1}q}\,dx}\right)^{\frac 1q}\lesssim |t|^{-\frac{n}{2m}(1-\frac{1}{q})}.
\end{align*}
At the endpoint $(p,q)=(1,\tau_m)$, the above estimate with $L^{\tau_m}$ replaced by $L^{\tau_m,\infty}$ (or $L^1$ replaced by $H^1$) follows from the weak Young's inequality (see e.g. \cite[p.22]{Gro}) and the boundedness of the Riesz potential $f\mapsto|\cdot|^{-\frac{n(m-1)}{2 m-1}}*f$. 

Next, if $(\frac1p,\frac1q)$ lies in the interior of triangle ABC, we deduce \eqref{equ-Lp-Lq} from the Riesz-Thorin interpolation and the fact that
$\| e^{-\i tH}\|_{L^2-L^2}=1$.
On the other hand, when $(\frac1p,\frac1q)$ lies in the edge AB, the above estimate follows from the Marcinkiewicz interpolation theorem (see e.g. \cite[p.31]{Gro}). Therefore estimate \eqref{equ3.5} is valid if $(\frac1p,\frac1q)$ lies in the triangle $ABC$, while the case of triangle $ADC$ follows by duality.

\noindent\emph{Proof of  statement \emph{(\romannumeral2)}.}
If $\alpha>0$, then the smoothing operator $H^\frac{\alpha}{2m} e^{-\i tH}P_{ac}(H)$ can be defined through spectral calculus, and can also be split into low and high energy parts.
\begin{equation*}
	\begin{split}
			H^\frac{\alpha}{2m} e^{-\i tH}P_{ac}(H)&=\frac{1}{2\pi i} \int_{0}^{+\infty} \lambda^\frac{\alpha}{2m} e^{-\i  t \lambda}\left(R^{+}(\lambda)-R^{-}(\lambda)\right) (\chi(\lambda)+\tilde{\chi}(\lambda))\, \d \lambda,\\
		& :=H^\frac{\alpha}{2m}e^{-\i tH}\chi(H)P_{ac}(H)+H^\frac{\alpha}{2m}e^{-\i tH}\tilde{\chi}(H)P_{ac}(H),
	\end{split}
\end{equation*}
where $\chi$ is defined by \eqref{equ4.cutoff} and $\tilde{\chi}(\lambda)=1-\chi(\lambda)$.

Let $K^{low}_{\alpha}(t,x,y)$ and $K^{high}_{\alpha}(t,x,y)$ be kernels respectively of the low and high energy parts of $H^\frac{\alpha}{2m} e^{-\i tH}P_{ac}(H)$.
It is now reasonable to discuss the kernels for these two parts in a completely parallel way, because the only relevant changes are the function classes to which the amplitude functions in the oscillatory integrals in $\lambda$ belong, and similar changes happen after whatever further decomposition or treatment. The oscillatory integral estimates Lemma \ref{lemmaA.2} and Lemma \ref{lemmaA.222} are then still applicable with a few adjustments, where one should be slightly more careful when checking the differentiability of the amplitude functions in every detail.

It turns out that a minor change of the decay assumption on $V$ suffices to give the following result,
%
%
%
\begin{equation*}\label{estalphalow}
	\left|K_\alpha^{low}(t,x,y)\right|\lesssim(1+|t|)^{-h(m,n,\mathbf{k})-\frac{\alpha}{2m}}(1+|t|^{-\frac{n+\alpha}{2m}})\left(1+|t|^{-\frac{1}{2m}}|x-y|\right)^{-\frac{n(m-1)-\alpha}{2m-1}},\quad t\neq0,\,x,y\in\mathbb{R}^n,
\end{equation*}
and we only point out a few places most relevant. Under the stronger decay assumption, we are able to show the following:
\begin{itemize}
	\item \eqref{equ4.1.1} in Lemma \ref{prop4.1} holds for $l=0,\cdots,\frac{n+1}{2}$.
	
	\item In Proposition \ref{lemma4.2}, the  statements \eqref{equ4.3}, \eqref{equ4.4}, \eqref{equ4.4.4}, \eqref{equ4.4.5}, \eqref{equ4.34.2} and \eqref{equ4.34.3} remain valid in their $\frac{n+1}{2}$-times differentiable versions. 

	\item All relevant details when the above changes are applied in the proof for the low energy part in section \ref{section-4} are correspondingly adjusted, which relates to the appropriate choice of $K$ in the application of Lemma \ref{lemmaA.2}, and this is why the new  $\beta$  is needed where relevant estimates take place.
\end{itemize}

To establish the high energy estimate
\begin{equation*}\label{estalphahigh}
	\left|K_\alpha^{high}(t,x,y)\right|\lesssim|t|^{-\frac{n+\alpha}{2m}}\left(1+|t|^{-\frac{1}{2m}}|x-y|\right)^{-\frac{n(m-1)-\alpha}{2m-1}},\quad t\neq0,\,x,y\in\mathbb{R}^n,
\end{equation*}
we only remark that to prove a relevant version of the bounds \eqref{eq-est-for-omega-j} and \eqref{estOmegaK}, we need to apply Lemma \ref{lmdTlambda} with $l=\frac{n+1}{2}$ where the decay assumption of $V$ will allow such change, and the rest of necessary adjustments are obvious with the application of Lemma \ref{lemmaA.222}.

\begin{appendix}
	\renewcommand{\appendixname}{Appendix\,\,}

\section{Technical lemmas and proof of Proposition \ref{lemma4.2}}\label{app-002-8-9}



To prove Proposition \ref{lemma4.2}, we first establish two technical lemmas concerning functions with certain vanishing moments.

\begin{lemma}\label{lemmaA.1}
Let $f\in L_{\sigma}^2({\R}^n)$ with $\sigma>\max\{j+1,p\}+n/2$ for some  $\frac{1-n}{2}\leq p\in\mathbb{Z}$ and $j\in\mathbb{N}_0$. Suppose $\left\langle x^\alpha, f(x) \right\rangle=0$ for all $|\alpha|\le j$. Then, we have
\begin{equation}\label{eqA.1.1}
	\left| \int f(y)|x-y|^p\d y\right| \lesssim\left\|f \right\|_{L_{\sigma}^2} \left\langle  x\right\rangle ^{p-j-1},\quad x\in\mathbb{R}^n.
\end{equation}
\end{lemma}

\begin{proof}
If $|x|\leq1$, \eqref{eqA.1.1} holds by Schwartz inequality and the assumption $p-\sigma<-\frac{n}{2}$. Therefore it suffices to
consider the case $|x|\geq1$ and we divide the proof into three cases. \\

\noindent\textbf{Case 1: $p\geq0$ is even.} The estimate \eqref{eqA.1.1} follows by applying the expansion
\begin{equation}\label{eqA.2}
|y-x|^{p} = \sum_{|\alpha|+|\beta|=p} (-1)^{|\beta|} C_{\alpha,\beta} x^{\alpha} y^{\beta},
\end{equation}
to the left-hand side of \eqref{eqA.1.1}, together with the observations that
\begin{equation*}
\begin{cases}
\left| \langle x^{\alpha} y^{\beta}, f(y) \rangle \right| \lesssim \|f\|_{L_{\sigma}^2} \langle x \rangle^{|\alpha|} \leq \|f\|_{L_{\sigma}^2} \langle x \rangle^{p-j-1}, & \text{if } |\beta| > j, \\
\langle x^{\alpha} y^{\beta}, f(y) \rangle = 0, & \text{if } |\beta| \leq j.
\end{cases}
\end{equation*} \\

\noindent\textbf{Case 2: $p\geq0$ is old.} We begin by decomposing $|x-y|^p$ into
\begin{equation*}
|x-y|^p = \left(|x-y|^p - |x||x-y|^{p-1}\right) + |x||x-y|^{p-1}.
\end{equation*}
Since $p-1$ is even, Case 1 yields
\begin{equation}\label{equ4.707}
\left|\langle |x||x-\cdot|^{p-1}, f(\cdot)\rangle\right| \lesssim \langle x\rangle^{p-j-1}.
\end{equation}

We need to use the algebraic identity
\begin{align}\label{eq:algebraic-identity}
\frac{1}{a+b} = \sum_{j=0}^{k-1}\frac{(-b)^j}{a^{j+1}} + \frac{(-b)^k}{a^k(a+b)}.
\end{align}
We also need the following algebraic identity
\begin{align}\label{eq:algebraic-identity-2}
\frac{1}{(a+b)^k} =\sum_{j=1}^{k}\frac{(a-b)}{2^{j}(a+b)^{k-j+1} a^{j}}+\frac{1}{(2a)^k},
\end{align}
and the identity 
$$|x-y|-|x|=\frac{-2xy+|y|^2}{|x-y|+|x|}.$$

Fixing $l,\ q,\ j\in \N_0$, $\gamma\in {\N_0}^n$ with $|\gamma|\le j$, and denoting $\tau=|\gamma|+l-q$, we claim that
\begin{equation}\label{claimA.1}
	\begin{aligned}
		\frac{|x|^ly^{\gamma}}{(|x|+|x-y|)^q}=\sum\limits_{s=1}\limits^{q+j-|\gamma|+1}\sum\limits_{|\alpha_1|\in\{j+1, j+2\} \atop
			|\alpha_1|+h-s=\tau} L_{s,\alpha_1}(x) \frac{|x|^hy^{\alpha_1}}{(|x|+|x-y|)^s}
		+\sum\limits_{|\alpha_2|\le  j\atop i+|\alpha_2|=\tau}L_{\alpha_2}(x)|x|^{i}y^{\alpha_2},
	\end{aligned}
\end{equation}	
where $|L_{s,\alpha_1}(x)|,\ |L_{\alpha_2}(x)|\lesssim 1$.	Rigorously speaking, we first suppose that $j - |\gamma| = 0$, then by \eqref{eq:algebraic-identity-2}, we have
\begin{equation}\label{eq:intuition identity}
\begin{aligned}
\frac{|x|^l y^\gamma}{(|x| + |x - y|)^q} 
&= \sum_{i=1}^q \frac{1}{2^{i - 1}}  \frac{|x|^{l - i} y^\gamma (|x| - |x - y|)}{2(|x| + |x - y|)^{q - i + 1}} + \frac{|x|^{l - q} y^\gamma}{2^q} \\
&= \sum_{i=1}^q \frac{1}{2^{i - 1}}  \frac{|x|^{l - i} y^\gamma (2x\cdot y - |y|^2)}{2(|x| + |x - y|)^{q - i + 2}} + \frac{|x|^{l - q} y^\gamma}{2^q} \\
&= \sum_{i=1}^q \frac{1}{2^{i - 1}} \sum_{s=1}^n \left( \frac{x_s}{|x|} \frac{|x|^{l - i + 1} y^\gamma y_s}{(|x| + |x - y|)^{q - i + 2}} - \frac{|x|^{l - i} y^\gamma y_s^2}{(|x| + |x - y|)^{q - i + 2}} \right) + \frac{|x|^{l - q} y^\gamma}{2^q}, 
\end{aligned}
\end{equation}
which implies that the assertion holds for $j - |\gamma| = 0$.
In the above expressions, $x_s$ and $y_s$ denote the $s$-th components of $x$ and $y$ respectively. Now, we assume that \eqref{claimA.1} holds when \( j - |\gamma| = 0, \cdots, \mu < j \), and consider the case where \( j - |\gamma| = \mu + 1 \). In this scenario, we apply \eqref{eq:intuition identity} to 
$$\frac{|x|^hy^{\gamma}}{(|x|+|x-y|)^s}$$
in  \eqref{claimA.1}, and note that 
\[
\frac{|x|^{l-i+1} y^{\gamma} y_s}{(|x| + |x - y|)^{q-i+2}} \quad \text{and} \quad \frac{|x|^{l-i} y^{\gamma} y_s^2}{(|x| + |x - y|)^{q-i+2}},
\]
satisfy all the assumptions required in our assertion, except that \( j - |\gamma| - 1, j - |\gamma| - 2 \leqslant \mu \). Therefore, using induction together with  proves \eqref{claimA.1} for \( j - |\gamma| = \mu + 1 \).

We expand
\begin{equation*}
\begin{aligned}
|x-y|^p - |x||x-y|^{p-1}
&= (|x-y| - |x|)|x-y|^{p-1}
=\frac{-2x\cdot y+|y|^2}{|x-y|+|x|} |x-y|^{p-1} \\
&= \sum_{|\alpha|+|\beta|=p-1} (-1)^{|\beta|} C_{\alpha,\beta} \sum_{i=1}^n \left(\frac{2x^\alpha x_i}{|x|^{|\alpha|+1}} \frac{|x|^{|\alpha|+1}y^\beta y_i}{|x|+|x-y|} - \frac{x^\alpha}{|x|^{|\alpha|}} \frac{|x|^{|\alpha|}y^\beta y_i^2}{|x|+|x-y|}\right).
\end{aligned}
\end{equation*} 
We apply \eqref{claimA.1}  to 
$$ \frac{|x|^{|\alpha|+1}y^\beta y_i}{|x|+|x-y|}\quad\text{and}\quad\frac{|x|^{|\alpha|}y^\beta y_i^2}{|x|+|x-y|}$$ 
in the above. For the term $\frac{|x|^{|\alpha|+1}y^{\beta}{y_i}}{|x|+|x-y|}$, we apply \eqref{claimA.1} with $l=|\alpha|+1$, $\gamma=\beta+e_i$ and $q=1$;
for the term $\frac{|x|^{|\alpha|}y^{\beta}{y_i}^2}{|x|+|x-y|}$, we apply \eqref{claimA.1} with $l=|\alpha|$, $\gamma=\beta+2e_i$ and $q=1$. In particular, it follows that $\tau=p$, and $h-s=p-|\alpha_1|\le p-j-1$.
Since $|\alpha_1|\in\{j+1, j+2\}$ in \eqref{claimA.1},
it follows that
\begin{equation*}
	\left| \frac{|x|^hy^{\alpha_1}}{(|x|+|x-y|)^s}\right|\le |x|^{p-j-1}\langle y\rangle^{j+1}.
\end{equation*}
By the vanishing assumption on $f$,  
$$\Big\langle \sum\limits_{|\alpha_2|\le  j\atop i+|\alpha_2|=\tau}L_{\alpha_2}(x)|x|^{i}y^{\alpha_2}, f(y) \Big\rangle=0.$$
Then we have
\begin{equation*}\label{equ4.708}
	\begin{aligned}
		\left| \left\langle |x-\cdot|^p-|x||x-\cdot|^{p-1}, \ f(\cdot) \right\rangle\right|
		\lesssim& \left\|f \right\|_{L_{\sigma}^2}\|\left\langle y\right\rangle^{-\sigma}\left\langle y\right\rangle^{j+1} \|_{L^2} \left\langle x\right\rangle^{p-j-1}\\
		\lesssim& \left\|f \right\|_{L_{\sigma}^2} \left\langle x\right\rangle^{p-j-1},
	\end{aligned}
\end{equation*}
which together with \eqref{equ4.707} yields \eqref{eqA.1.1}. \\

\noindent\textbf{Case 3:} $\frac{1-n}{2}\le  p<0$.
Set $k=-p$, we have $0<k\le  \frac{n-1}{2}$. For any fixed $j\in\N_0$, applying \eqref{eq:algebraic-identity}, we have the following identity
\begin{equation}\label{eqA.3}
	|x-y|^{-k}=\sum_{l=0}^{k-1}C_{l}\frac{(|x|-|x-y|)^{j+1}}{|x|^{j+l+1}|x-y|^{k-l}}+\sum_{l=0}^{j}C'_{l}\frac{(|x|-|x-y|)^{l}}{|x|^{k+l}},
\end{equation}
for some $C_{l}, C_{l}^{'}>0$.
On one hand, note that by Lemma \ref{lem3.10} and the assumption on $\sigma$, we have
\begin{equation*}\label{eqA.333}
	\begin{aligned}
		\left| \left\langle \frac{(|x|-|x-\cdot|)^{j+1}}{|x|^{j+l+1}|x-\cdot|^{k-l}}, \ f(\cdot) \right\rangle\right|
		\lesssim& \left\|f \right\|_{L_{\sigma}^2}\left\langle x\right\rangle^{-j-l-1}\left(\int_{\mathbb{R}^n}{\langle y\rangle^{-2(\sigma-j-1)}|x-y|^{-2(k-l)}\d y}\right)^{\frac12}\\
		\lesssim& \left\|f \right\|_{L_{\sigma}^2} \left\langle x\right\rangle^{p-j-1},
	\end{aligned}
\end{equation*}
On the other hand, the binomial expansion gives
$$\frac{(|x|-|x-y|)^{l}}{|x|^{k+l}}=\sum_{s=0}^{l}\frac{(-1)^{s}l!}{s!(l-s)!}\frac{|x|^{s}|x-y|^{l-s}}{|x|^{k+l}}.$$
Then by \textbf{Case 1} and \textbf{Case 2}, we have for each $s\in\{0,\cdots,l\}$ that
\begin{equation*}
	\begin{aligned}
		|x|^{s-k-l}\left| \left\langle |x-\cdot|^{l-s}, \ f(\cdot) \right\rangle\right|
		\lesssim& \left\|f \right\|_{L_{\sigma}^2} \left\langle x\right\rangle^{p-j-1}.
	\end{aligned}
\end{equation*}
In the view of \eqref{eqA.3}, we conclude \eqref{eqA.1.1} by combining the above two estimates.
\end{proof}

By the definition of $Q_j$ for $j\in J_{\k}$, we have for any function $f$ in $L^2$ that
$$\langle x^{\alpha}, vQ_{j} f \rangle=0, \quad \text{for all $|\alpha|< \delta(j)$},$$
where
\begin{equation}\label{eq-delta-j-8-17}
    \delta(j)=\max\{0, [j+\frac 12]\}.
\end{equation}
Then by Lemma \ref{lemmaA.1}, we have
\begin{align*}
  \left\|Q_jv|x-\cdot|^p\right\|_{L^2}^2 = \left\langle |x-\cdot|^p, vQ_jv|x-\cdot|^p\right\rangle
&\lesssim \left\|vQ_jv|x-\cdot|^p\right\|_{L_{\sigma}^2} \left\langle  x\right\rangle ^{p-\delta(j)}\\
& \lesssim \left\|Q_jv|x-\cdot|^p\right\|_{L^2}\left\langle  x\right\rangle ^{p-\delta(j)},
\end{align*}
where in the last inequality, we have used $|v| \lesssim \langle x\rangle^{-\beta/2}$ by \eqref{assum1}.
Thus, if $p\le 2j$, we have 
\begin{equation}\label{eqS-j-v-p}
   \|Q_jv|x-\cdot|^p\|_{L^2}\lesssim \left\langle  x\right\rangle ^{p-\delta(j)}.
\end{equation}
Further, we have
\begin{lemma}\label{prop4.1}
Let $-\frac{n-1}{2}\le p\le j$ with $j\in J_{\mathbf{k}}$. Then for $\lambda\in (0,1)$,  $k\in I^{\pm}$ (see \eqref{equ2.1.1'}), the following uniform estimates
\begin{equation}\label{equ4.1.1}
\sup_{s\in [0,1]}\left\| \partial_{\lambda}^l Q_j \left(v(\cdot)|x-\cdot|^p e^{ \i s \lambda_{k}|x-\cdot|\mp \i s\lambda|x|}\right)\right\|_{L^2}
\lesssim
\begin{cases}
 \lambda^{-l}\left\langle x \right\rangle^{p-\delta(j)},\quad &\lambda\left\langle x \right\rangle\le 1,\\
 \lambda^{\delta(j)-l}\left\langle x \right\rangle^{p},\quad &\lambda\left\langle x \right\rangle>1,	
\end{cases}
\end{equation}
 hold for all $0\le l\le[\frac{n}{2m}]+1$,  where  $\delta(j)$ is given by \eqref{eq-delta-j-8-17}.
\end{lemma}

\begin{proof}
 First note that $e^{\i \lambda_k|x|}=\exp(\i \lambda e^{\frac{\i k\pi}{m}})$ when $k\in I^{\pm}$. It follows that 
 $$|\partial_{\lambda}^{l}\left( e^{ \i s(\lambda_k\mp\lambda)|x|}\right)|\lesssim \lambda^{-l}.$$
 This, together with the triangle inequality $||x-y|-|x||\le |y|$ yields
\begin{equation}\label{equ4.1.2}
	\begin{array}{ll}
	\left|\partial_{\lambda}^{l}\left( e^{ \i s \lambda_{k}|x-y|\mp  \i s \lambda|x|}\right)\right|
	=\left|\partial_{\lambda}^{l}\left( e^{ \i s \lambda_{k}(|x-y|-|x|)} e^{ \i s(\lambda_k\mp\lambda)|x|}\right)\right| \lesssim_{l}\lambda^{-l}\langle y\rangle^{l},
    \end{array}
\end{equation}
which implies the estimate if $\delta(j)=0$

For $j\in J_{\k}$ such that $\delta(j)>0$, the Taylor formula for $e^{ \i  s \lambda_{k}(|x-y|-|x|)}$ gives
\begin{equation*}\label{equ4.1.333}
	e^{ \i s \lambda_{k}(|x-y|-|x|)}=\sum_{\gamma=0}^{\delta(j)-1} \frac{1}{\gamma!}\left( \i  \lambda_{k} s(|x-y|-|x|)\right)^{\gamma}+\tilde{r}_{j}\left(\lambda_{k} s(|x-y|-|x|)\right),
\end{equation*}
where
$$\tilde{r}_{j} \left(\lambda_{k} s(|x-y|-|x|)\right)=\frac{1}{(\delta(j)-1)!}\left( \i  \lambda_{k} s(|x-y|-|x|)\right)^{\delta(j)} \int_{0}^{1} e^{ \i  \lambda_{k} s u(|x-y|-|x|)} \left (1-u\right )^{\delta(j)-1}\d u.$$
Meanwhile, a direct computation shows for $l\in \N_0$ that
\begin{equation*}
	\left|\partial_{\lambda}^{l} \tilde{r}_{j}\left(\lambda_{k}s(|x-y|-|x|)\right)\right| \lesssim_{l} \lambda^{\delta(j)-l}\langle y\rangle^{\max\{\delta(j),l\}}.
\end{equation*}
Then it follows that
\begin{equation}\label{equ4.333}
  \begin{aligned}
     \left\|v(\cdot)|x-\cdot|^p\partial_{\lambda}^{l} \tilde{r}_{j}\left(\lambda_{k} s(|x-\cdot|-|x|)\right)\right\|_{L^2} & \lesssim \lambda^{\delta(j)-l}\||x-\cdot|^p
\left\langle \cdot \right\rangle^{-\beta/2+\max\{\delta(j),l\}}\|_{L^2}\\
      & \lesssim\lambda^{\delta(j)-l}\left\langle x \right\rangle^{p},
   \end{aligned}
\end{equation}
where  in the last inequality, we have used $\frac{\beta}{2}-\max\{p+l, p+j+1\}>\frac{n}{2}$ which follows from  $0\le l\le [\frac{n}{2m}]+1$, $p\le j$ and the assumption on $V$.
On the other hand, we use Lemma \ref{eqS-j-v-p}  to obtain
\begin{equation*}
\|Q_{j}v(\cdot)|x-\cdot|^p(|x-\cdot|-|x|)^{\gamma}\|_{L^2}\lesssim \left\langle x\right\rangle^{p+\gamma-\delta(j)}.
\end{equation*}
Hence, it follows that
\begin{equation*}\label{equ4.1.4}
\left\|  \partial_{\lambda}^l Q_j \left(v(\cdot)|x-\cdot|^p \sum_{\gamma=0}^{j} \frac{1}{\gamma !}\left(i \lambda_{k} s(|x-\cdot|-|x|)\right)^{\gamma} \right)\right\| _{L^2}
\lesssim
\begin{cases}
	\lambda^{-l}\left\langle x \right\rangle^{p-\delta(j)},\quad & \lambda\left\langle x \right\rangle\le 1,\\
	\lambda^{\delta(j)-l}\left\langle x \right\rangle^{p},\quad &  \lambda\left\langle x \right\rangle>1,	
\end{cases}
  \end{equation*}
which, together with \eqref{equ4.333}, yields \eqref{equ4.1.1}, and the proof is complete.
\end{proof}

Finally, we give the proof of Proposition \ref{lemma4.2}.

\begin{proof}[Proof of Proposition \ref{lemma4.2}]
We first prove \eqref{equ4.2.3}--\eqref{equ4.4} and divide it into two cases.

When $n=1$ and $\delta(j)=0$ ($j=0$), we set
\begin{equation*}
\begin{aligned}
k_{j,0,0}^{\pm}(\lambda,s,\cdot,x)&:=Q_0v\left(e^{\mp \i \lambda|x|} R_{0}^{ \pm}(\lambda^{2 m})(x-\cdot)\right),\\
k_{j,0,1}^{\pm}(\lambda,s,\cdot,x)&:=0.	
\end{aligned}
\end{equation*}
\eqref{equ4.2.3} immediately follows since $k_{j,0,0}^{\pm}$ is independent of $s$. Furthermore, we observe that $e^{i\lambda_k|x|}=e^{i\lambda|x|}$ when $k=0$, and $e^{\i\lambda_k|x|}=e^{-\i\lambda|x|}$ when $k=m$, then using the triangle inequality $\left||x-y|-| x|\right | \le |y|$ repeatedly, we obtain for $k\in I^{\pm}$ (as defined in \eqref{equ2.1.1'}) that
	\begin{equation*}\label{equ4.5}
			\left|\partial_{\lambda}^{l}\left(e^{\mp \i \lambda|x|} e^{\pm \i \lambda_k|x-y|}\right)\right| \lesssim_{l}\langle y\rangle^{l},\quad  l\in\mathbb{N}_0.
	\end{equation*}
This, together with \eqref{eq2.8},  implies when  $0<\lambda<1$ that
	\begin{equation*} \label{equ4.5.1}
		\left|\partial_{\lambda}^{l}\left(e^{\mp \i \lambda|x|} R_{0}^{ \pm}(\lambda^{2 m})(x-y)\right)\right| \lesssim \lambda^{1-2m-l}\langle y\rangle^{l}.
	\end{equation*}
Thus
\begin{equation*}
	\left\| v(y) \partial_{\lambda}^{l}\left(e^{\mp i\lambda |x|} R_{0}^{ \pm} (\lambda^{2 m})(x-y)\right)\right\|_{L_y^2} \lesssim \lambda^{1-2m-l},
\end{equation*}
holds for $l=0, 1$, which yields \eqref{equ4.3} and \eqref{equ4.4} for $n=1$ and $j=0$.
	
When $n=1$ with $j \in J_{\k}\setminus\{0\}$, or when $n \geq 2$ and $j \in J_{\k}$, we choose a smooth cutoff function $\phi \in C^\infty(\mathbb{R})$ satisfying
\begin{equation*}
\phi(t) = \begin{cases}
1, & |t| \leq \frac{1}{2}, \\
0, & |t| \geq 1,
\end{cases}
\end{equation*}
and set
\begin{equation}\label{eq4.theta}
  \theta=\begin{cases}
      \min\{\delta(j), 2m-n\}, \quad &\text{if $n<2m$},\\
      1,&\text{if $2m<n<4m$}.
  \end{cases}
\end{equation}
By \eqref{equ2.2.2}, we have 
\begin{equation*}	
R_0^{\pm}(\lambda^{2m})(x-y)=\begin{cases}
    \sum\limits_{0\le j\le \frac{\theta-1}{2}}a_j^\pm\lambda^{n-2m+2j}|x-y|^{2j}+r_{\theta}^{\pm}(\lambda)(x-y), \quad &\text{if $n<2m$},\\
    r_{2m-n}^{\pm}(\lambda)(x-y),&\text{if $2m<n<4m$},
\end{cases}
\end{equation*}
with 
\begin{align}\label{eq2.10-1}
r_{\theta}^{\pm}(\lambda)(x)= \sum_{j=\min\{0,\frac{n-3}{2}\}}^{\frac{n-3}{2}}\sum_{k\in I^{\pm}} C_{j,\theta}  \lambda_{k}^{n-2 m+\theta}
|x|^{\theta } \int_{0}^{1} e^{\i s \lambda_{k}|x|}(1-s)^{n-j+\theta-3} \d s,
\end{align}
as presented in \eqref{eq2.10}. Note the fact that 
$$Q_{j}\left(x^{\alpha} v(x)\right)=0 \quad \text{ for  $|\alpha|<\delta(j),~ \alpha \in \mathbb{N}_{0}^{n}$,} $$
one has 
$$Q_{j}vR_{0}^{\pm}(\lambda^{2 m})(x-\cdot)=Q_{j}vr_{\theta}^{\pm}(\lambda)(x-\cdot).$$
When $n\ge 3$, we also use Taylor's formula to obtain
\begin{equation}\label{eq:to locally intagral}
\begin{aligned}
R_0^{\pm}(\lambda^{2m})(x-y)=&\sum_{j=0}^{\frac{n-5}{2}}\sum_{k\in I^{\pm}} C_{j,\theta}  \lambda_{k}^{\frac{n+1}{2}-2 m}|x-y|^{-\frac{n-1}{2}} \int_{0}^{1} e^{\i s \lambda_{k}|x-y|}(1-s)^{\frac{n-1}{2}-j-2} \d s  \\
&+ \sum_{k\in I^{\pm}} e^{\i \lambda_{k}|x-y|} D_{\frac{n-3}{2}}\lambda_k^{\frac{n+1}{2}-2m} |x-y|^{-\frac{n-1}{2}},
\end{aligned}
\end{equation}
where we remark that each term in the RHS of \eqref{eq:to locally intagral} is locally integrable in $L^2_y$. When $n=1$, we define
 \begin{equation}\label{equ4.27.11}
\begin{cases}
k_{j,1,0}^{\pm}(\lambda,s,\cdot,x)=(1-\phi(\lambda\langle x\rangle))Q_jv\left(e^{\mp \i \lambda|x|} R_{0}^{\pm }(\lambda)(x-\cdot)\right), \\
k_{j,1,1}^{\pm}(\lambda,s,\cdot,x)=\phi(\lambda\langle x\rangle)Q_jv\left(e^{\mp \i \lambda s|x|}\sum\limits_{k\in I^{\pm}}C_{-1,\theta}\lambda_{k}^{1-2 m+\theta}|x-\cdot|^{\theta }e^{i s \lambda_{k}|x-\cdot|}(1-s)^{\theta-1}\right),
\end{cases}
\end{equation}
and when $n\geqslant3$, we define
 \begin{equation}\label{equ4.8.222}
\begin{cases}
k_{j,1,0}^{\pm}(\lambda,s,\cdot,x)=(1-\phi(\lambda\langle x\rangle))
		Q_jv\left(e^{\mp \i \lambda |x|} \sum\limits_{k\in I^{\pm}} D_{\frac{n-3}{2}} \lambda_k^{\frac{n+1}{2}-2m} \frac{e^{\i \lambda_{k}|x-\cdot|}}{|x-\cdot|^{\frac{n-1}{2}}}\right), \\
k_{j,1,1}^{\pm}(\lambda,s,\cdot,x)=\phi(\lambda\langle x\rangle)Q_jv\left( e^{\mp \i \lambda s|x|}\sum\limits_{k\in I^{\pm}}\sum\limits_{l=0}^{\frac{n-3}{2}} C_{l, \theta} \lambda_{k}^{n-2 m+\theta}|x-\cdot|^{\theta }  e^{\i s\lambda_{k}|x-\cdot|}(1-s)^{n-l+\theta-3}\right) \\
		\quad+(1-\phi(\lambda\langle x\rangle))Q_jv\left(e^{\mp \i s\lambda |x|} \sum\limits_{k\in I^{\pm}} \sum\limits_{l=0}^{\frac{n-5}{2}} C_{l,\theta} \lambda_k^{\frac{n+1}{2}-2 m}|x-\cdot|^{-\frac{n-1}{2}}e^{\i s \lambda_{k}|x-\cdot|}(1-s)^{\frac{n-5}{2}-l}\right),
\end{cases}
\end{equation}
where we have abused the notation on $k_{j,1,0}^{\pm}(\lambda,s,\cdot,x)$ for it does not depend on $s$.
Then we derive
\begin{equation*}
	\begin{aligned}
		Q_{j}vR_{0}^{\pm}(\lambda^{2 m})(x-\cdot)&=\phi(\lambda\langle x\rangle)Q_{j}v
		R_{0}^{\pm}(\lambda^{2 m})(x-\cdot)+(1-\phi(\lambda\langle x\rangle))Q_{j}vR_{0}^{\pm}(\lambda^{2 m})(x-\cdot)\\
		&=\phi(\lambda\langle x\rangle)Q_{j}vr_{\theta}^{+}(\lambda)(x-\cdot)+(1-\phi(\lambda\langle x\rangle))Q_{j}vR_{0}^{\pm}(\lambda^{2 m})(x-\cdot)\\
		&=\int_{0}^{1}e^{ \pm \i \lambda|x|}k_{j, 1,0}^{\pm}(\lambda,s,\cdot,x) \d s+\int_{0}^{1} e^{\pm \i \lambda s|x|} k_{j,1,1}^{\pm}(\lambda, s,\cdot, x) \d s,
	\end{aligned}
\end{equation*}
where in the third equality above, we use  \eqref{equ4.27.11} (when $n=1$),  \eqref{eq:to locally intagral} and \eqref{equ4.8.222} (when $n\ge 3$).
This proves \eqref{equ4.2.3}.

We next prove the estimates \eqref{equ4.3}, \eqref{equ4.4} and \eqref{equ4.4-1-1}. For each $l\in\{0,\cdots,[\frac{n}{2m}]+1\}$ and $k\in I^{\pm}$, it follows from  Lemma  \ref{prop4.1} that
\begin{equation}\label{equ4.17.3}
	W_1:=\left\|\partial_{\lambda}^{l}\left(\lambda_k^{n-2m+\theta}Q_{j} \left(|x-\cdot|^{\theta}v(\cdot) e^{is \lambda_{k}|x-\cdot|\mp is\lambda|x|}\right)\right) \right\|_{L^{2}}\lesssim\langle x\rangle^{\theta-\delta(j)}\lambda^{n-2m+\theta-l},
\end{equation}
and consequently when $\lambda\langle x\rangle\le 1$, it follows uniformly for $0\le s \le1$ that
\begin{equation}\label{equ4.17.3333}
W_1\lesssim\begin{cases}
	\lambda^{\frac{n+1}{2}-2m+\delta(j)-l}\langle x\rangle^{-\frac{n-1}{2}} ,	 &\text{if}\,\, j< 2m-\frac{n}{2},\\
	\lambda^{-l}\langle x\rangle^{-\frac{n+1}{2}} ,	 &\text{if}\,\,  j=2m-\frac{n}{2},
\end{cases}
\end{equation}
where we have used the fact that 
$$\mbox{$\theta-\delta(j)=\min\{\delta(j), 2m-n\}-\delta(j)\geq-\frac{n-1}{2}$}, \quad \text{if} ~~\mbox{$j<2m-\frac n2$},$$ 
in the last inequality. We apply  Lemma  \ref{prop4.1} with $p=-\frac{n-1}{2}$  and obtain
\begin{equation}\label{equ4.17.4444}
	W_2:=\left\|\partial_{\lambda}^l \left( \lambda_k^{\frac{n+1}{2}-2 m} Q_{j}v(\cdot)|x-\cdot|^{-\frac{n-1}{2}} e^{is \lambda_{k}|x-\cdot|\mp i s\lambda|x|}\right)\right\|_{L^{2}}\lesssim \lambda^{\frac{n+1}{2}-2m+\delta(j)-l}\langle x\rangle^{-\frac{n-1}{2}},
\end{equation}
and consequently when $\lambda\langle x\rangle>\frac12$, it follows that
\begin{equation}\label{equ4.17.4}
	W_2\lesssim\lambda^{n-2m+\delta(j)-l},
\end{equation}
which holds uniformly for $0\le s \le1$. Note that the support of $\phi(\lambda\langle x\rangle))$ is contained in $\{\lambda\langle x\rangle\le 1\}$ and the support of $1-\phi(\lambda\langle x\rangle)$ is contained in $\{\lambda\langle x\rangle>\frac12\}$. Since
$$n-2m+\theta=\min\{n-2m+\delta(j), 0\},$$
we conclude \eqref{equ4.3}, \eqref{equ4.4} and \eqref{equ4.4-1-1}.

We are left to prove \eqref{equ4.4.3}--\eqref{equ4.4.5}. For $j\in \{m-\frac{n}{2}, 2m-\frac{n}{2}\}$,
we define $k_{j,2,0}^{\pm}(\lambda,s,\cdot,x)=k_{j,1,0}^{\pm}(\lambda,s,\cdot,x)$ and $k_{j,2,1}^{\pm}(\lambda,s,\cdot,x)=k_{j,1,1}^{\pm}(\lambda,s,\cdot,x)$.
Instead of \eqref{eq4.theta}, we choose
\begin{equation*}\label{eq4.theta.11}
  \theta=\delta(j)=\max\{\mbox{$[j+\frac12]$}, 0\},
\end{equation*}
in \eqref{equ4.27.11} and \eqref{equ4.8.222}.
By \eqref{equ2.2.2}, we have
\begin{equation*}
\begin{aligned}
R_{0}^{+}(\lambda^{2 m})(x-y)-R_{0}^{-}(\lambda^{2 m})(x-y)
=&\sum_{l=0}^{\left[\frac{\theta-1}{2}\right]} (a_{l}^{+}-a_l^{-}) \lambda^{n-2 m+2 l}|x-y|^{2 l}\\
&+r_{\theta}^{+}(\lambda,|x-y|)-r_{\theta}^{-}(\lambda,|x-y|).
\end{aligned}
\end{equation*}
Therefore \eqref{equ4.4.3} follows. The estimates \eqref{equ4.4.4}-\eqref{equ4.34.3} follow from
\eqref{equ4.17.3}-\eqref{equ4.17.4} and a direct computation. Now the proof of Proposition \ref{lemma4.2} is finished.
\end{proof}

\end{appendix}

\section*{Acknowledgements}
We would like to express our sincere gratitude to all those who have supported and contributed to this work through their valuable insights, discussions, and assistance.

T. Huang was supported by National Key R\&D Program of China under the grant 2023YFA1010300 and the National Natural Science Foundation of China under the grants 12101621 and 12371244.
S. Huang was supported by the National Natural Science Foundation of China under the grants 12171178 and 12171442.
Q. Zheng was supported by the National Natural Science Foundation of China under the grant 12171178.

%
%


\end{document}